\documentclass[12pt]{amsart}
\usepackage{amsmath,eucal,verbatim,amsopn,amsthm,amsfonts,amssymb,mathdots,mathrsfs,esint,mathtools,enumitem,mnsymbol,bbm}
\usepackage{amsmath,amssymb,latexsym}
\usepackage{amscd}
\usepackage[margin=0.9in]{geometry}
\usepackage[hyperfootnotes=false,colorlinks=true,allcolors=blue]{hyperref}

\theoremstyle{plain}
\newtheorem{theorem}{Theorem}
\newtheorem{proposition}[theorem]{Proposition}
\newtheorem{corollary}[theorem]{Corollary}
\newtheorem{lemma}[theorem]{Lemma}
\newtheorem{claim}[theorem]{Claim}
\newtheorem{definition}[theorem]{Definition}

\newtheorem{remark}[theorem]{Remark}
\newtheorem{example}[theorem]{Example}

\newtheorem{theo}{Theorem}[section]

\numberwithin{equation}{section}

\def\R{\mathbb{R}}
\def\A{\mathbb{A}}

\def\C{\mathbb{C}}
\def\Mat{\text{Mat}}

\newcommand{\lam}{\lambda}

\DeclareMathOperator{\tr}{tr}

\DeclareMathOperator{\Hom}{Hom}
\DeclareMathOperator{\Bil}{Bil}
\DeclareMathOperator{\ind}{ind}

\DeclareMathOperator{\Sym}{Sym}

\DeclareMathOperator{\diag}{diag}
\DeclareMathOperator{\Ind}{Ind}
\DeclareMathOperator{\Orth}{O}
\DeclareMathOperator{\GL}{GL}
\DeclareMathOperator{\SL}{SL}
\DeclareMathOperator{\SO}{SO}
\DeclareMathOperator{\Sp}{Sp}
\DeclareMathOperator{\Real}{Re}

\begin{document}

\title[Doubling Constructions: the Linear Case]{Doubling Constructions and Tensor Product $L$-Functions: the linear case}
\author{Yuanqing Cai}
\author{Solomon Friedberg}
\author{David Ginzburg}
\author{Eyal Kaplan}
\address{Friedberg:  Department of Mathematics, Boston College, Chestnut Hill, MA 02467-3806, USA}
\email{solomon.friedberg@bc.edu}

\address{Cai:  Department of Mathematics, Boston College, Chestnut Hill, MA 02467-3806, USA}
\address{Current address: Department of Mathematics, Weizmann Institute of Science, Rehovot 7610001, Israel}
\email{yuanqing.cai@weizmann.ac.il}

\address{Ginzburg: School of Mathematical Sciences, Tel Aviv University, Ramat Aviv, Tel Aviv 6997801,
Israel}
\email{ginzburg@post.tau.ac.il}
\address{Kaplan: Department of Mathematics, Bar Ilan University, Ramat Gan 5290002, Israel}
\email{kaplaney@gmail.com}

\thanks{This research was supported by the ERC, StG grant number 637912 (Cai), by the BSF, grant number 2012019 (Friedberg and Ginzburg), by the NSF, grant numbers DMS-1500977 and DMS-1801497 (Friedberg), and by the Israel Science Foundation, grant number 421/17 (Kaplan).}
\subjclass[2010]{Primary 11F70; Secondary 11F55, 11F66, 22E50, 22E55}
\keywords{Doubling method, Eisenstein series, Whittaker--Speh--Shalika representation,
Rankin--Selberg $L$-function, non-generic automorphic representation, unipotent orbit, metaplectic cover}
\begin{abstract}
We present an integral representation for the tensor product $L$-function of
a pair of automorphic cuspidal representations, one of a classical group, the other of a general linear group.
Our construction is uniform over all classical groups, and is applicable to {\it all} cuspidal representations; it does not require genericity. The main new ideas of the construction are the use of generalized Speh representations as inducing data for the Eisenstein series and the introduction of a new (global and local) model, which generalizes the Whittaker model. Here
we consider linear groups, but our construction also extends to arbitrary degree metaplectic coverings; this will be the topic of an upcoming work.
\end{abstract}

\maketitle

\section{Introduction}\label{intro}
One of the pillars of the Langlands program is the study of global
automorphic $L$-functions as mediating agents in the framework of functoriality.
The analytic properties of $L$-functions for representations of classical groups twisted by representations of general linear groups
played a central role in the proof of functoriality for classical groups by Cogdell \textit{et al.}\ \cite{CKPS}.
That proof relied on the Converse Theorem of Cogdell and Piatetski-Shapiro \cite{CPS3,CPS1999}: strong analytic properties of the twisted $L$-functions imply automorphicity.
Cogdell \textit{et al.}\ only considered globally generic representations -- those affording a Whittaker--Fourier coefficient -- because constructions of the $L$-functions,
either using the Langlands--Shahidi method or the Rankin--Selberg method, were limited to such representations.

On the other hand, Piatetski-Shapiro and Rallis \cite{PSR} introduced a different type of global integral which
represents the standard $L$-function for any classical group. Their construction is advantageous in two important aspects.
First, it presents a unified approach to integral representations of these $L$-functions, comparable to the uniformity of the Langlands--Shahidi method.
Second, it is applicable to {\sl any} cuspidal automorphic representation on the classical group. Previously known integrals unfolded to
a special model, afforded by some but not all cuspidal automorphic
representations, most notably the Whittaker model. In contrast, the construction of \cite{PSR}, now known as the doubling method,
unfolded to an integral involving a global matrix coefficient on the classical group, which is always nontrivial for some choice of data, and for decomposable data can be expressed as the (infinite) product of local matrix coefficients. On the downside, these constructions were limited to the standard representation, or its twists by characters.  Thus they did not provide enough information to be used in concert with the Converse Theorem to establish functoriality for non-generic automorphic representations.

Here we describe a new construction, which extends the doubling method, to provide integral representations for arbitrary automorphic cuspidal representations of classical groups twisted by
automorphic cuspidal representations of arbitrary rank general linear groups. Our integrals inherit the benefits of the doubling method in that the construction is uniform across all classical groups and applies to all cuspidal automorphic representations (as opposed to only globally generic ones), but in sharp contrast with the doubling method, we are not limited to rank-one twists.

This paper removes a fundamental obstruction to extending the functoriality results of
Cogdell \textit{et al.} \cite{CKPS} to any automorphic cuspidal representation.
Although such liftings also follow from the work of Arthur, an independent proof is of high interest.
In a forthcoming work we further develop the global and local theory, analyze the local integrals over both non-archimedean and archimedean fields and define $\gamma$-, $L$- and $\epsilon$-factors, along the lines of the work of Lapid and Rallis \cite{LR} on the original doubling method (see also \cite{Gan,Yamana}). We use these results to construct a functorial lift of $\pi$ to $\GL_N(\A)$ using the Converse Theorem (see \cite{CFK}).

Let $F$ be a number field with a ring of adeles $\A$, and $G$ be a split classical group.
Let $\pi$ and $\tau$ be irreducible cuspidal automorphic representations of $G({\A})$ and
$\GL_k({\A})$, respectively. We construct an Eisenstein series $E(h;f,s)$ on $H({\A})$,
where $H$ is an auxiliary classical group defined depending on $G$ and $k$. The inducing data of the Eisenstein series is a generalized Speh representation $\mathcal{E}_\tau$ attached to $\tau$. We choose a unipotent subgroup $U$ of $H$ and an automorphic character $\psi_U$ of $U$, such that $G\times G$ is embedded in the normalizer of $U$ and stabilizer of $\psi_U$. We consider the integral
\begin{align}\label{double-double}
Z(s,\varphi_1,\varphi_2,f)=\int\limits_{G(F)\times G(F)\backslash G({\A})\times G({\A})}
\varphi_1(g_1)\,\overline{\varphi_2({}^{\iota}g_2)}\,
E^{U,\psi_U}((g_1,g_2);f,s)\,dg_1\,dg_2.
\end{align}
Here $\varphi_1$ and $\varphi_2$ are cusp forms in the space of $\pi$, ${}^{\iota}$ is an involution of $G$ and $E^{U,\psi_U}$ denotes the Fourier coefficient of the series with respect to $U$ and $\psi_U$. This is a ``doubling construction" in the sense that
the integral is over two copies of $G$ and for $k=1$, reproduces the doubling integral of \cite{PSR}
(the original doubling method was motivated by doubling in the context of quadratic forms).

The main result of this paper is the following.
\begin{theo}\label{theo:main}
The global integral \eqref{double-double} represents the global partial $L$-function $L^S(s,\pi\times\tau)$.
\end{theo}
\noindent
The function $L^S(s,\pi\times\tau)$ here is the product of local $L$-functions over all finite places of $F$
for which the local data are unramified. Theorem~\ref{theo:main} follows by combining Theorem~\ref{theorem:main theorem classical groups},
identity~\eqref{eq:almost Euler}, Theorem~\ref{theorem:local integrals} and Theorem~\ref{theorem:unramified computation for Sp(2n),SO(2n)} below. Here we treat two cases in detail: $G=\Sp_{2n}$ and $\SO_{2n}$.

In a subsequent paper we elaborate on the details of this construction also for $\SO_{2n+1}$ and split connected general spin groups of arbitrary rank (see \cite{CFK}). In this paper we do not treat these groups as they require additional work of a technical nature.

The novel ingredients of \eqref{double-double} compared to the doubling method of \cite{PSR}, are the usage of the specialized inducing data, namely the representation $\mathcal{E}_\tau$, and the replacement of the Eisenstein series there with its Fourier coefficient. Critically, it turns out that the representation $\mathcal{E}_\tau$ is supported on a sufficiently small unipotent orbit.
The unfolding process leads us to introduce a new (global and local) model, which we call a Whittaker--Speh--Shalika model, since it generalizes the Whittaker and Shalika models for generalized Speh representations
(see Definition~\ref{def1} below). The nonvanishing of the appropriate Fourier coefficient of $\mathcal{E}_{\tau}$, as well as the vanishing properties of $\mathcal{E}_{\tau}$ that we use, were proved by Ginzburg \cite{G4}; see also Jiang and Liu \cite{JL2013} for a detailed study of these representations in a global context.
Then to deduce that the integral is Eulerian (Theorem~\ref{theorem:main theorem classical groups}) we establish multiplicity one results, at least over the unramified places.

One immediate consequence
of Theorem~\ref{theo:main} is that $L^S(s,\pi\times\tau)$ admits meromorphic continuation to the plane, see
Theorem~\ref{theorem:mero}. This is of course a well-known result of Langlands (e.g., \cite{La2,La5}), who established it by analyzing the constant term of the Eisenstein series. However, the constant term approach is not sufficient to handle local factors at the remaining places, nor to get the full analytic behavior necessary to apply the Converse Theorem. Local factors for irreducible generic representations are usually defined via Shahidi's celebrated method of local coefficients (e.g., \cite{Sh3}). This method is not applicable in general to non-generic representations, hence the aforementioned functoriality results \cite{CKPS} were limited to generic ones. By contrast, the local version of our integrals may be used to define and study local $L$- and $\epsilon$-factors at all places. In fact,
the definition of local factors using integral representations may well be the only available analytic method
for the general case. For further reference see, e.g., \cite{PR2,PS,Ik,HKS,Ik2,LR,me4,Yamana}.
Note that historically, for general linear groups this was the original definition, see \cite{GJ,JPSS}.

Our ideas and construction apply also to non-linear coverings. Starting with genuine representations $\pi$ and $\tau$ of certain covering groups of $G(\A)$ and $\GL_k(\A)$, we construct a similar global integral, the main difference being the rank of $H$, which now also depends on the parameters of the covering. We will describe this construction in a subsequent work. See also Gao \cite{Gao2018} for the extension of the constant term approach to covering groups.

For linear groups, the descent method was used to construct an explicit realization of an inverse to the functorial lift from globally generic representations of classical groups to $\GL_N$; see Ginzburg \textit{et al.} \cite{GRS5,GRS7,RGS} and also Soudry \cite{Soudry6}. We expect to use the integrals developed here to extend the descent method to functorial lifts of arbitrary automorphic cuspidal representations, and also to obtain new descent constructions for covering groups.

The doubling method has had numerous important applications. We list several of these. Its strong relation to the theta correspondence, via the Siegel-Weil formula, has been studied in \cite{KudlaRallis1994,HKS,WGS,Yamana}; B\"ocherer and Schmidt \cite{BochererSchmidt2000} used the doubling method to construct standard $p$-adic $L$-functions for Siegel modular forms; recently, Eischen \textit{et. al.} \cite{EischenHarrisLiSkinner} used this method to construct $p$-adic $L$-functions for unitary groups, completing the results of Harris \textit{et. al.} \cite{HarrisLiSkinner2006} (see also \cite{HarrisLiSkinner2005}), which are part of a long-term project by these authors; and Garrett \cite{Garrett} developed the doubling method in a classical framework. Among other works on the doubling method we mention \cite{Takano1997,JahwanKim2000}.

The original doubling method was developed for classical groups of symplectic, orthogonal or unitary type \cite{PSR,LR},
and later extended to several more cases including the double cover of the symplectic group by Gan \cite{Gan}, and
unitary groups of hermitian or skew-hermitian forms over division algebras in the work of Yamana \cite{Yamana}. We expect similar extensions to be applicable here. Interestingly, the odd orthogonal case was excluded from \cite{PSR} and was first treated only in \cite{LR} (for technical reasons). Also note that we develop the theory for connected groups, i.e., $\SO_{2n}$ instead of $\Orth_{2n}$. This is compatible with the theories of Langlands and Shahidi, which were formulated for connected groups, and with several other works on Rankin--Selberg integrals.

Earlier works on integral representations include \cite{JS1,GPS,BF,G,JS4,BG,Soudry,Tk}. Recent works \cite{GPSR,GJRS,JZ} developed $L$-functions for tensor products of automorphic cuspidal representations of classical groups and general linear groups. In these works the Whittaker model was replaced by a pairing with a suitable auxiliary cuspidal representation. In particular, Jiang and Zhang \cite{JZ} extended the construction from cuspidal representations of $\GL_k$ to isobaric representations. The isobaric sum was previously considered in \cite{RGS}, albeit in a less general context. We also mention two recent works by Soudry \cite{Soudry7,Soudry8}, who reduced local computations with non-generic data to the known generic case, in the context of the integrals of \cite{GPSR,GJRS,JZ}. By contrast, our approach to the (local) study of non-generic representations is to use the uniqueness of the pairing of an irreducible representation with its contragredient, which is true without any additional assumption (even over covering groups).

The rest of this work is organized as follows. In \S~\ref{global} we present the global construction, starting with the integral
(\S~\ref{global classical}), then discuss the generalized Speh representation and its properties (\S~\ref{speh}) and carry out the unfolding process (\S~\ref{global symplectic}). The computation of the local integrals with unramified data is described in \S~\ref{Computation of the local factors with unramified data}. The local integrals are presented in \S~\ref{Local factors for classical} and their computation is reduced to a similar computation on general linear groups, which is further reduced to a rank-$1$ case (\S~\ref{Local factors for GLn}). The latter integral is computed in \S~\ref{final reduction n = 1 linear groups} by, surprisingly enough, reducing it to the familiar Rankin--Selberg integrals of $\GL_1\times\GL_k$ and $\GL_1\times\GL_{2k}$ from \cite{JS1,JPSS}.

\subsection*{Acknowledgments}
Part of this work was done while the fourth named author was a Zassenhaus Assistant Professor
at The Ohio State University, under the supervision of Jim Cogdell. Eyal\footnote{Eyal dedicates his part of the work to his beloved Sophie Kaplan who passed away unexpectedly a few weeks before the submission of the first version of this work.} wishes to express his gratitude to Jim for his kind encouragement and support.

\section{The global construction}\label{global}
\subsection{The global integral}\label{global classical}
We introduce the general global integral. Let $n$ and $k$ denote two positive integers,
$F$ be a number field with a ring of adeles $\A$, and
$G$ be a split connected classical group of rank $n$.
Let $\pi_1$ and $\pi_2$ denote two irreducible cuspidal automorphic representations of
$G(\A)$, and $\tau$ denote an irreducible cuspidal automorphic representation of $\GL_{k}(\A)$.

Let $c=c(n)$ be the rank of the natural general linear group containing $G$, i.e., $c=2n$ for $G=\Sp_{2n}$ and $\SO_{2n}$, and $c=2n+1$ for $G=\SO_{2n+1}$. Depending on $G$, we introduce another classical group $H$ of rank $kc$, on which we shall construct an Eisenstein series. For example if $G=\Sp_{2n}$, $H=\Sp_{4kn}$.
Fix a Borel subgroup $B_H$ in $H$ and let $P=M_P\ltimes U_P$ denote a maximal
standard parabolic subgroup of $H$ with a Levi part $M_P\cong\GL_{kc}$, i.e., a so-called Siegel parabolic subgroup. The precise definitions of $H$ and $P$ will be given near the end of this section.

The key building block in our construction is a residue representation $\mathcal{E}_{\tau}$ of $\GL_{kc}(\A)$,
which we call a Whittaker--Speh--Shalika representation of type $(k,c)$. In this work it is the generalized
Speh representation corresponding to $c$ copies of $\tau$. The definition and construction are detailed in \S~\ref{speh} below.
Its fundamental properties are that it is supported on a sufficiently small unipotent orbit (in the sense of \cite{G2}), and on this orbit it supports a $(k,c)$ functional $\Lambda$. This functional is ``almost decomposable" (see Claim~\ref{claim:decomp Lambda}). These properties are crucial for the unfolding argument and proof that the global integral is ``almost" an Euler product
(see \eqref{eq:almost Euler}). Explicitly, if $\varphi$ belongs to the space of $\mathcal{E}_{\tau}$,
\begin{align*}
\Lambda(\varphi)=\int\limits_{V_{(c^k)}(F)\backslash V_{(c^k)}({\A})}
\varphi(v)\,\psi^{-1}(\mathrm{tr}(\sum_{i=1}^{k-1}v_{i,i+1}))\,dv,
\end{align*}
where $V_{(c^k)}$ is a subgroup of upper triangular matrices, which is the unipotent radical of the parabolic subgroup of $\GL_{kc}$ corresponding to the partition $(c^k)$, $v_{1,2},\ldots,v_{k-1,k}$ are the $c\times c$ blocks above the main diagonal of $v$ (see \S~\ref{speh}) and $\mathrm{tr}$ is the trace map. For example when $c=1$, this is the usual Whittaker--Fourier coefficient and $\mathcal{E}_{\tau}$ is simply $\tau$, which is known to be globally generic, i.e., supports a $(k,1)$ functional. Note that for $c>1$ we do not claim or expect arbitrary cuspidal automorphic representations of $\GL_{kc}(\A)$ to be of type $(k,c)$, this is a special property enjoyed by the representations $\mathcal{E}_{\tau}$; see also Remark~\ref{remark:uniqueness comment} for a local discussion.

Let $K_H$ be a maximal compact subgroup of $H$ which is in a ``good position" with respect to the maximal torus of $B_H$ (see e.g., \cite[\S~I.1.4]{MW2}). Form the Eisenstein series $E(h;f,s)$ on $H({\A})$, attached to the induced representation
\begin{align*}
\Ind_{P({\A})}^{H({\A})}({\mathcal E}_{\tau}\delta_P^s).
\end{align*}
Here $\delta_P$ is the modulus character of $P$ (throughout, induction is normalized). By definition, for $\Real(s)\gg0$,
\begin{align*}
E(h;f,s)=\sum\limits_{\gamma\in P(F)\backslash H(F)}f(\gamma h,s),\qquad h\in H({\A}),
\end{align*}
where $f(h,s)$ is a standard section, i.e., a section whose restriction to $K_H$ is independent of $s$. We will use a certain Fourier coefficient of this series.

To describe this coefficient, let $Q=M\ltimes U$ be a standard parabolic subgroup of $H$, whose Levi part $M$ is isomorphic to $k-1$ copies of $\GL_c$ multiplied by a split classical group of rank $c$. The subgroup $Q$ is uniquely defined given $k$ and the type of $H$. E.g., for $G=\Sp_{2n}$, $M=\GL_{2n}\times\ldots\times\GL_{2n}\times\Sp_{4n}$. Recall that unipotent orbits for classical groups are indexed by (certain) partitions (see e.g., \cite{Spaltenstein1982,Cr,CM}). Consider the unipotent orbit
\begin{align*}
((2k-1)^c1^c)
\end{align*}
associated with the group $H$. It follows from Collingwood
and McGovern \cite{CM} that this is a well-defined orbit for every group $H$ (for $\Sp_{2n}$, odd numbers occur with even multiplicity, in the orthogonal cases this is clear since there are no even parts),
and that the stabilizer of this orbit over an algebraically closed field contains the group $G\times G$.
From \cite{G2} we deduce that a Fourier coefficient associated with
this orbit can be constructed along $U$, and an automorphic character $\psi_U$ of $U(\A)$ can be defined
such that its stabilizer inside $M(\A)$ contains $G(\A)\times G(\A)$. For an example of $U$, $\psi_U$ and the embedding $(g_1,g_2):G(\A)\times G(\A)\rightarrow M(\A)<H(\A)$ in the cases of $\Sp_{2n}$ and $\SO_{2n}$
see \S~\ref{global symplectic}. For brevity, we denote the identity element of $G$ by $1$ in the embedding, e.g., write $(1,g)$.

The global integral we consider is
\begin{align*}
Z(s,\varphi_1,\varphi_2,f)=\int\limits_{G(F)\times G(F)\backslash G({\A})\times G({\A})}\,
\int\limits_{U(F)\backslash U({\A})}\varphi_1(g_1)\,\overline{\varphi_2({}^{\iota}g_2)}\,
E(u(g_1,g_2);f,s)\,\psi_U(u)\,du\,dg_1\,dg_2.
\end{align*}
Here $\varphi_i$ is a cusp form in the space of $\pi_i$, $(g_1,g_2)$ is the embedding and $\iota$ is a certain involution of $G$ (see below). The integral converges absolutely for $\Real(s)\gg0$ and admits meromorphic continuation to the whole complex plane; this follows from the rapid decay of cusp forms, moderate growth of the Eisenstein series and the meromorphic continuation of the Eisenstein series.

Let $L=(G\times G)U$. It is a subgroup of $Q$. The action of $L(F)$ on the right on the homogeneous space $P(F)\backslash H(F)$ has a unique open orbit. Let $\delta\in H(F)$ be a representative for this orbit. The involution $\iota$ is chosen such that $\delta(g,{}^{\iota}g)\delta^{-1}\in M_P(\A)$ for all $g\in G(\A)$. Denote $U_0=U\cap U_{P}$. Also let
\begin{align*}
\langle\varphi_1,\varphi_2\rangle=\int\limits_{G(F)\backslash G({\A})}\varphi_1(g)\overline{\varphi_2(g)}\,dg
\end{align*}
be the standard inner product on $G(\A)$. Refer to \S~\ref{global symplectic} for the concrete choices of $\iota$ and $\delta$
(for $\Sp_{2n}$ see \eqref{eq:iota} and \eqref{open2}).

In the following theorem we state the basic properties of the integral.
\begin{theorem}\label{theorem:main theorem classical groups}
The integral $Z(s,\varphi_1,\varphi_2,f)$ is absolutely convergent for $\Real(s)\gg0$ and admits meromorphic continuation to the plane. It is not identically zero only if
$\pi_1=\pi_2=\pi$. In this case, for $\Real(s)\gg0$ it is equal to
\begin{align}\label{global2}
\int\limits_{G({\A})}\int\limits_{U_0({\A})}
\langle\varphi_1,\pi(g)\varphi_2\rangle f_{W({\mathcal E}_{\tau})}(\delta u_0(1,{}^{\iota}g),s)
\,\psi_U(u_0)\,du_0\,dg.
\end{align}
Here $f_{W({\mathcal E}_{\tau})}$ is the composition of the section and the unique functional $\Lambda$ attached to
${\mathcal E}_{\tau}$: for any $s\in\C$ and $h\in H(\A)$,
\begin{align*}
f_{W(\mathcal{E}_{\tau})}(h,s)=\int\limits_{V_{(c^k)}(F)\backslash V_{(c^k)}({\A})}
f(vh,s)\,\psi^{-1}(\mathrm{tr}(\sum_{i=1}^{k-1}v_{i,i+1}))\,dv.
\end{align*}
\end{theorem}
We prove the main identity \eqref{global2} for $\Sp_{2n}$ in \S~\ref{global symplectic}. The proof for $\SO_{2n}$ is similar, and the changes
that are needed for this group are described in Remark~\ref{remark:global parameters for SO} below. Also note that, while we do not provide details for other groups in this work, Theorem~\ref{theorem:main theorem classical groups} is also valid for $\SO_{2n+1}$, and (with minor changes) for split connected general spin groups of even or odd rank.

As explained in the introduction, here we describe in detail the cases of the split groups $\Sp_{2n}$ and $\SO_{2n}$.
For concreteness, $\Sp_{2n}$ is defined as the subgroup of matrices $g\in \SL_{2n}$ such that
${}^tg\left(\begin{smallmatrix}&J_{n}\\-J_{n}\end{smallmatrix}\right)g=\left(\begin{smallmatrix}&J_{n}\\-J_{n}\end{smallmatrix}\right)$, where ${}^tg$ is the transpose of $g$ and
$J_{n}$ is the $n\times n$ permutation matrix having $1$ on its anti-diagonal. Also define $\SO_{2n}=\{g\in\SL_{2n}:{}^tgJ_{2n}g=J_{2n}\}$.

Put $c=2n$. For $G=\Sp_{2n}$ let $H=\Sp_{2kc}$, if $G=\SO_{2n}$ take $H=\SO_{2kc}$.
Regarding $H$ as a subgroup of $\GL_{2kc}$, choose the Borel subgroup $B_H=B_{\GL_{2kc}}\cap H$, where
$B_{\GL_{2kc}}<\GL_{2kc}$ is the subgroup of upper triangular invertible matrices,
and similarly $B_G$ for $G<\GL_c$. This already fixes $P$ unless $H=\SO_{2kc}$, then we choose $P$ with $M_P=\{\diag(g,J_{kc}{}^tg^{-1}J_{kc}):g\in\GL_{kc}\}$.

The doubling construction can also be described for general linear groups extending the case $k=1$ of \cite[\S~4.2]{PSR}. One must divide by the center and handle convergence (as in \cite[\S~4.2]{PSR}). We omit the details, since we will only be using these integrals locally, for the purpose of computing the integrals for classical groups with unramified data.

\subsection{Whittaker--Speh--Shalika Representations}\label{speh}
We present the family of representations $\mathcal{E}_{\tau}$ used in \S~\ref{global classical} to define the Eisenstein series.

In the group $\GL_l$, write $B_{\GL_l}=T_{\GL_l}\ltimes N_{\GL_l}$ where $T_{\GL_l}$ is the diagonal torus. For a composition $(l_1,\ldots,l_r)$ of $l$, $P_{(l_1,\ldots,l_r)}=M_{(l_1,\ldots,l_r)}\ltimes V_{(l_1,\ldots,l_r)}$ denotes the standard parabolic subgroup of $\GL_l$ whose Levi part $M_{(l_1,\ldots,l_r)}$ is isomorphic to $\GL_{l_1}\times\ldots\times\GL_{l_r}$.
(We recall that a composition of a positive integer $l$ is an ordered sequence of positive integers summing to $l$.)
Also let $C_l$ be the center of $\GL_l$, denote the additive group of $l\times l'$ matrices by $\Mat_{l\times l'}$, and set $\Mat_{l}=\Mat_{l\times l}$.

Recall that the unipotent orbits of $\GL_l$ are in bijection with the partitions of $l$, and for such a partition there is a corresponding unipotent subgroup and a set of generic characters (see \cite[\S~2]{G2} for these definitions).

Let $k$ and $c$ be positive integers. The unipotent subgroup corresponding to the orbit $(k^c)$ is $V_{(c^k)}$. Fix a nontrivial character $\psi$ of $F\backslash\A$. Denote a matrix $v\in V_{(c^k)}$ by
$v=(v_{i,j})_{1\leq i,j\leq k}$, where $v_{i,j} \in\Mat_{c}$.
For an automorphic function $\varphi$ on $\GL_{kc}(F)\backslash\GL_{kc}({\A})$, consider the integral
\begin{align}\label{whspeh1}
\Lambda(\varphi)=\int\limits_{V_{(c^k)}(F)\backslash V_{(c^k)}({\A})}
\varphi(v)\,\psi^{-1}(v)\,dv,
\end{align}
where $\psi$ is the character of $V_{(c^k)}$ defined by
\begin{align}\label{eq:wss character}
\psi(v)=\psi(\mathrm{tr}(\sum_{i=1}^{k-1}v_{i,i+1})).
\end{align}
This is a Fourier coefficient corresponding to the orbit $(k^c)$, and we call it a Whittaker--Speh--Shalika coefficient.

\begin{example}
In particular when $c=1$,
\begin{align*}
\Lambda(\varphi)=\int\limits_{N_{\GL_{k}}(F)\backslash N_{\GL_{k}}({\A})}
\varphi(v)\,\psi^{-1}(\sum_{i=1}^{k-1}v_{i,i+1})\,dv
\end{align*}
is the well-known Whittaker--Fourier coefficient.
An automorphic representation $\rho$ of $\GL_{k}(\A)$ is globally generic when this functional is not identically zero on the elements $\varphi$ in the space of $\rho$. As we will see below, the representation $\mathcal{E}_{\tau}$ is defined for $c=1$ to be $\tau$ itself. Since $\tau$ is cuspidal, by \cite{JL,Sh,PiatetskiShapiro1975} it is globally generic.
\end{example}

\begin{definition}\label{def1}
An irreducible automorphic representation $\rho$ of $\GL_{kc}({\A})$ is a Whittaker--Speh--Shalika representation of type $(k,c)$, or briefly a $(k,c)$ representation, if the following holds.
\begin{enumerate}[leftmargin=*]
\item\label{def:Whittaker--Speh--Shalika 1}
The Fourier coefficient $\Lambda(\varphi)$ does not vanish identically on the space of $\rho$, and
moreover, for all unipotent orbits greater than or non-comparable with $(k^c)$, all corresponding Fourier coefficients are zero for all choices of data.
\item\label{def:Whittaker--Speh--Shalika 2}
Let $\rho_\nu$ denote the irreducible constituent of $\rho$ at a finite place $\nu$, and assume $\rho_\nu$ is unramified. Then for all unipotent orbits greater than or non-comparable with $(k^c)$,
the corresponding twisted Jacquet module of $\rho_\nu$ vanishes (i.e., the local analogue of \eqref{def:Whittaker--Speh--Shalika 1} holds). Moreover,
${\Hom}_{V_{(c^k)}(F_\nu)}(\rho_\nu,\psi_\nu)$ is one-dimensional, where $\psi_\nu$ is given by \eqref{eq:wss character}.
\end{enumerate}
\end{definition}
In the notation of \cite{G2}, condition~\eqref{def:Whittaker--Speh--Shalika 1} may be written as ${\mathcal O}_{\GL_{kc}}(\rho)=(k^c)$. The local vanishing properties of
$\rho_{\nu}$ in the definition imply the global vanishing by a local-global principle (see e.g., \cite[Proposition~1]{JR}). In the opposite direction, the nonvanishing of the global functional \eqref{whspeh1} implies
${\Hom}_{V_{(c^k)}(F_\nu)}(\rho_\nu,\psi_\nu)\ne0$ for all $\nu$ (not only the unramified places), because (in general) the global functional gives rise to nonzero local functionals at all places.

For a unitary continuous character $\eta:F^*\backslash\A^*\rightarrow\C$, let $L^2(\GL_{kc}(F)\backslash \GL_{kc}(\A),\eta)$ be the space of measurable $L^2$-functions $\varphi:\GL_{kc}(F)\backslash \GL_{kc}(\A)\rightarrow\C$ such that $\varphi(zg)=\eta(z)\varphi(g)$ for all $z\in C_{kc}(\A)$. The group $\GL_{kc}(\A)$ acts on $L^2(\GL_{kc}(F)\backslash \GL_{kc}(\A),\eta)$ by right-translation and we denote the action by $g\cdot\varphi$, where $g\in \GL_{kc}(\A)$.

Let $\rho_0$ be an irreducible subrepresentation of $L^2(\GL_{kc}(F)\backslash \GL_{kc}(\A),\eta)$ for some $\eta$
and $\rho=|\det|^r\rho_0$ for some $r\in\R$. Assume $\rho$ is a $(k,c)$ representation.
The space $W(\rho)$ of functions
\begin{align*}
g\mapsto\Lambda(g\cdot\varphi),
\end{align*}
where $\varphi$ varies in the space of $\rho$, is called a global $(k,c)$ model of $\rho$.

Write $\rho=\otimes_{\nu}'\rho_{\nu}$
as a restricted tensor product, with respect to a system $\{\xi_{\nu}^0\}_{\nu\notin S}$ of spherical vectors,
where $S$ is a finite set of places of $F$ depending on $\rho$. For all $\nu\notin S$, $\rho_{\nu}$ is unramified and then the space $\Hom_{V_{(c^k)}(F_\nu)}(\rho_\nu,\psi_\nu)$ is one-dimensional. We fix $\Lambda_{\nu}^0\in{\Hom}_{V_{(c^k)}(F_\nu)}(\rho_\nu,\psi_\nu)$ at these places by requiring $\Lambda_{\nu}^0(\xi_{\nu}^0)=1$. We can further define
\begin{align*}
\Lambda_S\in{\Hom}_{V_{(c^k)}(F_S)}(\rho_S,\psi_S),
\end{align*}
where the subscript $S$ denotes the finite product over the places of $S$ (e.g., $\rho_S=\otimes_{\nu\in S}\rho_{\nu}$), by
\begin{align}\label{eq:partial Lambda S}
\Lambda_S(\xi_S)=\Lambda(\xi_S\otimes'_{\nu\notin S}\xi_{\nu}^0).
\end{align}
Then we have the following decomposition result.
\begin{claim}\label{claim:decomp Lambda}
Let $\varphi$ be a decomposable vector in the space of $\rho$, which we identify with the element
$\xi_S\otimes'_{\nu\notin S}\xi_{\nu}$ in $\otimes_{\nu}'\rho_{\nu}$. Then
for all $g\in\GL_{kc}(\A)$,
\begin{align*}
\Lambda(\rho(g)\varphi)=\Lambda_S(\rho_S(g_S)\xi_S)\prod_{\nu\notin S}\Lambda_{\nu}(\rho_{\nu}(g_{\nu})\xi_{\nu}),
\end{align*}
where $\Lambda_{\nu}$ is a scalar multiple of $\Lambda_{\nu}^0$ for all $\nu\notin S$.
\end{claim}
\begin{proof}
Similar to \cite[Proposition~3.14]{Tk}, which is an adaptation of the decomposition result when uniqueness holds everywhere
(see \cite[\S~4]{Sh}, \cite[Theorem~3.5.2]{Bump1997}).
\end{proof}

Let $F'$ be a local field of characteristic $0$.
\begin{definition}
Let $\sigma$ be a smooth admissible finite length (complex) representation of $\GL_{kc}(F')$.
We say that $\sigma$ is a $(k,c)$ representation if the following holds:
\begin{enumerate}[leftmargin=*]
\item\label{it:local def k,c 1} For all unipotent orbits $\beta$ greater than or non-comparable with $(k^c)$,
$\Hom_{V(\beta)(F')}(\sigma,\psi'_{\beta})=0$, where $V(\beta)$ is the unipotent subgroup corresponding to $\beta$ and
$\psi'_{\beta}$ is any generic character of $V(\beta)$.
\item\label{it:local def k,c 2} The space $\Hom_{V_{(c^k)}(F')}(\sigma,\psi)$ (continuous morphisms over archimedean fields) is
one-dimensional, where $\psi$ is defined by \eqref{eq:wss character}.
\end{enumerate}
\end{definition}

Any nonzero $\lambda\in\Hom_{V_{(c^k)}(F')}(\sigma,\psi)$ is called a $(k,c)$ functional on $\sigma$, and if we fix one such $\lambda$, the
$(k,c)$ model $W(\sigma)$ is the space of functions $g\mapsto\lambda(\sigma(g)\xi')$ where $\xi'$ varies in the space of $\sigma$ and $g\in\GL_{kc}(F')$. We mention that even if $\Hom_{V_{(c^k)}(F')}(\sigma,\psi)$ is not one-dimensional, we can still consider spaces of such functions, defined for each choice of $(k,c)$ functional, but they will typically depend on the choice of the functional, i.e., the model is not unique.

\begin{example}
A $(k,1)$ representation is a representation of $\GL_k$ affording a unique Whittaker model.
\end{example}
\begin{remark}
If we do have local uniqueness everywhere, then we can decompose $W(\rho)=\otimes'_{\nu}W(\rho_{\nu})$ as a restricted tensor product (see the argument in \cite[\S~4]{Sh}). This is the case, for example, when $c=1$ and the representation is globally generic.
\end{remark}

Let $\varphi$ belong to the space of $\rho$.
The Fourier coefficient $\Lambda(\varphi)$ enjoys an extra invariance property. Let $\GL_c^{\Delta}$ denote the image of $\GL_c$ inside $\GL_{kc}$ under the diagonal embedding $h\mapsto h^{\Delta}=\diag(h,h,\ldots,h)$.
\begin{claim}\label{claim:extra invariance}
For all $h\in \SL_c({\A})$, $\Lambda(h^{\Delta}\cdot\varphi)=\Lambda(\varphi)$.
\end{claim}
\begin{proof}
The group $\GL_c^{\Delta}$ is the stabilizer of the character
$\psi$ inside $M_{(c^k)}$. If we expand along any unipotent subgroup of $\SL_c({\A})$, the nontrivial contribution to the expansion vanishes, because the nontrivial term of the expansion is
associated with a unipotent orbit which is greater than or non-comparable with $(k^c)$, while the unipotent orbit attached to $\rho$ is $(k^c)$. See \cite[Proposition~3]{FG2} for details.
\end{proof}

We proceed to show that the generalized Speh representations are $(k,c)$ representations.
Let $\tau$ denote an irreducible unitary cuspidal automorphic representation of $\GL_k({\A})$,
$\underline{s}=(s_1,\ldots,s_c)\in\C^c$, and $E(g;\xi,\underline{s})$ denote the Eisenstein series associated with the induced representation
\begin{align*}
\text{Ind}_{P_{(k^c)}({\A})}^{\GL_{kc}({\A})}
(|\det|^{s_1}\tau\otimes\ldots\otimes |\det|^{s_c}\tau),
\end{align*}
where $\xi$ is a standard section.
Let $\underline{s}_0\in\C^c$ be the point defined by
\begin{align*}
s_1+\ldots+s_c=0;\qquad s_i-s_{i+1}=1; \qquad 1\leq i\leq c-1.
\end{align*}
The series has a simple multi-residue at $\underline{s}_0$,
\begin{align*}
E_{\underline{s}_0}(g;\xi)=\lim\limits_{\underline{s}\to \underline{s}_0}\prod_{i=1}^{c-1}(s_i-s_{i+1}-1)M(w_0,\underline{s})\xi(g,\underline{s}),
\end{align*}
where $M(w_0,\underline{s})$ is the standard intertwining operator defined by (the meromorphic continuation of)
\begin{align*}
M(w_0,\underline{s})\xi(g,\underline{s})=\int\limits_{V_{(k^c)}(\A)}\xi(w_0ug,\underline{s})\,du,
\qquad w_0=\left(\begin{smallmatrix}&&&I_{c}\\&&I_{c}\\&\udots\\I_{c}\end{smallmatrix}\right).
\end{align*}
The automorphic representation
${\mathcal E}_{\tau}$ of $\GL_{kc}({\A})$ generated
by all the residue functions $E_{\underline{s}_0}(\cdot;\xi)$ lies in the discrete spectrum of the space $L^2(\GL_{kc}(F)\backslash \GL_{kc}({\A}),\eta_{\tau}^c)$, where $\eta_{\tau}$ is the central character of $\tau$ (\cite{La5,MW2}), and is irreducible (\cite{MW4}).

Furthermore, write $\tau=\otimes'_\nu\tau_\nu$. At all places, $\tau_{\nu}$ is irreducible unitary and generic, and
at almost all places $\tau_{\nu}$ is unramified and can be written in the form
\begin{align}\label{rep:tau}
\Ind_{B_{\GL_k}(F_\nu)}^{\GL_k(F_\nu)}(\chi_1\otimes\ldots\otimes\chi_k),
\end{align}
where $\chi_1,\ldots,\chi_k$ are unramified quasi-characters of $F_{\nu}^*$. In this case we also denote
\begin{align}\label{sigma k c first def}
\sigma_{k,c}=\Ind_{P_{(c^k)}(F_\nu)}^{\GL_{kc}(F_\nu)}(\chi_1\otimes \ldots \otimes \chi_k),
\end{align}
where each $\chi_i$ is pulled back to a character of $\GL_c(F_{\nu})$ using $\det$.
Observe that since $\tau_{\nu}$ is unitary,
by \cite[Corollary~2.5]{JS} (applied to $\tau_{\nu}$ and $\tau_{\nu}^{\vee}$), $q_{\nu}^{-1/2}<|\chi_i|<q_{\nu}^{1/2}$ for all $i$, where
$q_{\nu}$ is the residue cardinality of $F_{\nu}$. Thus the segments corresponding to
$\chi_i\circ\det$ and $\chi_j\circ\det$, for all $i\ne j$, are not linked, using the terminology of Zelevinsky \cite[\S~3, \S~4]{Z3}, and then $\sigma_{k,c}$ is irreducible \cite[Theorem~4.2]{Z3}.

To extend the applicability of some of our local arguments, we define $\sigma_{k,c}$ in the same way for arbitrary unramified quasi-characters $\chi_i$, i.e., not necessarily the inducing data of an irreducible unitary generic representation of $\GL_k(F_{\nu})$. Then $\sigma_{k,c}$ may be reducible.

\begin{claim}\label{claim:E tau sub and quotient}
Assume $\tau_{\nu}$ is given by \eqref{rep:tau} and let $\sigma_{k,c}$ be given by \eqref{sigma k c first def}.
Then $(\mathcal{E}_{\tau})_\nu=\sigma_{k,c}$.
\end{claim}
\begin{proof}
Since $(\mathcal{E}_{\tau})_\nu$ is irreducible (\cite{MW4}), by construction
it is the unique irreducible unramified quotient of
\begin{align}\label{rep:unr 1}
\Ind_{P_{(k^c)}(F_\nu)}^{\GL_{kc}(F_\nu)}((\tau_\nu\otimes \ldots \otimes \tau_\nu)\delta_{P_{(k^c)}}^{1/(2k)}).
\end{align}
Permuting the inducing characters of $\tau_\nu$ in the full induced representation \eqref{rep:unr 1}, we
reach
\begin{align}\label{rep:unr 2}
\Ind_{B_{\GL_{kc}}(F_\nu)}^{\GL_{kc}(F_\nu)}(\chi_1\delta_{B_{\GL_c}}^{1/2}\otimes \ldots \otimes \chi_k\delta_{B_{\GL_c}}^{1/2}).
\end{align}
Here $\chi_i\delta_{B_{\GL_c}}^{1/2}$ is regarded as a representation of $T_{\GL_c}$.
By Bernstein and Zelevinsky \cite[Theorem~2.9]{BZ2}, the constituents of \eqref{rep:unr 1} and \eqref{rep:unr 2} are isomorphic (including multiplicities). Therefore the unique irreducible unramified quotient of \eqref{rep:unr 1}, which is $(\mathcal{E}_{\tau})_\nu$,
is the unique irreducible unramified constituent of \eqref{rep:unr 2}.

Since the trivial representation is the unique irreducible unramified quotient of
$\Ind_{B_{\GL_c}}^{\GL_c}(\delta_{B_{\GL_c}}^{1/2})$, $\sigma_{k,c}$ is an unramified quotient of \eqref{rep:unr 2}. Therefore $(\mathcal{E}_{\tau})_\nu$ is already a constituent of $\sigma_{k,c}$, which is irreducible because $\tau_{\nu}$ is unitary, as explained above.
\end{proof}

\begin{claim}\label{claim:sigma k c is always at most k c}
The representation $\sigma_{k,c}$ with arbitrary unramified quasi-characters $\chi_i$ is $(k,c)$.
\end{claim}
\begin{proof}
We proceed with local notation and omit references to the field.
For any composition $\lambda$ of $kc$, let $\psi_{\lambda}$ be the character of $N_{\GL_{kc}}$ which restricts to $\psi$ on the simple root subgroups of $M_{\lambda}$ and acts trivially otherwise. Extend the partial order on partitions to compositions by comparing their underlying partitions. According to \cite[Proposition~5.5]{Cai2} (see also \cite{MW3,GGS}), to deduce the vanishing property it is enough to prove that for any $\lambda$ which is greater than or non-comparable with $(k^c)$, the twisted Jacquet module $J_{N_{\GL_{kc}},\psi_\lam}(\sigma_{k,c})$ vanishes.

Assuming $\lam_i>k$ for some $i$, we prove $J_{N_{\GL_{kc}},\psi_\lam}(\sigma_{k,c})=0$. We argue by induction on $k$. When $k=1$ this is trivial because $\sigma_{1,c}$ is a character of $\GL_c$.
Since $\sigma_{k,c}=\sigma_{k-1,c}\times\sigma_{1,c}$, where $\times$ is the parabolic induction functor (see \cite{BZ2}), by \cite[4.14]{BZ2} $J_{N_{\GL_{kc}}\psi_{\lambda}}(\sigma_{k,c})$ is glued from the representations
\begin{align*}
J_{N_{\GL_{(k-1)c}},\psi_{\lam'}}(\sigma_{k-1,c})\times J_{N_{\GL_c},\psi_{\lam''}}(\sigma_{1,c}),
\end{align*}
where $\lam'$ and $\lam''$ vary over the compositions of $(k-1)c$ and $c$ (resp.) such that
$\lambda_i=\lambda_i'+\lambda_i''$ for all $i$.
If $\lam_i>k$ for some $i$, then either $\lambda_i'>(k-1)$ or $\lambda_i''>1$, whence by the induction hypothesis all the representations vanish. Thus $J_{N_{\GL_{kc}},\psi_{\lam}}(\sigma_{k,c})=0$.

It remains to show $\dim J_{V_{(c^k)},\psi}(\sigma_{k,c})=1$. By \cite[Proposition~5.5]{Cai2}, $\dim J_{V_{(c^k)},\psi}(\sigma_{k,c})=\dim J_{N_{\GL_{kc}}\psi_{(k^c)}}(\sigma_{k,c})$, so that we can prove $\dim J_{N_{\GL_{kc}},\psi_{(k^c)}}(\sigma_{k,c})=1$, using induction on $k$. This is clear for $k=1$. Now looking at the filtration above with $\lam=(k^c)$, the contribution is nontrivial if and only if $\lam'=((k-1)^c)$ and $\lam''=(1^c)$. Applying the induction hypothesis, $\dim J_{N_{\GL_{(k-1)c}},\psi_{\lambda'}}(\sigma_{k-1,c})=1$, thus $J_{N_{\GL_{kc}}\psi_{\lambda}}(\sigma_{k,c})$ is one-dimensional.
\end{proof}
\begin{remark}\label{remark:k,c uniqueness}
Note that this result holds without any assumption on $\sigma_{k,c}$.
\end{remark}
\begin{remark}
Fourier coefficients corresponding to $N_{\GL_{kc}}$ and $\psi_{\lambda}$ are called semi--Whittaker coefficients. They are intimately related to Fourier coefficients associated with unipotent orbits, both locally and globally, see \cite{AGS2015,Cai2,GGS,GGS2,MW3}
(an archimedean analog of Claim~\ref{claim:sigma k c is always at most k c} appeared in \cite{AGS2015}).
\end{remark}

\begin{theorem}\label{exthspeh1}
Let $\tau$ be an irreducible cuspidal automorphic representation of $\GL_k({\A})$.
The representation ${\mathcal E}_{\tau}$ is a $(k,c)$ representation.
\end{theorem}
\begin{proof}
Since $\tau=|\det|^d\tau_0$ for some $d\in\R$ and a similar representation $\tau_0$ which is also unitary, we can already assume $\tau$ is unitary.
The global condition of Definition~\ref{def1} was proved in \cite[Proposition~5.3]{G4} (see also \cite{JL2013}). The local condition
now follows immediately from Claims~\ref{claim:E tau sub and quotient} and \ref{claim:sigma k c is always at most k c}.
\end{proof}

\begin{remark}\label{remark:uniqueness comment}
It is important to note that the representations $\sigma_{k,c}$ are special, in the sense that they admit a unique $(k,c)$ functional.
We do not expect an arbitrary irreducible representation $\sigma$ of $\GL_{kc}$ to enjoy this property. In fact for $c>1$, the dimension of $J_{V_{(c^k)},\psi}(\sigma)$ can be infinite. We also mention that for $k=2$, the character \eqref{eq:wss character} is the Shalika character, and again $J_{V_{(c^2)},\psi}(\sigma)$ can be infinite dimensional: to obtain uniqueness results we require the additional invariance property with respect to the reductive part of the stabilizer - the diagonal embedding of $\GL_c$ in $\GL_{2c}$. Nonetheless, for specific representations (similar to $\sigma_{2,c}$) invariance with respect to the reductive part is automatic (see e.g., \cite{BeD}).
\end{remark}

\subsection{Unfolding of the global integral for symplectic groups}\label{global symplectic}
In this section we complete the proof of Theorem~\ref{theorem:main theorem classical groups} for the symplectic group. Let
$G=\Sp_{2n}$ and recall that $c=2n$, $H=\Sp_{4kn}$ and $Q=M\ltimes U$ where
\begin{align*}M=\GL_{2n}\times\ldots\times\GL_{2n}\times\Sp_{4n}.
\end{align*}
Here $\GL_{2n}$ appears $k-1$ times. Identify the quotient $U/[U,U]$ with
\begin{align*}
\Mat_{2n}\oplus\ldots\oplus \Mat_{2n} \oplus \Mat_{2n\times 4n},
\end{align*}
where $\text{Mat}_{2n}$ appears $k-2$ times. For $Y\in \Mat_{2n\times 4n}$ write
\begin{align*}
Y=\begin{pmatrix} Y_1&Z_1&Y_2\\ Y_3&Z_2&Y_4 \end{pmatrix},\qquad Y_i\in \Mat_{n},Z_j\in\Mat_{n\times 2n}.
\end{align*}
Let $\psi_U$ be the pullback to $U$ of the character of $U/[U,U]$ given by
\begin{align}\label{def psiU for Sp}
(X_1,\ldots,X_{k-2},Y)\mapsto\psi(\text{tr}(X_1+\cdots +X_{k-2}+Y_1+Y_4)).
\end{align}
The corresponding Fourier coefficient given by $U$ and $\psi_U$ is associated with the unipotent
orbit $((2k-1)^{2n}1^{2n})$.
The embedding of $G\times G$ in $H$ is given by
\begin{align*}
(g_1,g_2)\mapsto\diag(g_1,\ldots,g_1,\begin{pmatrix} g_{1,1}&&g_{1,2}\\ &g_2&\\ g_{1,3}&&g_{1,4}\end{pmatrix},g_1^*,\ldots,g_1^*),
\end{align*}
where $g_1=\left(\begin{smallmatrix}g_{1,1}&g_{1,2}\\g_{1,3}&g_{1,4}\end{smallmatrix}\right)$, $g_{1,i}\in\Mat_{n}$ and $g_1^*=J_{2n}{}^tg_1^{-1}J_{2n}$ appears $k-1$ times. Note that the middle $4n\times 4n$ block is the standard embedding of $G\times G$ in the middle $\Sp_{4n}$ block of $M$. The involution $\iota$ is defined by ${}^{\iota}g=\iota g\iota^{-1}$ with
\begin{align}\label{eq:iota}
\iota=\left(\begin{smallmatrix}&I_n\\I_n\end{smallmatrix}\right).
\end{align}

We have to show that for $\Real(s)\gg0$,
\begin{align*}
Z(s,\varphi_1,\varphi_2,f)&=
\int\limits_{G(F)\times G(F)\backslash G({\A})\times G({\A})}\,
\int\limits_{U(F)\backslash U({\A})}\varphi_1(g_1)\,\overline{\varphi_2({}^{\iota}g_2)}\,
E(u(g_1,g_2);f,s)\,\psi_U(u)\,du\,dg_1\,dg_2\\
&=\int\limits_{G({\A})}\int\limits_{U_0({\A})}
\langle\varphi_1,\pi(g)\varphi_2\rangle f_{W({\mathcal E}_{\tau})}(\delta u_0(1,{}^{\iota}g),s)
\,\psi_U(u_0)\,du_0\,dg.
\end{align*}
(The right hand side is \eqref{global2}.)
The element $\delta$ is given in \eqref{open2}.

Recall that $P=M_P\ltimes U_P$ is the standard maximal parabolic subgroup of $H$ with $M_P\cong\GL_{2kn}$, and let
$L=(G\times G)U$ denote the subgroup of $Q$ embedded in $H$ as described above. In general for $h,h'\in H$ and $H'<H$, put
\begin{align}\label{eq:conjugations notation}
{}^hh'=hh'h^{-1},\qquad {}^hH'=\{{}^hh':h'\in H'\}.
\end{align}
Unfolding the Eisenstein series in $Z(s,\varphi_1,\varphi_2,f)$, the integral
becomes
\begin{align}\label{sum global unfolding}
&\sum\limits_{\gamma\in P(F)\backslash H(F)/L(F)}\mathrm{I}(\gamma),
\end{align}
where
\begin{align*}
&\mathrm{I}(\gamma)=\int\limits_{L_{\gamma}(F)\backslash L({\A})}
\varphi_1(g_1)\,\overline{\varphi_2({}^{\iota}g_2)}\,
f(\gamma u(g_1,g_2),s)\,\psi_U(u)\,du\,dg_1\,dg_2.
\end{align*}
Here $L_{\gamma}={}^{\gamma^{-1}}\!P\cap L$.
We show that there is a unique representative $\gamma$ such that $\mathrm{I}(\gamma)$ is equal to integral \eqref{global2}, and that for all other representatives $\mathrm{I}(\gamma)=0$. The representative contributing to the sum corresponds to the open orbit.

In general, there are three ways to show $\mathrm{I}(\gamma)=0$ (and we use all three). The first is using the character $\psi_U$. Specifically, if
there is a unipotent subgroup $U'$ of $U$ on which $\psi_U$ is nontrivial and ${}^{\gamma}U'<U_P$, the integral $\mathrm{I}(\gamma)$ vanishes because the integral of $\psi_U$ on $U'(F)\backslash U'(\A)$ is zero.
The second option is to use the cuspidality of $\pi_i$:  if $L_{\gamma}$ contains a unipotent radical $V$ of
a parabolic subgroup of one of the copies of $G$, and the suitable integral over $f$ is invariant under $V(\A)$, then
$\mathrm{I}(\gamma)$ vanishes because $\pi_i$ is cuspidal. The third alternative is to use the smallness of $\mathcal{E}_{\tau}$, which is the $(k,2n)$ representation appearing in the inducing data of the Eisenstein series, and is attached to the unipotent orbit $(k^{2n})$ of $\GL_{2kn}$.
Thus, if we obtain as an inner integration a Fourier coefficient attached to a unipotent orbit which is greater than or non-comparable with $(k^{2n})$, then we get zero contribution from this representative.

We begin with a parametrization of the representatives $\gamma$ of $P(F)\backslash H(F)/L(F)$. Let $N_H$ be the unipotent radical of $B_H$. By the Bruhat decomposition the
double cosets $P\backslash H/B_H=P\backslash H/N_H$ can be represented using Weyl elements, and since $N_H<Q=M U$, every representative $\gamma$ can be written in the form
\begin{align*}
\gamma=wu,
\end{align*}
for a Weyl element $w$ of $H$ and $u\in M\cap N_H$. In the following, we will gradually reduce the number of possible
representatives contributing to \eqref{sum global unfolding}, until we remain with only one, which we will denote by $\delta$.
Hence $Z(s,\varphi_1,\varphi_2,f)$ is equal to $\mathrm{I}(\delta)$, which will then be slightly modified to produce integral \eqref{global2}.

Our main tool for reducing the number of representatives is the following claim. Its proof, along with the proofs of several subsequent statements,
is deferred until later in this section.
\begin{lemma}\label{Lemma A}
If $\gamma=wu$ and there is a one-parameter subgroup $U'$ of $U$ such that $\psi_U|_{U'}\ne1$ and ${}^{w}U'<U_P$, then
$\mathrm{I}(\gamma)=0$.
\end{lemma}
Using the action of $\GL_{2kn}$ on the left and $(I_{2n},G)$ on the right, we may assume
\begin{align}\label{dd1}
w=\begin{pmatrix} \mu_1&&&&\mu_2\\ &\epsilon_1&&\epsilon_2&\\
&&I_{2n}&&\\ &\epsilon_3&&\epsilon_4&\\ \mu_3&&&&\mu_4,
\end{pmatrix},
\end{align}
where $\mu_i\in\Mat_{(k-1)2n}$; $\epsilon_i\in\Mat_{n}$; $\mu_1, \epsilon_1, \epsilon_4$ and $\mu_4$ are
diagonal matrices whose entries are zeros and ones; $\mu_2, \epsilon_2, \epsilon_3$ and $\mu_3$ are matrices
whose nonzero entries are on the anti-diagonal; the nonzero entries of $\mu_2$ and $\epsilon_2$ are ones;
the nonzero entries of $\epsilon_3$ and $\mu_3$ are $-1$.
Since $w\in H$ (and is a Weyl element), it is completely determined by $\mu_1$ and $\epsilon_1$. Further
write $\mu_1=\text{diag}(\mu_{1,1},\mu_{1,2},\ldots,\mu_{1,k-1})$
where $\mu_{1,i}\in\Mat_{2n}$. We shall denote the $(l,l)$-th entry of
$\mu_{1,i}$ by $\mu_{1,i}(l)$. Similarly, $\epsilon_1(l)$ is the $(l,l)$-th coordinate of $\epsilon_1$.

Set $u=u^1u^2$, with an upper triangular matrix $u^2\in \Sp_{4n}$ ($\Sp_{4n}<M$). Using $(G,I_{2n})$ we may assume
\begin{align}\label{uni1}
u^2=\begin{pmatrix} I_{(k-1)2n}&&&&\\ &I_n&T&&\\ &&I_{2n}&T'&\\ &&&I_n&\\ &&&&I_{(k-1)2n}\end{pmatrix},\qquad T=\begin{pmatrix} T_1&0\end{pmatrix},
\end{align}
where $T'$ is defined uniquely by $T$ and the definition of $H$, and $T_1\in\Mat_n$. Put $u^1=(u_{1}, u_{2},\ldots,u_{k-1})$, for upper triangular matrices $u_{i}\in \GL_{2n}$, and regard
$u^1$ as an element in the product of $k-1$ factors of $\GL_{2n}$ in $M$.

For any $h,h'\in H$, write $h\sim h'$ if $PhL=Ph'L$.
\begin{claim}\label{claim:u1 nontrivial}
If $u^1$ is nontrivial, either $\mathrm{I}(\gamma)=0$ or $\gamma\sim wu^2$.
\end{claim}

Fix $\gamma=wu^2$ as in \eqref{dd1} and \eqref{uni1}.
For $1\le i\le n-1$, let $v_i$ denote the simple reflections in the Weyl group of $G$, which are contained inside the standard maximal parabolic subgroup of $G$ whose Levi part is $\GL_n$ (the Siegel parabolic subgroup). Using the reflections
\begin{align}\label{eq:refl}
\mathrm{e}(v_i)=(v_i,I_{2n})\in H,
\end{align}
we may assume $\epsilon_1=\text{diag}(I_j,0_{n-j})$, where $0_{n-j}\in\Mat_{n-j}$ is the zero matrix and $0\le j\le n$.
This implies $\epsilon_4=\text{diag}(0_{n-j},I_j)$, whence
\begin{align}\label{uni2}
u^2=\begin{pmatrix} I_{(k-1)2n+j}&&&&&\\ &I_{n-j}&T&&&\\ &&I_{n}&&&\\ &&&I_n&T'&\\ &&&&I_{n-j}&\\ &&&&&I_{(k-1)2n+j}\end{pmatrix}.
\end{align}
\begin{claim}\label{claim:preparation mu l 1}
Assume that $u^2$ takes the form \eqref{uni2}. Then $\mathrm{I}(\gamma)=0$ unless there is some
$0\le j<n$ such that $\mu_{1,l}=\text{diag}(I_{j_l},0_{2n-j_l})$ with $0\le j_l\le j$ for all $1\le l\le k-1$.
\end{claim}

\begin{claim}\label{claim: a>0}
If $j>0$, $\mathrm{I}(\gamma)=0$.
\end{claim}

It remains to consider $j=0$. Thus $j_l=0$ for all $l$. This already implies
$w=\left(\begin{smallmatrix}&I_{2kn}\\-I_{2kn}\end{smallmatrix}\right)$. Multiplying on
the right by $(I_{2n},\left(\begin{smallmatrix}&I_n\\-I_n\end{smallmatrix}\right))$ (this shifts the block $T$ in \eqref{uni2}
and ${}^w(I_{2n},\left(\begin{smallmatrix}&I_n\\-I_n\end{smallmatrix}\right))\in P$), we need to consider all representatives with
\begin{align}\label{open1}
u^2=u^2[T]=\begin{pmatrix}
I_{(k-1)2n}&&&&&\\ &I_n&&T&&\\ &&I_n&&T'&&\\ &&&I_n&&\\ &&&&I_n&\\
&&&&&I_{(k-1)2n}\end{pmatrix}.
\end{align}
The group $\GL_n\times\GL_n$, embedded inside $G\times G$ as the
group of matrices of the form
\begin{align*}[A,B]=\left (\left(\begin{smallmatrix} A&\\ &A^*\end{smallmatrix}\right), \left(\begin{smallmatrix} B&\\ &B^*\end{smallmatrix}\right)\right )\qquad
(X^*=J_{n}{}^tX^{-1}J_{n}),
\end{align*}
acts on all matrices \eqref{open1} by $[A,B]\cdot u^2[T]=u^2[AT(J_{n}{}^tBJ_{n})]$.
Hence, a set of representatives for this action can be taken to be any $n+1$ matrices whose ranks are $0,1,\ldots, n$, e.g., the matrices
$\left(\begin{smallmatrix} &0\\ I_l&\end{smallmatrix}\right)$, where $0\leq l\leq n$.
\begin{claim}\label{claim:rank X < n}
If $T=\left(\begin{smallmatrix} &0\\ I_l&\end{smallmatrix}\right)$ with $0\leq l<n$, $\mathrm{I}(\gamma)=0$.
\end{claim}
Finally denote the remaining representative by $\delta$,
\begin{align}\label{open2}
\delta=\begin{pmatrix} &I_{2kn}\\ -I_{2kn}\end{pmatrix}\begin{pmatrix}
I_{(k-1)2n}&&&&\\ &I_{2n}&&I_{2n}&\\ &&&I_{2n}&\\
&&&&I_{(k-1)2n}\end{pmatrix}.
\end{align}
The group $L_\delta$ is described as follows. First, inside $G\times G$ we obtain the group
$G$ embedded as $g\mapsto (g,{}^{\iota}g)$. In
$U$, we obtain the subgroup $V$ of all matrices of the form
\begin{align*}
\begin{pmatrix} u_1&&&\\ &I_{2n}&&\\ &&I_{2n}&\\ &&&u_1^*\end{pmatrix}
\begin{pmatrix} I_{(k-1)2n}&u_2&-u_2&0\\ &I_{2n}&&-u_2'\\ &&I_{2n}&-u_2'\\
&&&I_{(k-1)2n}\end{pmatrix}\qquad (u_2'=J_{2n}{}^tu_2J_{(k-1)2n}).
\end{align*}

Since all summands but $\mathrm{I}(\gamma)=\mathrm{I}(\delta)$ vanish,
\begin{align}\label{eq:int one rep}
Z(s,\varphi_1,\varphi_2,f)=\mathrm{I}(\delta)=
\int\limits_{L_{\delta}(F)\backslash L({\A})}
\varphi_1(g_1)\,\overline{\varphi_2({}^{\iota}g_2)}\,
f(\delta u(g_1,g_2),s)\,\psi_U(u)\,du\,dg_1\,dg_2.
\end{align}
We factor the integration through $L_{\delta}(\A)$. The quotient $L_{\delta}\backslash L=(I_{2n},G)\ltimes U_0$, where we recall that $U_0=U\cap U_{P}$.
Therefore \eqref{eq:int one rep} becomes
\begin{align*}
\int\limits_{G(\A)}
\int\limits_{U_0(\A)}
\int\limits_{G(F)\backslash G(\A)}
\int\limits_{V(F)\backslash V(\A)}
\varphi_1(g_1)\,\overline{\varphi_2({}^{\iota}({}^{\iota}g_1g_2))}\,
f(\delta vu_0(g_1,{}^{\iota}g_1)(1,g_2),s)\,\psi_U(vu_0)\,dv\,dg_1\,du_0\,dg_2.
\end{align*}
Conjugating $V$ across $\delta$, we obtain the subgroup $V_{((2n)^{k})}$ of $\GL_{2kn}$,
where $V_{((2n)^{k})}$ is the unipotent subgroup defined in \S~\ref{speh}, and after adjusting $\psi_U(v)$ we obtain the character
\eqref{eq:wss character}, so that we can write the integral in the form
\begin{align*}
\int\limits_{G(\A)}
\int\limits_{U_0(\A)}
\int\limits_{G(F)\backslash G(\A)}
\varphi_1(g_1)\,\overline{\varphi_2(g_1{}^{\iota}g_2)}\,
f_{W(\mathcal{E}_{\tau})}(\delta u_0(g_1,{}^{\iota}g_1)(1,g_2),s)\,\psi_U(u_0)\,dg_1\,du_0\,dg_2.
\end{align*}
We remind the reader that
\begin{align*}
f_{W(\mathcal{E}_{\tau})}(h,s)=\int\limits_{V_{((2n)^k)}(F)\backslash V_{((2n)^k)}({\A})}
f(vh,s)\,\psi^{-1}(v)\,dv,
\end{align*}
where $\psi$ is defined by \eqref{eq:wss character}.
By Claim~\ref{claim:extra invariance}, the $(k,2n)$ functional is invariant under
$\SL_{2n}^{\Delta}(\A)$. Since ${}^{\delta}(g_1,{}^{\iota}g_1)=\text{diag}(g_1,\ldots,g_1)\in \SL_{2n}^{\Delta}(\A)$, the integral over $G(F)\backslash G(\A)$ produces the inner product $\langle\varphi_1,\pi({}^{\iota}g_2)\varphi_2\rangle$. Changing $g_2\mapsto{}^{\iota}g_2$, we reach
\begin{align*}
\int\limits_{G(\A)}
\int\limits_{U_0(\A)}
\langle\varphi_1,\pi(g_2)\varphi_2\rangle
f_{W(\mathcal{E}_{\tau})}(\delta u_0(1,{}^{\iota}g_2),s))\,\psi_U(u_0)\,du_0\,dg_2.
\end{align*}
This is the integral \eqref{global2}.
In particular $\mathrm{I}(\delta)=0$ whence $Z(s,\varphi_1,\varphi_2,f)$ itself vanishes, unless $\pi_1=\pi_2=\pi$.
The proof of the theorem is complete.
\begin{remark}\label{remark:global parameters for SO}
The only differences in the construction for $G=\SO_{2n}$ are that $M=\GL_{2n}\times\ldots\times\GL_{2n}\times\SO_{4n}$
and $\iota=\left(\begin{smallmatrix}&I_n\\-I_n\end{smallmatrix}\right)$. The remaining parameters are similar:
the character $\psi_U$ is still given by \eqref{def psiU for Sp}, the embedding $(g_1,g_2)$ is the same,
$U_0=U\cap U_{P}$, ${\mathcal E}_{\tau}$ (still) corresponds to the unipotent orbit $(k^{2n})$, and
$\delta=\left(\begin{smallmatrix} &I_{2kn}\\ I_{2kn}\end{smallmatrix}\right)
\diag(I_{(k-1)2n},\left(\begin{smallmatrix}I_{2n}&A\\ &I_{2n}\end{smallmatrix}\right),I_{(k-1)2n})$ with
$A=\left(\begin{smallmatrix} -I_{n}\\ &I_{n}\end{smallmatrix}\right)$.
\end{remark}

\begin{proof}[Proof of Lemma~\ref{Lemma A}]
The result clearly follows if ${}^{\gamma}U'<U_P$: indeed by definition, $U'<{}^{\gamma^{-1}}P\cap U<L_{\gamma}$, so that we can factor the integral $\mathrm{I}(\gamma)$ through $U'(F)\backslash U'(\A)=F\backslash\A$ ($U'$ is a one-parameter subgroup), then the left invariance properties of $f$ yield an inner integration $\int_{F\backslash {\A}}\psi(u)du$, which vanishes.

Assume that $U'$ exists with the stated properties. We will show that ${}^{w}U'<U_P$ implies the existence of
another one-parameter subgroup $U''$ of $U$ such that ${}^{\gamma}U''<U_P$ and $\psi_U|_{U''}\ne1$, whence
$\mathrm{I}(\gamma)=0$.

For $1\leq i<j\le 4kn$, let $x_{i,j}$ denote the one-parameter unipotent subgroup along the positive root $(i,j)$ of $H$. Put
\begin{align}\label{def of xi}
x_i=\begin{cases}x_{i,2n+i}&1\le i\le (k-1)2n-n,\\x_{i,4n+i}&(k-1)2n-n+1\le i\le (k-1)2n.\end{cases}
\end{align}
Then $\psi_U|_{x_i}\ne 1$ for all $i$, and $\psi_U$ is trivial on any other subgroup $x_{i,j}$ contained in $U$.
Our assumption is thus ${}^{w}x_i\in U_P$ for some $i$. Since $u\in M$ ($\gamma=wu$),
$U''={}^{u^{-1}}x_i<U$, and because $u$ is unipotent, $\psi_U(x_i(a))=\psi_U({}^{u^{-1}}x_i(a))$ for $a\in\A$.
Hence $U''$ has the required properties.
\end{proof}

\begin{proof}[Proof of Claim~\ref{claim:u1 nontrivial}]
The proof is by induction on $(l,i_1,i_2)$ where $1\leq l\leq k-1$ and $1\leq i_1<i_2\leq 2n$. Assume that the
$(i_1,i_2)$-th entry of $u_{l}$ is nonzero, where $l$ is minimal such that for all
$(i'_1,i'_2)$ with either $i'_1<i'_2<i_2$ or $i_1< i'_1<i'_2=i_2$, and all $l\leq l'\le k-1$, the $(i'_1,i'_2)$-th entry of
$u_{l'}$ is zero. For instance if the $(1,2)$-th coordinate of $u_1$ is nonzero, we set $l=1$ and $(i_1,i_2)=(1,2)$.

In general, if we can write $u^1=u'v^1$ where ${}^wu'\in P$ and $v^1$ is of the same form as $u^1$, that is, an element in the product of $k-1$ factors of $\GL_{2n}$ in $M$, then $h=wu^1u^2\sim wv^1u^2$. Hence if $(l,i_1,i_2)$ are given, it is implicitly assumed
we cannot write $u^1=u'v^1$ where the $(i_1,i_2)$-th coordinate of $v^1$ is zero.

All $1\leq l\leq k-1$ are handled in the same manner, and for simplicity we assume $l=1$.
Denote by $e_{i_1,i_2}\in\Mat_{2n}$ the matrix whose $(i_1,i_2)$-th entry is one and whose other entries are all zero, and
for $a\in\A$ put $y_{i_1,i_2}(a)=I_{2n}+ae_{i_1,i_2}$. Then
\begin{align*}
u^1=(y_{i_1,i_2}(t_1)u_{1}',y_{i_1,i_2}(t_2)u_{2}',\ldots,y_{i_1,i_2}(t_{k-1})u_{k-1}')
\end{align*}
with $t_1\ne 0$. The matrices $u_{l}'\in\Mat_{2n}$ are upper triangular unipotent matrices whose $(j_1,j_2)$-th entries are zero
for all $i_1\le j_1$ and $j_2\le i_2$.
We show that either $\mathrm{I}(\gamma)=0$ or
\begin{align*}
wu\sim wv^1{v}^2,
\end{align*}
where $v^1=(u_{1}',u_{2}',\ldots,u_{k-1}')$ and
$v^2$ takes the form \eqref{uni1}.

There are three cases to consider. Either $i_1<i_2\le n$, $i_1\le n<i_2$, or $n<i_1<i_2$.  We will present the details in the second case; the other two cases are treated similarly.  Assume
\begin{align*}
1\le i_1\le n,\qquad i_2=n+j,\qquad 1\le j\le n.
\end{align*}
There are initially $4$ possibilities for the values of $\mu_{1,1}(i_1)$ and
$\mu_{1,1}(n+j)$. Matrix multiplication implies
\begin{align*}
\mu_{1,1}(i_1)=0,\qquad\mu_{1,1}(n+j)=1,
\end{align*}
since otherwise ${}^{w}(y_{i_1,i_2}(t_1),1,\ldots,1)\in P$, contradicting our assumption.

Now we show
\begin{align}\label{eq:mu l n+j 1}
\forall 2\le l\le k-1, \quad \mu_{1,l}(n+j)=1,\qquad \text{ and  }\quad \epsilon_4(j)=0.
\end{align}
Here $\epsilon_4(j)$ is the $(j,j)$-th coordinate of $\epsilon_4$.
Indeed, if $\mu_{1,2}(n+j)=0$, then ${}^{w}x_{n+j}<U_P$ ($x_{n+j}$ is defined by \eqref{def of xi}) and since $\psi_U|_{x_{n+j}}\ne 1$,
Lemma~\ref{Lemma A} implies $\mathrm{I}(\gamma)=0$. In general for any $2\le l\le k-1$, using the unipotent subgroup $x_{(l-1)n+j}$ we deduce from Lemma~\ref{Lemma A} that $\mu_{1,l}(n+j)=1$. To deduce $\epsilon_4(j)=0$ use Lemma~\ref{Lemma A} with $x_{(k-1)2n-n+j}$. This proves \eqref{eq:mu l n+j 1}.

Since $w\in H$,
\begin{align*}
\epsilon_1(n-j+1)=\epsilon_4(j)=0.
\end{align*}

Next, we show
\begin{align}\label{eq:mu l  1}
\forall 1\le l\le k-1,\quad\mu_{1,l}(n-j+1)=0.
\end{align}
Indeed, if $\mu_{1,l}(n-j+1)=1$, matrix multiplication implies
${}^{w}x_{2nl-n-j+1}<U_P$.
Since $\psi_U|_{x_{2nl-n-j+1}}\ne 1$, it follows from Lemma~\ref{Lemma A} that $\mathrm{I}(\gamma)=0$.

Now there are two possibilities, $\mu_{1,2}(i_1)=1$ or $0$. Assume the former. Using Lemma~\ref{Lemma A} as
above (the proof of \eqref{eq:mu l n+j 1}), we conclude that $\mu_{1,l}(i_1)=\epsilon_1(i_1)=1$ for all
$3\le l\le k-1$. Since $w\in H$, $\epsilon_1(i_1)=1$ implies $\epsilon_4(n-i_1+1)=1$, and an argument similar to the one used for \eqref{eq:mu l  1} implies that $\mu_{1,l}(2n-i_1+1)=0$ for all $1\le l\le k-1$. Then we see
that ${}^{w}(1,y_{i_1,i_2}(t_2),\ldots,y_{i_1,i_2}(t_{k-1}))\in P$ ($i_2=n+j$). Therefore, we may assume that $u$ is such that
\begin{align*}
u^1=(y_{i_1,i_2}(t_1)u_{1}',u_{2}',\ldots,u_{k-1}').
\end{align*}
Let $x(t_1)=y_{i_1,n+j}(-t_1)y_{n-j+1,2n-i_1+1}(\zeta t_1)\in G$, where $\zeta=\pm 1$. For brevity, put
$\mathrm{e}(x(t_1))=(x(t_1),I_{2n})\in H$. We have
\begin{align*}
wu=wu^1u^2\sim wu^1u^2\mathrm{e}(x(t_1))=wu^1\mathrm{e}(x(t_1))v^2,
\end{align*}
where $v^2$ has the same form as $u^2$ ($\mathrm{e}(x(t_1))$ does not commute with $u^2$).
Also
\begin{align*}
u^1\mathrm{e}(x(t_1))&=(y_{i_1,i_2}(t_1)\mathrm{e}(x(t_1)) u_{1}',\mathrm{e}(x(t_1)) u_{2}',\ldots,\mathrm{e}(x(t_1)) u_{k-1}')\\&=(x_{n-j+1,2n-i_1+1}(t_1),\mathrm{e}(x(t_1)),\ldots,\mathrm{e}(x(t_1)))v^1.
\end{align*}
However, from the structure of $w$ it follows that
\begin{align*}
{}^{w}(x_{n-j+1,2n-i_1+1}(t_1),\mathrm{e}(x(t_1)),\ldots,\mathrm{e}(x(t_1)))\in P.
\end{align*}
Hence $wu\sim wv^1v^2$. This completes the first case $\mu_{1,2}(i_1)=1$.

Now assume $\mu_{1,2}(i_1)=0$. Recall that we
assumed
\begin{align*}
u^1=(y_{i_1,i_2}(t_1)u_{1}',y_{i_1,i_2}(t_2)u_{2}',\ldots,y_{i_1,i_2}(t_{k-1})u_{k-1}')
\end{align*}
with $t_1\ne 0$. Let
\begin{align*}
x(a,t_1,t_2)=x_{n+j,2n+i_1}(a)x_{i_1}(-at_1)x_{n+j}(at_2),
\end{align*}
where $x_{n+j,2n+i_1}$ was defined before \eqref{def of xi}. By the definition of $x_{i_1},x_{n+j}$, $x_{n+j,2n+i_1}$ and $\psi_U$,
\begin{align*}
\psi_U(x_{n+j,2n+i_1}(a_1)x_{i_1}(a_2)x_{n+j}(a_3))=\psi_U(a_2+a_3),\qquad\forall a_i\in \A.
\end{align*}
It follows from matrix multiplication that one can choose $u'\in U$ in such a way that $\psi_U(u')=1$ and
\begin{align*}
u^1x(a,t_1,t_2)u'=x_{n+j,2n+i_1}(a)u^1.
\end{align*}
In fact, $u'$ is a product of matrices of the form $x_{2n+c,d}(\cdot)$ where $1\le c<d\le 2n$.
Thus if $t_1\ne t_2$,
\begin{align*}
\psi_U(x(a,t_1,t_2)u')=\psi_U(x(a,t_1,t_2))=\psi(a(t_2-t_1))
\end{align*}
which is nontrivial, and also ${}^{w}x_{n+j,2n+i_1}<U_P$, hence $\mathrm{I}(\gamma)=0$.
Therefore, we may assume $t_1=t_2$.
Next consider the value of $\mu_{1,3}(i_1)$. If $\mu_{1,3}(i_1)=1$, we proceed as in the case $\mu_{1,2}(i_1)=1$. If $\mu_{1,3}(i_1)=0$, continue as above to deduce $t_3=t_1$.
Proceeding in this manner we need to consider the case where
$u^1$ is such that $t_l=t_1$ for all $2\le l\le k-1$. Finally, if $\epsilon_1(i_1)=1$ we proceed as in the case $\mu_{1,2}(i_1)=1$.
If $\epsilon_1(i_1)=0$, since we established $t_{k-1}=t_1$ we can use Lemma~\ref{Lemma A} with the unipotent subgroup
$x_{(k-1)2n-n+j}$. This completes the proof of the case $\mu_{1,2}(i_1)=0$ and thereby the case $i_1\le n<i_2$.
\end{proof}

\begin{proof}[Proof of Claim~\ref{claim:preparation mu l 1}]
Our first step is to prove that if $I(\gamma)\neq0$, then
\begin{align}\label{mu1l}
\forall 1\le l\le k-1,\quad \mu_{1,l}=\text{diag}(\alpha_l,0_{n-j},\beta_l,0_j).
\end{align}
Here $\alpha_l\in\Mat_j$ and $\beta_l\in\Mat_{n-j}$ are diagonal matrices, whose diagonal elements are zeros and ones.
Consider the matrix $\mu_{1,l}$, and assume that for some $1\le i\le j$ we have $\mu_{1,l}(2n-j+i)=1$.
Then, it follows from matrix multiplication that ${}^{w}x_{2nl-n+i}<U_P$. Hence by Lemma~\ref{Lemma A}, $\mathrm{I}(\gamma)=0$ for this representative. Thus we may assume $\mu_{1,l}=\text{diag}(d,0_j)$ where $d$ is a diagonal matrix of size $2n-j$.
Similarly, if $\mu_{1,l}(j+i)=1$ for some $1\le i\le n-j$, we
use the unipotent subgroup $x_{2nl+j+i}$ and Lemma~\ref{Lemma A} to deduce that this
representative contributes zero. Thus, we have shown \eqref{mu1l}.

Consider the matrix $\mu_{1,1}$. Assume $\alpha_1$ contains $j_1$ nonzero entries and $\beta_1$ contains
$b_1$ nonzero entries (these nonzero entries must then be $1$), where $0\le j_1\le j$ and $0\le b_1\le n-j$. Using the Weyl elements $\mathrm{e}(v_1), \dots, \mathrm{e}(v_{j-1})$ and $\mathrm{e}(v_{j+1}),\ldots, \mathrm{e}(v_{n-1})$ (see \eqref{eq:refl}), we have $wu^2\sim w'v^2$ where $v^2$ is a matrix
of the form \eqref{uni2} with possibly a different matrix $T$, and $w'$ is as follows. First, we have
\begin{align*}
\mu_{1,1}=\text{diag}(I_{j_1},0_{n-j_1},I_{b_1},0_{n-b_1})
\end{align*}
(we do not need to use $\mathrm{e}(v_{j})$ for this),
and for $2\le l\le k-1$, the matrices $\mu_{1,l}$ are still of the form \eqref{mu1l},
perhaps with different $\alpha_l$ and $\beta_l$. Finally, the matrix $\epsilon_1$
of both $w$ and $w'$ is the same. Re-denote $w=w'$ and $u^2=v^2$.

Next we claim that if $\mathrm{I}(wu^2)\ne0$, then the first $j_1$ (resp., $b_1$) diagonal entries of $\alpha_2$ (resp., $\beta_2$) are $1$, i.e., $\alpha_2=\diag(I_{j_1},\alpha_2')$, $\beta_2=\diag(I_{b_1},\beta_2')$ and then
\begin{align*}
\mu_{1,2}=\text{diag}(I_{j_1},\alpha_2',0_{n-j},I_{b_1},\beta_2',0_j).
\end{align*}
Indeed, suppose $\mu_{1,2}(i)=0$ for some $1\le i\le j_1$.
Then ${}^{w}x_{i}<U_P$ and we get zero contribution. Similarly for $\beta_2$.

Now using multiplication on the right by the Weyl elements $\mathrm{e}(v_{j_1+1}),\ldots,\mathrm{e}(v_{j-j_1-1})$ (if $j_1<j-1$),
or $\mathrm{e}(v_{j+1}),\ldots,\mathrm{e}(v_{n-b_1-1})$ (if $b_1<n-j-1$), we deduce $wu^2\sim w'v^2$.
Here $\mu_{1,1}$ and $\epsilon_1$ are the same for $w$ and $w'$, but the matrix $\mu_{1,2}$ of
$w'$ is
\begin{align*}
\text{diag}(I_{j_2},0_{n-j_2},I_{b_2},0_{n-b_2}),
\end{align*}
 where
$0\le j_1\le j_2\le j$ and $0\le b_1\le b_2\le n-j$. Again, put $w=w'$ and $u^2=v^2$.

Proceeding this way, we may assume that if $wu^2$ is a representative with nonzero contribution, then
for some $0\le j\le n$ we have $\epsilon_1=\text{diag}(I_j,0_{n-j})$, and there are $0\le j_1\le \ldots\le j_{k-1}\le j$ and
$0\le b_1\le\ldots\le b_{k-1}\le n-j$ such that
\begin{align}\label{eq:mu l before v0}
\forall 1\le l\le k-1,\quad\mu_{1,l}=\text{diag}(I_{j_l},0_{n-j_l},I_{b_l},0_{n-b_l}).
\end{align}
Also $u^2$ has the form \eqref{uni2}.

Note that in all cases we can assume that $wu^2\not\sim w$, in particular that $u^2$ is not the identity matrix, i.e., $T\ne 0$. Otherwise
${}^{wu^2}(I_{2n},G)\cap U_P={}^{w}(I_{2n},G)\cap U_P$, and ${}^{w}(I_{2n},G)\cap U_P$ contains the unipotent
radical of the Siegel parabolic subgroup of $G$. Hence, since $\pi_2$ is cuspidal,
$\mathrm{I}(\gamma)=0$ (then the claim is proved).

The last paragraph implies that we can also assume $j<n$, since otherwise $wu^2\sim w$.

Next assume $b_l\ne 0$, for some $l$. Since $b_l\leq b_{k-1}$, we also have $b_{k-1}\ne0$. The rank of $T$ is at most $n-j$.
Further assume that one of the first $b_{k-1}$ columns of $T'$ is nonzero. We prove that in this case, $wu^2$ contributes zero to the integral. Conjugating by a suitable element in
$\GL_{2kn}(F)$, we may assume that the $(1,l)$-th entry of $T'$ is $1$ for some $1\le l\le b_{k-1}$. Consider the unipotent element
$x(a)=x_{(k-1)2n-n+l,2kn+1}(a)x_{(k-1)2n-n+l}(-a)\in U$. Then $\psi_U(x(a))\ne 1$ and
$u^2x(a)(u^2)^{-1}=x_{(k-1)2n-n+l,2kn+1}(a)$. Also since
$b_{k-1}\ne 0$ and $1\le l\le b_{k-1}$, we have
${}^{w}x_{(k-1)2n-n+l,2kn+1}<U_P$. Hence by Lemma~\ref{Lemma A} we get zero contribution from this element.

Assuming the first $b_{k-1}$ columns of $T'$ are zero, let
\begin{align*}
v_0=(\begin{pmatrix}I_{n-b_{k-1}}&&&\\ &&I_{b_{k-1}}&\\ &-I_{b_{k-1}}&&\\
&&&I_{n-b_{k-1}}\end{pmatrix},I_{2n})\in H.
\end{align*}
Then $v_0$ and $u^2$ commute. Also, since $b_l\le b_{k-1}$ for all
$l$, then  $wv_0\sim w'$ (because $v_0\in G\times G$), where in $w'$ we have
\begin{align}\label{eq:mu l new}
\mu_{1,l}=\text{diag}(I_{j_l},0_{n-j_l-b_l},I_{b_l},0_n).
\end{align}
Here we used the fact that
$b_l\le n-j\le n-j_l$ for all $l$, in order to permute each $\mu_{1,l}$ from \eqref{eq:mu l before v0} to \eqref{eq:mu l new}
(the last $0_{b_l}$ block of $0_{n-j_l}$ was permuted to form $0_n$ with the last $0_{n-b_l}$ block).
Multiplying on the left by a
suitable permutation matrix in $\GL_{2kn}$, we deduce $w'u^2\sim w''v^2$, where $v^2$ is defined as in \eqref{uni2} with
$j$ replaced by $j'$ such that $j<j'$, and $T$ is replaced by a
matrix of size $(n-j')\times n$. The matrix $w''$ has the structure
that
\begin{align*}
\mu_{1,l}=\text{diag}(I_{j_l'},0_{n-j_l'},0_n)
\end{align*}
for some $j_l'\le j'$.

Therefore, we reduced the structure of $\gamma$ to the form $wu^2$ where for some $0\le j<n$,
$\mu_{1,l}=\text{diag}(I_{j_l},0_{2n-j_l})$ with $0\le j_l\le j$ for all $1\le l\le k-1$,
and $u^2$ takes the form \eqref{uni2}.
\end{proof}

\begin{proof}[Proof of Claim~\ref{claim: a>0}]
We have $0<j<n$, $\mu_{1,l}=\text{diag}(I_{j_l},0_{2n-j_l})$ and $0\le j_l\le j$ for all $1\le l\le k-1$. Also $u^2$ is of the
form \eqref{uni2}. First assume $j_l>0$ for some $l$, and let $l_0$ be the minimal $l$ with this property.
Let $V$ be the unipotent radical of the standard parabolic subgroup of $\GL_{2kn}$ corresponding to the following composition
\begin{align*}
((2n)^{l_0-1},j_{l_0},j_{l_0+1},\ldots,j_{k-1},j,b)
\end{align*}
of $2kn$ (here $b$ is uniquely determined by the previous integers). Identify $V/[V,V]$ with the abelian group
\begin{align*}
\Mat_{2n}\oplus\ldots\oplus \Mat_{2n}\oplus\Mat_{j_{l_0}\times (j_{l_0+1})}\oplus\Mat_{(j_{l_0+1})\times (j_{l_0+2})}
\oplus\ldots\oplus \Mat_{j_{k-1}\times j}\oplus \Mat_{j\times b}.
\end{align*}
Let $X$ be the subgroup of $V/[V,V]$ consisting of vectors such that in their projection into the rightmost component $\Mat_{j\times b}$, the first
$b-(2n-j_{l_0})$ columns are zero. Define the unipotent group $Y$ as the preimage of $X$ under the quotient map $V\mapsto V/[V,V]$.
Define a character of $V/[V,V]$ by multiplying $\psi\circ\text{tr}$ of each component. Here $\text{tr}$ of a
non-square matrix is still defined as the sum of entries on the (principal) diagonal. Pulling back this character to a character
of $V$ and restricting to $Y$ yields a character denoted $\psi_Y$.
We see that $Y$ is contained in ${}^{\gamma^{-1}}U\cap \GL_{2kn}$. Thus we obtain the Fourier
coefficient
\begin{align}\label{inner3}
\int\limits_{Y(F)\backslash Y({\A})}f(yh,s)\psi_Y(y)\,dy,\qquad h\in H(\A),
\end{align}
as an inner integration. We claim that this coefficient vanishes for all data. Indeed, after a suitable conjugation of a subgroup of $Y$ by a Weyl element of $\GL_{2kn}(F)$ ($f$ is left-invariant by $\GL_{2kn}(F)$), we obtain in \eqref{inner3} an inner integration
\begin{align*}
\int\limits_{Y'(F)\backslash Y'({\A})}f(y'h,s)\psi(\sum_{i=1}^{k+1}y'_{i,i+1})dy'.
\end{align*}
Here $Y'=V_{(1^{k+1},2kn-k-1)}$. This Fourier coefficient is attached
to the unipotent orbit $((k+1)1^{2kn-k-1})$. See e.g., \cite{G2}. However, the representation $\mathcal{E}_\tau$ is attached
to the unipotent orbit $(k^{2n})$ which is non-comparable with
$((k+1)1^{2kn-k-1})$. Hence this integral and thereby \eqref{inner3} is identically zero.

It remains to consider $\gamma$ such that
for $w$ we have $0<j<n$ and $j_l=0$ for all $l$. This case is omitted here, as it is very similar to the
proof of Claim~\ref{claim:rank X < n} below.
\end{proof}

\begin{proof}[Proof of Claim~\ref{claim:rank X < n}]
Here we consider representatives $wu^2[\left(\begin{smallmatrix}&0\\I_l\end{smallmatrix}\right)]$ with $w=\left(\begin{smallmatrix}&I_{2kn}\\-I_{2kn}\end{smallmatrix}\right)$ and $0\leq l<n$. In fact we can assume $l>0$, since in the proof of Claim~\ref{claim:preparation mu l 1} above we showed $\mathrm{I}(w)=0$
(i.e., $\mathrm{I}(wu^2)=0$ when $u^2$ is the identity, see after \eqref{eq:mu l before v0}).

Put
$\gamma_l=\mu_lwu^2[\left(\begin{smallmatrix}&0\\I_l\end{smallmatrix}\right)]$ where
\begin{align*}
\mu_l=\diag(\left(\begin{smallmatrix}&I_{n-l}\\I_l\end{smallmatrix}\right),I_{(4k-2)n}, \left(\begin{smallmatrix}&I_{l}\\I_{n-l}\end{smallmatrix}\right))\in P
\end{align*}
($\mu_l$ is included to simplify the notation of $L_{\gamma_l}$). The groups $L_{\gamma_l}$ are
similar to what we obtained for the element $\delta$ defined in  \eqref{open2}. In particular,
inside $G\times G$ we obtain the unipotent subgroup
\begin{align}\label{stab10}
(I_{2n},\begin{pmatrix} I_{n-l}&&\\ u_1&I_{l}&\\u_2&&I_l\\u_3&u_2'&u_1'&I_{n-l}\end{pmatrix}).
\end{align}
The important observation is that ${}^{\gamma_l}U\cap\GL_{2kn}=V_{((2n)^{k})}$, but the character we obtain
on this group is different from the one obtained for the representative $\delta$. To describe it,
put $V=V_{((2n)^{k})}$ and identify $V/[V,V]$ with the direct product of $k-1$ copies of $\Mat_{2n}$.
Define the character $\psi_{l}$ of $V$ by pulling back the character
\begin{align*}
(X_1,\ldots X_{k-1})\mapsto\psi(\text{tr}(X_1\left(\begin{smallmatrix}0&\\&I_{n+l}\end{smallmatrix}\right)+X_2+\ldots+X_{k-1})).
\end{align*}
Thus, as an inner integration we obtain the integral
\begin{align}\label{inner1}
\int\limits_{(U\cap U_P)({\A})}
\int\limits_{V(F)\backslash V({\A})}
f(v\gamma_l u_0(g_1,g_2),s)\psi^{-1}_{l}(v)\psi_U(u_0)\,dv\,du_0.
\end{align}

Consider the subgroup
\begin{align*}
V'=\left\{\begin{pmatrix}I_{n-l}&z\\&I_{n+l}\\& &I_{(k-1)2n}\end{pmatrix}\right\}<\GL_{2kn}.
\end{align*}
For a fixed $h\in H(\A)$, expand the function
\begin{align*}
a\mapsto\int\limits_{V(F)\backslash V({\A})}f(v\diag(a,I_{(k-1)2n})h,s)\psi^{-1}_{l}(v)\,dv,\qquad a\in\GL_{2n}(\A),
\end{align*}
along $V'(F)\backslash V'({\A})$. All nontrivial
Fourier coefficients correspond to unipotent orbits which are strictly
greater than $(k^{2n})$, hence by the definition of the $(k,2n)$ representation
(Definition~\ref{def1} part~\eqref{def:Whittaker--Speh--Shalika 1}), they all vanish. We are left with the constant term, so that
integral~\eqref{inner1} becomes
\begin{align*}
\int\limits_{(U\cap U_P)({\A})}
\int\limits_{V'(F)\backslash V'({\A})}
\int\limits_{V(F)\backslash V({\A})}f(vv'\gamma_l u_0(g_1,g_2),s)\psi^{-1}_l(v)\psi_U(u_0)\,dv\,dv'\,du_0.
\end{align*}
As a function of $g_2$, this integral is left invariant under the unipotent radical \eqref{stab10}:
indeed for $v$ of the form \eqref{stab10}, ${}^{\gamma_l}v=\diag({v'},{v'}^*)u'$ where
\begin{align*}
v'=\diag(\begin{pmatrix}I_{n-l}&u_1'&-u_2'\\&I_l\\&&I_l\\&&&I_{n-l}\end{pmatrix},I_{(k-1)2n})\in V'
\end{align*}
and $u'\in U_P$, and now we use
the left invariance properties of the function $f$ and change variables in $V'$. Using this fact we obtain the constant term of
$\varphi_2$ along \eqref{stab10}, which is zero because $\pi_2$ is cuspidal.
Therefore $\mathrm{I}(\gamma_l)=0$ for all $0\leq l< n$.
\end{proof}

\section{Computation of the local factors with unramified data}\label{Computation of the local factors with unramified data}
Recall that by Theorem~\ref{theorem:main theorem classical groups} and using the same notation, for $\Real(s)\gg0$,
\begin{align*}
Z(s,\varphi_1,\varphi_2,f)=
\int\limits_{G({\A})}\int\limits_{U_0({\A})}
\langle\varphi_1,\pi(g)\varphi_2\rangle f_{W({\mathcal E}_{\tau})}(\delta u_0(1,{}^{\iota}g),s)
\,\psi_U(u_0)\,du_0\,dg.
\end{align*}
Assume $\varphi_1$ and $\varphi_2$ are decomposable. Then we can write
$\langle\varphi_1,\pi(g)\varphi_2\rangle=\prod_{\nu}\omega_{\nu}(g_{\nu})$, where $\omega_{\nu}$ is a matrix coefficient of
$\pi_{\nu}^{\vee}$ for all $\nu$. Let $S$ be a sufficiently large finite set of places of $F$ (which depends only on $\tau$),
and write $F_S$, $\tau_S$, etc., for the product of local factors over the places of $S$. Then if $f$ is decomposable, by Claim~\ref{claim:decomp Lambda} we can write
\begin{align*}
f_{W({\mathcal E}_{\tau})}(h,s)=f_{W(({\mathcal E}_{\tau})_S)}(h_S,s)\prod_{\nu\notin S}f_
{W(({\mathcal E}_{\tau})_{\nu})}(h_{\nu},s)\qquad (h\in H(\A)).
\end{align*}
Here $f_{W(({\mathcal E}_{\tau})_S)}$ is the composition of a function in the space of the representation
$\Ind_{P(F_S)}^{H(F_S)}((\mathcal{E}_{\tau})_S\delta_P^s)$ with the functional \eqref{eq:partial Lambda S}, and
$f_{W(({\mathcal E}_{\tau})_{\nu})}$ belongs to the space of $\Ind_{P(F_{\nu})}^{H(F_{\nu})}(W((\mathcal{E}_{\tau})_{\nu})\delta_P^s)$
where $W((\mathcal{E}_{\tau})_{\nu})$ is the unique $(k,c)$ model of $(\mathcal{E}_{\tau})_{\nu}$. Both $f_{W(({\mathcal E}_{\tau})_S)}$ and $f_{W(({\mathcal E}_{\tau})_{\nu})}$ are regarded as complex-valued functions. Then
\begin{align}\label{eq:almost Euler}
Z(s,\varphi_1,\varphi_2,f)=Z_S(s,\omega_S,f_
{W(({\mathcal E}_{\tau})_S)})\prod_{\nu\notin S}Z_{\nu}(s,\omega_{\nu},
f_{W(({\mathcal E}_{\tau})_{\nu})}),
\end{align}
where
\begin{align*}
&Z_S(s,\omega_S,f_
{W(({\mathcal E}_{\tau})_S)})=\int\limits_{G(F_S)}\int\limits_{U_0(F_S)}
\omega_{S}(g)f_
{W(({\mathcal E}_{\tau})_S)}(\delta_{S} u_0(1,{}^{\iota_{S}}g),s)\,\psi_{U,S}(u_0)\,du_0\,dg,\\
&Z_{\nu}(s,\omega_{\nu},f_
{W(({\mathcal E}_{\tau})_{\nu})})=\int\limits_{G(F_{\nu})}\int\limits_{U_0(F_{\nu})}
\omega_{\nu}(g)f_
{W(({\mathcal E}_{\tau})_{\nu})}(\delta_{\nu} u_0(1,{}^{\iota_{\nu}}g),s)\,\psi_{U,\nu}(u_0)\,du_0\,dg.\notag
\end{align*}
This is the weaker form of an Eulerian integral we can obtain, called an ``almost Euler product" by Takeda \cite{Tk}.


In this section we compute the local factors $Z_{\nu}(s,\omega_{\nu},f_{W(({\mathcal E}_{\tau})_{\nu})})$ with unramified data. In order to compute the integral for $G\times \GL_k$, we shall reduce it to the $\GL_n\times \GL_k$ integral. The latter will be further reduced to the case of $n=1$, which is computed directly.
Throughout this section notation is local and references to the field are omitted (e.g., $\Sp_{2n}=\Sp_{2n}(F)$).
Local fields are of characteristic $0$. Representations are always assumed to act on complex vector spaces, and are smooth. Over archimedean fields representations are also assumed to be admissible Fr\'{e}chet of moderate growth (e.g., an irreducible representation is automatically assumed to have these properties as well). Note that the local representation $\pi_{\nu}$ is irreducible and unitary, $\tau_{\nu}$ is irreducible generic, and an unramified twist of $\tau_{\nu}$ is unitary (usually the cuspidal representation $\tau$ is already taken to be unitary, then this twist is trivial and $\tau_{\nu}$ is already unitary). However, parts of the arguments are more convenient to state in a more general context.

\subsection{The integrals for $\Sp_{2n}$ and $\SO_{2n}$}\label{Local factors for classical}
We present the local integrals for $G=\Sp_{2n}$ and $\SO_{2n}$ over a local field $F$.
Let $\pi$ be an irreducible representation of $G$ and $\tau$ be an irreducible, generic representation of $\GL_k$.

We now consider two possible cases for the representation $\tau$: these are the cases relevant for the study of the integrals on the right hand side of \eqref{eq:almost Euler}.
In the first case $\tau$ is a component at a place $\nu$ of an irreducible cuspidal automorphic representation $\Upsilon$ of $\GL_k(\A)$, and define $\rho_c(\tau)=(\mathcal{E}_{\Upsilon})_{\nu}$. The representation $\rho_c(\tau)$ affords at least one $(k,c)$ model $W(\rho_c(\tau))$, which we denote for brevity $W_c(\tau)$ (recall $c=2n$).

In the second case $F$ is $p$-adic and $\tau$ is irreducible, generic and unramified. Write
$\tau=\Ind_{B_{\GL_k}}^{\GL_k}(\chi_1\otimes\ldots\otimes\chi_k)$, then by Claim~\ref{claim:sigma k c is always at most k c} the representation $\sigma_{k,c}=\Ind_{P_{(c^k)}}^{\GL_{kc}}(\chi_1\otimes \ldots \otimes \chi_k)$ is $(k,c)$. The representation $\sigma_{k,c}$ might be reducible, but it is irreducible when the parameters $\chi_i$ are in ``general position". Either way, we let $\rho_c(\tau)$ be the unique irreducible constituent of $\sigma_{k,c}$ which is $(k,c)$.

When the local integrals arise as the integrals in the decomposition \eqref{eq:almost Euler} at $\nu\notin S$,
by Claim~\ref{claim:E tau sub and quotient} both cases above coincide and $\rho_c(\tau)=\sigma_{k,c}$, which affords a unique $(k,c)$ model $W_c(\tau)$.

Recall that $H=\Sp_{2kc}$ if $G=\Sp_{2n}$, and $H=\SO_{2kc}$ when $G=\SO_{2n}$. Let $P=M_P\ltimes U_P$ be the Siegel parabolic subgroup of $H$ with $M_P=\{\diag(g,g^*):g\in\GL_{kc}\}$. Also fix a maximal compact subgroup $K_H$ in $H$.

The local integral takes the form
\begin{align*}
Z(s,\omega,f_{W_c(\tau)})=\int\limits_{G}\int\limits_{U_0}
\omega(g)f_{W_c(\tau)}(\delta u_0(1,{}^{\iota}g),s)\,\psi_U(u_0)\,du_0\,dg.\notag
\end{align*}
Here $\omega$ is a matrix coefficient of $\pi^{\vee}$; $f_{W_c(\tau)}$ belongs to the space of the representation
$\Ind_{P}^{H}(W_c(\tau))$, and is regarded as a complex-valued function; $h\mapsto f_{W_c(\tau)}(h,s)$ is the unique extension of $f_{W_c(\tau)}$ to a standard section of $\Ind_{P}^{H}(W_c(\tau)\delta_P^s)$
(i.e., its restriction to $K_H$ is independent of $s$, see e.g., \cite[\S~IV.1]{W});
$\delta=\delta_0\delta_1$ with
\begin{align*}
&\delta_0=\begin{pmatrix} &I_{kc}\\ \epsilon_0I_{kc}\end{pmatrix},\qquad
\delta_1=\begin{pmatrix}
I_{(k-1)c}&&&&\\&I_n&&-\epsilon_0I_n\\&&I_n&&I_n\\&&&I_n\\&&&&I_n\\&&&&&I_{(k-1)c}\end{pmatrix},\qquad
\epsilon_0=\begin{cases}-1&G=\Sp_{2n},\\1&G=\SO_{2n};
\end{cases}
\end{align*}
the unipotent subgroup $U_0$, the restriction of $\psi_U$ to $U_0$, $(1,g)$ and $\iota$ are defined by
\begin{align*}
&U_0=\left\{\left(\begin{array}{cccc}I_{(k-1)c}&&X&Z\\&I_c&&X'\\&&I_c\\&&&I_{(k-1)c}\end{array}\right)\in H\right\}\qquad\left(\begin{array}{c}{}^tZJ_{(k-1)c}+\epsilon_0J_{(k-1)c}Z=0\\ X'=-\epsilon_0J_{c}{}^tXJ_{(k-1)c}\end{array}\right),\\
&\psi_U(u_0)=\psi(\tr(\left(\begin{smallmatrix}0&I_n\end{smallmatrix}\right)X\left(\begin{smallmatrix}0\\I_n\end{smallmatrix}\right))),\\
&(1,g)=\diag(I_{(k-1)c+n},g,I_{n+(k-1)c}),\qquad \iota=\left(\begin{smallmatrix}&I_n\\-\epsilon_0I_n\end{smallmatrix}\right).
\end{align*}
The integral, at least formally, can be regarded as a morphism in the space
\begin{align}\label{eq:homspace G GLk}
\Hom_{G\times G}(J_{U,\psi_U^{-1}}(\Ind_{P}^{H}(W_c(\tau)\delta_P^s)),\pi^{\vee}\otimes\pi).
\end{align}
Here $J_{U,\psi_U^{-1}}(\cdots)$ is the twisted Jacquet module with respect to $U$ and $\psi_U^{-1}$, regarded as a representation of $G\times G$ by virtue of the embedding $(g_1,g_2)$. This follows from the construction and can be verified directly.

\begin{theorem}\label{theorem:local integrals}
The integrals $Z(s,\omega,f_{W_c(\tau)})$ satisfy the following properties.
\begin{enumerate}[leftmargin=*]
\item \label{item:convergence}They are absolutely convergent in a right half-plane $\Real(s)\gg0$ depending only on the representations.
\item \label{item:nonvanishing}Over non-archimedean fields, one can choose data $(\omega,f_{W_c(\tau)})$ such that the integral is absolutely convergent and equals $1$, for all $s$. Over archimedean fields, for any $s$ there is data $(\omega,f_{W_c(\tau)})$ where $f_{W_c(\tau)}$ is $K_{G}$-finite, such that the integral is holomorphic and nonzero in a neighborhood of $s$.
\item \label{item:meromorphic}They admit meromorphic continuation to the plane, and over a non-archimedean field with residue cardinality $q$, the continuation is a rational function in $q^{-s}$, i.e., belongs to $\C(q^{-s})$.
\end{enumerate}
\end{theorem}
\begin{proof}
We provide only a sketch of the proof, because very similar statements can be found in numerous places in the literature, e.g.,
\cite{GJ,JPSS,GPS,BG,Soudry,Soudry3,GRS4,LR,me4,me5,KM}.

Over non-archimedean fields, convergence follows from the following observations: if $\pi$ is supercuspidal, the matrix coefficient is compactly supported modulo the center; in the general case, one can write $\omega$ as a sum of products of matrix coefficients
of the representations appearing in the cuspidal support of $\pi^{\vee}$, as in \cite{GJ}; the Iwasawa decomposition can then be used to reduce to an integral over the torus of $G$; functions in $W_c(\tau)$ vanish on torus elements outside a cone, similarly to Whittaker functions (see \cite[\S~6]{CS1}); the unipotent integration is handled as in
\cite[\S~4]{Soudry}.

Part~\eqref{item:nonvanishing} is shown by selecting a section $f$ which is supported in the open orbit $P\delta U(G,G)$, such that the function $(u_0,g)\mapsto f(\delta u_0(1,{}^{\iota}g),s)$ on $U_0\times G$ vanishes outside the product of compact neighborhoods in $U_0$ and $G$. The compact neighborhood $\mathcal{N}_G$ in $G$ can be taken to be sufficiently small, such that $\omega$ is constant on $\mathcal{N}_G$. See \cite[\S~4]{RS}. The argument on the support also implies absolute convergence.

Regarding meromorphic continuation one can use Bernstein's continuation principle (in \cite{Banks}), which also implies the rationality statement. To apply this principle we need the following uniqueness result: outside a finite set of values of $q^{-s}$, the space \eqref{eq:homspace G GLk} is at most one-dimensional. The proof of this result is analogous to the global unfolding of the integral (for $k=1$ this uniqueness was proved in \cite{HKS}).
According to \eqref{item:convergence}, in a right half-plane the integral can be regarded as a morphism in \eqref{eq:homspace G GLk}.
Combining the uniqueness result with \eqref{item:nonvanishing}, the meromorphic continuation follows.

Over archimedean fields the proof of \eqref{item:nonvanishing} is similar. The meromorphic continuation is more difficult, because Bernstein's result is not applicable, but one can argue directly by reducing the integral over $G$ to an integral over a torus, then using
the Dixmier--Malliavin Lemma \cite{DM} and asymptotic results as in \cite{Soudry3} (see \cite[\S~3.2]{GRS4}). See also \cite{KM} for an asymptotic expansion of matrix coefficients.
\end{proof}

\subsection{The integrals for $\GL_n$}\label{Local factors for GL}
As mentioned in \S~\ref{global classical}, the global and hence local integrals can also be defined for general linear groups.
As we show in Lemma~\ref{lemma:reduction from classical to GLn} below, the
$G\times \GL_k$ integral reduces to a $\GL_n\times \GL_k$ integral. We therefore define this integral, in a purely local context, where it will be needed.

Let $\pi$ be an irreducible representation $\GL_n$, and $\tau$ and $\tau'$ be irreducible generic representations of $\GL_k$ such that the central character of $\tau$ is the inverse of the central character of $\tau'$ (e.g., $\tau'=\tau^{\vee}$).
The representation $\rho_n(\tau)$ is defined as in \S~\ref{Local factors for classical}, either as $(\mathcal{E}_{\Upsilon})_{\nu}$ when $\tau$ is assumed to be a component of a cuspidal representation $\Upsilon$, or as the unique irreducible $(k,n)$ constituent of
$\sigma_{k,n}=\Ind_{P_{(n^k)}}^{\GL_{kn}}(\chi_1\otimes \ldots \otimes \chi_k)$.
We similarly define $\rho_n(\tau')$. Note that for the applications in this work, one can simply take $\tau'=\tau^{\vee}$.

Put $G=\GL_n$, $H=\GL_{2kn}$ and
$P=P_{(kn,kn)}$. Then $M_P=M_{{(kn,kn)}}$ and $U_P=V_{(kn,kn)}$. Let $Q=M\ltimes U$ be the standard parabolic subgroup of $H$ with $M=M_{(n^{k-1},2n,n^{k-1})}$. To define $\psi_U$, note that $U$ contains a top left and bottom right copies of $V_{(n^{k-1})}$. The character $\psi_U$ is given by the inverse of \eqref{eq:wss character} on each copy of $V_{(n^{k-1})}$, and if the middle $4n\times 4n$ block of $u\in U$ is written in the form
\begin{align*}
\begin{pmatrix}I_n&u_1&u_2&u_3\\&I_n&&u_4\\&&I_n&u_5\\&&&I_n\end{pmatrix},\qquad u_i\in\Mat_n,
\end{align*}
then $\psi_U$ restricts to the character $\psi(\mathrm{tr}(-u_1+u_4))$. The embedding of $G\times G$ in $H$ is defined by
\begin{align*}
(g_1,g_2)=\diag(g_1,\ldots,g_1,g_1,g_2,g_1,\ldots,g_1),\qquad g_1,g_2\in G,
\end{align*}
where $g_1$ appears $k$ times on the left of $g_2$, and $k-1$ times on the right.
The $\GL_n\times\GL_k$ integral is
\begin{align*}
Z(s,\omega,f_{W_n(\tau)\otimes W_n(\tau')})=\int\limits_{G}\int\limits_{U_0}
\omega(g)f_{W_n(\tau)\otimes W_n(\tau')}(\delta u_0(1,g),s)\,\psi_U(u_0)\,du_0\,dg.\notag
\end{align*}
Here $\omega$ is a matrix coefficient of $\pi^{\vee}$; the section $h\mapsto f_{W_n(\tau)\otimes W_n(\tau')}(h,s)$ is on
\begin{align*}
\Ind_{P}^{H}((W_n(\tau)\otimes W_n(\tau'))\delta_P^{s});
\end{align*}
\begin{align*}
&\delta=\delta_0\delta_1,\qquad \delta_0=\begin{pmatrix}&I_{kn}\\I_{kn}\end{pmatrix},\qquad
\delta_1=\begin{pmatrix}I_{(k-1)n}\\&I_n&I_n\\&&I_n\\&&&I_{(k-1)n}\end{pmatrix},\\
&U_0=U\cap U_{P}=\left\{\left(\begin{array}{cccc}I_{(k-1)n}&&X&Z\\&I_n&&Y\\&&I_n\\&&&I_{(k-1)n}\end{array}\right)\right\},\\
&\psi_U(u_0)=\psi(\tr(Y\left(\begin{smallmatrix} I_n\\0\end{smallmatrix}\right))).
\end{align*}
We may also set $c=n$ in this case, to unify the notation. To avoid confusion, we will usually defer from this and write $n$ explicitly.

The immediate analog of Theorem~\ref{theorem:local integrals} applies to the $\GL_n\times\GL_k$ integral.
In particular over non-archimedean fields, it is absolutely convergent in a right half-plane, can be regarded as an element of
\begin{align}\label{eq:homspace GLn GLk}
\Hom_{G\times G}(J_{U,\psi_U^{-1}}(\Ind_{P}^{H}((W_n(\tau)\otimes W_n(\tau'))\delta_P^s)),\pi^{\vee}\otimes\pi),
\end{align}
and admits meromorphic continuation which belongs to $\C(q^{-s})$.

\subsection{Preliminaries for the unramified computation}\label{prelim unr}
Henceforth until the end of the paper, let $F$ be a non-archimedean local field with residue cardinality $q$, $\mathcal{O}$ be its ring of integers, $\varpi\in\mathcal{O}$ be a uniformizer, and normalize the absolute value so that $|\varpi|=q^{-1}$. Let $\psi$ be a nontrivial additive character of $F$, and assume it is unramified, i.e., its conductor is $0$. We choose a Haar measure on $F$ which is self-dual with respect to $\psi$, in particular it assigns the volume $1$ to $\mathcal{O}$.

We fix hyperspecial maximal compact subgroups: $K_{\GL_l}=\GL_l(\mathcal{O})$, $K_{G}=G(\mathcal{O})$ and $K_{H}=H(\mathcal{O})$, where $G$ is either $\Sp_{2n}$, $\SO_{2n}$ or $\GL_n$ and $H$ is defined according to $G$. These choices satisfy the compatibility conditions $(K_G,K_G)<K_H$, $K_{\GL_{kc}}=M_P\cap K_H$ and if $G\ne\GL_n$, ${}^{\iota}K_G=K_G$. The measures of $K_G$ and $K_H$ are normalized to be $1$.

For any irreducible unramified representations $\sigma$ and $\tau$ of $\GL_N$ and $\GL_k$ (resp.), the $L$-function
$L(s,\sigma\times\tau)$ was defined in \cite{JS1,JPSS}. If $t_{\sigma}$ and $t_{\tau}$ are the Satake parameters of $\sigma$ and $\tau$, regarded as representatives of the semi-simple conjugacy classes in $\GL_N(\C)$ and $\GL_k(\C)$ associated to $\sigma$ and $\tau$,
\begin{align*}
L(s,\sigma\times\tau)=\det(1-t_{\sigma}\otimes t_{\tau}q^{-s})^{-1}.
\end{align*}
Moreover, for any finite dimensional representation $\kappa$ of
$\GL_k(\C)$, define
\begin{align*}
L(s,\tau,\kappa)=\det(1-\kappa(t_{\tau})q^{-s})^{-1}.
\end{align*}
In particular by definition
\begin{align*}
L(s,\tau\times\tau)=L(s,\tau,\mathrm{\Sym}^2)L(s,\tau,\wedge^2),
\end{align*}
where $\mathrm{\Sym}^2$ is the symmetric square and $\wedge^2$ is the exterior square representation.
This equality actually holds for any irreducible admissible representation $\tau$ by Shahidi \cite[Corollary~8.2]{Sh5}.
Also denote $L(s,\tau)=L(s,\tau,\mathrm{id})$, where $\mathrm{id}$ is the identity representation of $\GL_k(\C)$.

Let $G$ be either $\Sp_{2n}$ or $\SO_{2n}$. Recall that if $G=\Sp_{2n}$, then ${}^LG=\SO_{2n+1}(\C)$ and we set $N=2n+1$, and
if $G=\SO_{2n}$, then ${}^LG=\SO_{2n}(\C)$ and $N=2n$. Assume that $\pi$ is an irreducible unramified representation of $G$. Let $\Pi$ be the lift of $\pi$ to $\GL_N$, obtained using the Satake isomorphism (\cite{Satake63,Bo}). The representation $\Pi$ is the irreducible unramified representation of $\GL_N$ whose Satake parameter is the transfer of the parameter of $\pi$ under the natural embedding ${}^LG\rightarrow\GL_N(\C)$. Then by definition $L(s,\pi\times\tau)=L(s,\Pi\times\tau)$.

Furthermore, let $R=M_R\ltimes U_R$ be a Siegel parabolic subgroup of $G$. One can choose an irreducible unramified principal series representation $\pi'_n$ of $\GL_n\cong M_R$ such that $\pi$ is the irreducible unramified constituent of $\Ind_{R}^{G}({\pi'_n}^{\vee})$. Then
the definition implies
\begin{align*}
L(s,\pi\times\tau)=[L(s,\tau)]L(s,\pi_n\times\tau)L(s,\pi_n^{\vee}\times\tau),
\end{align*}
where $L(s,\tau)$ appears only when $G=\Sp_{2n}$.

\subsection{Local decomposition of $(k,c)$ functionals}\label{decomposition of functionals}
Let $k$ and $c$ be positive integers. We describe a realization of local $(k,c)$ functionals.
Let
\begin{align*}
\tau=\Ind_{B_{\GL_k}}^{\GL_k}(\chi_1\otimes\ldots\otimes\chi_k),\qquad \sigma_{k,c}=\Ind_{P_{(c^k)}}^{\GL_{kc}}(\chi_1\otimes \ldots \otimes \chi_k),
\end{align*}
where $\tau$ is assumed to be irreducible and $\chi_1,\ldots,\chi_k$ are unramified quasi-characters of $F^*$.
We do not assume at this point that $\sigma_{k,c}$ is irreducible.
According to Claim~\ref{claim:sigma k c is always at most k c},
the space of $(k,c)$ functionals on $\sigma_{k,c}$ is one-dimensional. We construct such a functional using the Jacquet integral.

Put
\begin{align*}
&w_{k,c}=\left(\begin{smallmatrix}&&&I_{c}\\&&I_{c}\\&\udots\\I_{c}\end{smallmatrix}\right)\in \GL_{kc}.
\end{align*}
The following defines a $(k,c)$ functional on $\sigma_{k,c}$:
\begin{align}\label{eq:nk functional using w_{k,c}}
\xi\mapsto\int\limits_{V_{(c^{k})}}\xi(w_{k,c}v)\psi^{-1}(v)\,dv.
\end{align}
Here $\xi$ belongs to the space of $\sigma_{k,c}$, and $\psi$ is defined by \eqref{eq:wss character}.
Twisting the inducing data using auxiliary complex parameters, i.e., replacing $\chi_i$ by $|~|^{\zeta_i}\chi_i$ with $\zeta_i\in\C$ for $1\leq i\leq k$, there is a cone where the integral is absolutely convergent (the proof is identical to the proof for the similar intertwining integral, see e.g., \cite[\S~2]{Sh4}). We can also choose data such that \eqref{eq:nk functional using w_{k,c}} is absolutely convergent and equals $1$, for all choices of $\zeta_i$, namely a function with support in
$P_{(c^k)}w_{k,c}\mathcal{N}$, where $\mathcal{N}$ is a small compact open neighborhood of the identity in $\GL_{kc}$. Since the space of $(k,c)$ functionals on $\sigma_{k,c}$ is one-dimensional, Bernstein's continuation principle (in \cite{Banks}) implies that \eqref{eq:nk functional using w_{k,c}} admits analytic continuation in the parameters $\zeta_i$, and it also follows (by the aforementioned choice of data) that it is a nonzero functional for all $\zeta_i$, in particular on $\sigma_{k,c}$ when setting $\zeta_1=\ldots=\zeta_k=0$.

Let $0\subset\mathbb{V}_1\subset\ldots \subset\mathbb{V}_l\subset\sigma_{k,c}$ be a Jordan--H\"{o}lder series of $\sigma_{k,c}$, and $i$ be minimal such that \eqref{eq:nk functional using w_{k,c}} does not vanish on $\mathbb{V}_i$ ($i$ exists because the functional does not vanish on $\sigma_{k,c}$). Then \eqref{eq:nk functional using w_{k,c}} restricts to a nonzero functional on $\mathbb{V}_i$ and factors through the quotient $\mathbb{V}_{i-1}\backslash \mathbb{V}_i$. Since the Jacquet functor is exact and the dimension of
$J_{V_{(c^k)},\psi}(\sigma_{k,c})$ is $1$, $\mathbb{V}_{i-1}\backslash \mathbb{V}_i$ is the unique irreducible constituent of
$\sigma_{k,c}$ affording a $(k,c)$ functional, and we denote it by $\rho_c(\tau)$.
The corresponding $(k,c)$ model of $\rho_c(\tau)$ is denoted $W_c(\tau)$; it is isomorphic to $\rho_c(\tau)$ and a summand of the $(k,c)$ model of $\sigma_{k,c}$.

We describe a decomposition result for the functional \eqref{eq:nk functional using w_{k,c}} on $\sigma_{k,c}$. For simplicity, the dependence on the twisting parameters $\zeta_i$ is omitted from the notation.
Assume $c=a+b$ with $a,b\geq1$ and put $V=V_{(c^{k})}$.
As in \S~\ref{speh}, denote $v\in V_{(c^k)}$ by
$v=(v_{i,j})_{1\leq i,j\leq k}$, where $v_{i,j} \in\Mat_{c}$.
We rewrite the blocks $v_{i,j}$ of $v\in V$ in the form
\begin{align}\label{eq:the blocks of V}
v_{i,j}=\left(\begin{array}{cc}v_{i,j}^{1}&v_{i,j}^{2}\\v_{i,j}^{3}&v_{i,j}^{4}\end{array}\right),\qquad
v_{i,j}^{1}\in\Mat_a, \quad  v_{i,j}^{4}\in\Mat_b.
\end{align}
For $t=1,\ldots,4$, let $V^t$ be the subgroup consisting of the matrices $v\in V$ such that in each block $v_{i,j}$ with $i<j$, the coordinates of
$v_{i,j}^{t'}$ are zero for all $t'\ne t$.
Also define for any $a,b\geq1$,
\begin{align*}
&l_{a,b}=\left(\begin{array}{cccccccc}I_a\\0&0&I_a\\0&0&0&0&I_a&\ddots\\&&&&&&I_a&0\\0&I_b\\0&0&0&I_b&&\ddots\\&&&&&&0&I_b\end{array}\right)\in\GL_{k(a+b)}.
\end{align*}
For example if $k=3$, $a=2$ and $b=3$,
\begin{align*}
&l_{2,3}=\left(\begin{smallmatrix}I_2\\&&I_2\\&&&&I_2\\&I_3\\&&&I_3\\&&&&&I_3\end{smallmatrix}\right).
\end{align*}
\begin{lemma}\label{lemma:decomposition for V functionals, n=a+b}
For $a,b\geq1$ such that $c=a+b$, and for any $\xi$ in the space of $\sigma_{k,c}$,
\begin{align*}
\int\limits_{V_{(c^{k})}}\xi(w_{k,c}v)\psi^{-1}(v)\,dv=\int\limits_{V^3}\xi_{W_a(\tau)\otimes W_b(\tau)}(l_{a,b}v)\,dv,
\end{align*}
where $\xi_{W_a(\tau)\otimes W_b(\tau)}$ is defined by \eqref{eq:def of xi W W} below and belongs to the space of the representation
\begin{align}\label{eq:decomp a b space}
\Ind_{P_{(ka,kb)}}^{\GL_{kc}}((W_a(\tau)\otimes W_b(\tau))\delta_{P_{(ka,kb)}}^{-1/(2k)}).
\end{align}

This equality is valid in the domain
where \eqref{eq:nk functional using w_{k,c}} is absolutely convergent and in general by meromorphic continuation.
\end{lemma}
\begin{proof}
First note that
\begin{align*}
w_{k,c}={}^{l_{a,b}^{-1}}\diag(w_{k,a},w_{k,b})
\end{align*}
(see \eqref{eq:conjugations notation} for our notation regarding conjugations).
Write the integral over $V_{(c^{k})}$ as an iterated integral $dV^2\,dV^1\,dV^4\,dV^3$. We have
\begin{align*}
w_{k,c}V_{(c^{k})}=l_{a,b}^{-1}\ {}^{\diag(w_{k,a},w_{k,b})l_{a,b}}V^2
\diag(w_{k,a},w_{k,b})
\ {}^{l_{a,b}}(V^1V^4)\ l_{a,b}V^3.
\end{align*}
The character $\psi$ is trivial on $V^2$ and $V^3$.
Then \eqref{eq:nk functional using w_{k,c}} becomes
\begin{align}\label{eq:nk functional using w_{k,c} decomp 1}
\int\limits_{V^3}\int\limits_{V^4}\int\limits_{V^1}
\int\limits_{V^2}\xi(l_{a,b}^{-1}\,({}^{\diag(w_{k,a},w_{k,b})l_{a,b}}v^2)\diag(w_{k,a}{}^{l_{a,b}}v^1,w_{k,b}{}^{l_{a,b}}v^4)l_{a,b} v^3)\psi^{-1}(v^1)\psi^{-1}(v^4)\,dv^2\,dv^1\,dv^4\,dv^3.
\end{align}

Denote, for any $\xi$ in the space of $\sigma_{k,c}$,
\begin{align}\label{eq:def of xi'}
T_{l_{a,b}}\xi(g)=\int\limits_{V^2}\xi(l_{a,b}^{-1}\,({}^{\diag(w_{k,a},w_{k,b})l_{a,b}}v^2)g)\,dv^2\qquad(g\in\GL_{kc}).
\end{align}
For
$A_i\in \GL_a$ and $B_i\in \GL_b$,
\begin{align*}
&T_{l_{a,b}}\xi(\diag(A_1,\ldots,A_{k},B_1,\ldots,B_{k}))\\&=
\delta_{P_{(c^{k})}}^{1/2}(\diag(A_1,B_1,\ldots,A_{k},B_{k}))\\&\quad\times
\prod_{i=1}^k\chi_i(\det A_{i})\chi_i(\det B_{i})
\prod_{i=1}^k|\det A_{k-i+1}|^{b(k-i)}|\det B_{i}|^{-a(k-i)}T_{l_{a,b}}\xi(I_{kc}).
\end{align*}
Then if $A=\diag(A_1,\ldots,A_{k})$, $B=\diag(B_1,\ldots,B_{k})$ and
$\chi(A)=\prod_{i=1}^k\chi_i(\det A_{i})$,
\begin{align*}
T_{l_{a,b}}\xi(\diag(A,B))=&\delta_{P_{(a^{k})}}^{1/2}(A)\delta_{P_{(b^{k})}}^{1/2}(B)\delta_{P_{(ka,kb)}}^{(k-1)/(2k)}(\diag(A,B))
\chi(A)\chi(B)T_{l_{a,b}}\xi(I_{kc}).
\end{align*}
Also $T_{l_{a,b}}\xi(ug)=T_{l_{a,b}}\xi(g)$ for $u\in V_{(ka,kb)}$.
Therefore $T_{l_{a,b}}$ is an intertwining operator from the space of $\sigma_{k,c}$ to the space of the representation
\begin{align}\label{eq:img of xi'}
\Ind_{P_{(ka,kb)}}^{\GL_{kc}}((\sigma_{a,k}\otimes\sigma_{b,k})\delta_{P_{(ka,kb)}}^{-1/(2k)}).
\end{align}
Now \eqref{eq:nk functional using w_{k,c} decomp 1} becomes
\begin{align}\label{eq:nk functional using w_{k,c} decomp 2}
\int\limits_{V^3}\int\limits_{V^4}\int\limits_{V^1}
T_{l_{a,b}}\xi(\diag(w_{k,a}{}^{l_{a,b}}v^1,w_{k,b}{}^{l_{a,b}}v^4)l_{a,b} v^3)\psi^{-1}(v^1)\psi^{-1}(v^4)\,dv^1\,dv^4\,dv^3.
\end{align}
The integrals $dv^1dv^4$ constitute the applications of $(k,a)$ and $(k,b)$ functionals, e.g., ${}^{l_{a,b}}V^1=\diag(V_{(a^k)},I_{kb})$. Hence
if
\begin{align}\label{eq:def of xi W W}
\xi_{W_a(\tau)\otimes W_b(\tau)}(g)=\int\limits_{V^4}\int\limits_{V^1}
T_{l_{a,b}}\xi(\diag(w_{k,a}{}^{l_{a,b}}v^1,w_{k,b}{}^{l_{a,b}}v^4)g)\psi^{-1}(v^1)\psi^{-1}(v^4)\,dv^1\,dv^4,
\end{align}
the function $\xi_{W_a(\tau)\otimes W_b(\tau)}$ belongs to the space of \eqref{eq:decomp a b space}. Integral~\eqref{eq:nk functional using w_{k,c} decomp 2} is equal to
\begin{align*}
\int\limits_{V^3}\xi_{W_a(\tau)\otimes W_b(\tau)}(l_{a,b} v^3)\,dv^3,
\end{align*}
as claimed.
\end{proof}
\begin{corollary}\label{corollary:k n functionals nonzero}
Assume $1-q^{-s}\chi_i(\varpi)\chi_j^{-1}(\varpi)\ne0$ for $\Real(s)\geq1$, for all $i\ne j$. Then
the functional \eqref{eq:nk functional using w_{k,c}} is nonzero on the normalized unramified vector $\xi$ in the space of $\sigma_{k,c}$.
Note that the assumption always holds if we consider $|~|^{\zeta_i}\chi_i$ instead of $\chi_i$ and take $\zeta_i\gg\zeta_{j}$ for all $i<j$.
\end{corollary}
\begin{proof}
We use induction on $c$. For $c=1$, the functional \eqref{eq:nk functional using w_{k,c}} is the usual Whittaker functional given by a Jacquet integral. Since $\sigma_{k,1}=\tau$ which is irreducible, this functional is nonzero on $\xi$ by the Casselman--Shalika formula \cite{CS2} and the irreducibility criterion for principal series representations (e.g., \cite{BZ2,CS1}).

Assume $c>1$ and apply Lemma~\ref{lemma:decomposition for V functionals, n=a+b} using $a=1$ and $b=c-1$. Conjugating $V^3$ by $l_{1,c-1}$, we obtain
\begin{align*}
\int\limits_{V^3}\xi_{W_1(\tau)\otimes W_{c-1}(\tau)}({}^{l_{1,c-1}}v)\,dv.
\end{align*}
We will show that the coordinates of $v$ can be assumed
to lie in $\mathcal{O}$. If the coordinates of $v$ are given by \eqref{eq:the blocks of V}, i.e., the nontrivial coordinates of $v$ are the blocks $v_{i,j}^{3}$ where $v_{i,j}^{3}\in\Mat_{c-1\times 1}$, then
\begin{align*}
{}^{l_{1,c-1}}v=\begin{pmatrix}I_{k}\\ [v]&I_{k(c-1)}\end{pmatrix},\qquad
& [v]=\begin{pmatrix}0&v_{1,2}^3&\cdots&v_{1,k}^3\\\vdots&\ddots&\ddots&\vdots\\\vdots&&&v_{k-1,k}^3\\0&\cdots&\cdots&0\end{pmatrix}.
\end{align*}
(Direct matrix multiplication.) Consider matrices of the form
\begin{align*}
\begin{pmatrix}I_{k}& [x]\\&I_{k(c-1)}\end{pmatrix},\qquad
& [x]=\begin{pmatrix}
0&0&\cdots&0\\
\vdots&x_{1,2}&\ddots&\vdots\\
\vdots&\vdots&\ddots&0\\
0&x_{1,k}&\cdots&x_{k-1,k}\end{pmatrix},\qquad x_{i,j}\in\Mat_{1\times c-1}.
\end{align*}
For each $1\leq i\leq k-1$ and $2\leq j\leq k$, let $\mathcal{X}_{i,j}$ be the subgroup consisting of matrices
of this form where the coordinates of $[x]$
are zero except at the block $x_{i,j}$, which takes arbitrary coordinates in $\mathcal{O}$. Starting with
$i=k-1$ and $j=k$, for $x\in\mathcal{X}_{i,j}$ we have
\begin{align*}
\xi_{W_1(\tau)\otimes W_{c-1}(\tau)}({}^{l_{1,c-1}}v)=
\xi_{W_1(\tau)\otimes W_{c-1}(\tau)}({}^{l_{1,c-1}}vx)=\psi(\text{tr}(v_{i,j}^3x_{i,j}))
\xi_{W_1(\tau)\otimes W_{c-1}(\tau)}({}^{l_{1,c-1}}v).
\end{align*}
The first equality follows since $\xi$ is unramified, the second follows from the invariance properties of $W_{c-1}(\tau)$.
Thus the coordinates of $v_{i,j}^3$ can be taken in $\mathcal{O}$, and since $\xi$ is unramified, the integration over these coordinates becomes an integral of the constant function $1$ over $\Mat_{c-1\times1}(\mathcal{O})$. Since the measure of $\mathcal{O}$ was chosen to be $1$, this integration evaluates to the constant $1$. Proceeding with this argument for
$(i,j)=(k-2,k-1), (k-3,k-2)$, etc., the blocks $v_{l,l+1}^3$ can each be taken in $\mathcal{O}$ and the integral over these coordinates is $1$, for $l=k-1,\ldots,1$. Then we continue with
$v_{k-2,k}^3$ using $\mathcal{X}_{k-2,k}$ and in this way show that all the diagonal $v_{l,l+2}^3$ can be taken in $\mathcal{O}$, $l=k-2,\ldots,1$. The last block to consider is $v_{1,k}^3$, which
we handle using $\mathcal{X}_{1,k}$. We deduce
\begin{align*}
\int\limits_{V^3}\xi_{W_1(\tau)\otimes W_{c-1}(\tau)}({}^{l_{1,c-1}}v)\,dv&=
1\times\xi_{W_1(\tau)\otimes W_{c-1}(\tau)}(I_{kc})
\\&=\int\limits_{V^4}\int\limits_{V^1}
T_{l_{1,c-1}}\xi(\diag(w_{k,1}{}^{l_{1,c-1}}v^1,w_{k,c-1}{}^{l_{1,c-1}}v^4))\psi^{-1}(v^1)\psi^{-1}(v^4)\,dv^1\,dv^4,
\end{align*}
where for the second equality we used \eqref{eq:def of xi W W}, and $T_{l_{1,c-1}}$ is the intertwining operator given by
\eqref{eq:def of xi'}. Since
$\xi$ is unramified, $T_{l_{1,c-1}}\xi$ is a scalar multiple of the normalized unramified vector $\xi'$ in the space of \eqref{eq:img of xi'}. We only need to show that this scalar is nonzero. We may decompose $T_{l_{1,c-1}}$ into rank-$1$ intertwining operators on spaces of the form
\begin{align*}
\Ind_{B_{\GL_2}}^{\GL_2}(|~|^{-(c-2l+1)/2}\chi_i\otimes|~|^{-(c-2l'+1)/2}\chi_j),\qquad i<j,\qquad l'\leq l-1.
\end{align*}
According to the Gindikin--Karpelevich formula (\cite[Theorem~3.1]{CS1}), each intertwining operator takes the normalized unramified vector in this space to a constant multiple of the normalized unramified vector in its image, and this constant is given by
\begin{align*}
\frac{1-q^{-1-l+l'}\chi_i(\varpi)\chi_j^{-1}(\varpi)}{1-q^{-l+l'}\chi_i(\varpi)\chi_j^{-1}(\varpi)}.
\end{align*}
Since $-l+l'\leq-1$, if the quotient has a zero or pole, then
$1-q^{-s}\chi_i(\varpi)\chi_j^{-1}(\varpi)=0$ for $\Real(s)\geq1$, contradicting our assumption.

We deduce that
\begin{align*}
&\int\limits_{V^4}\int\limits_{V^1}
T_{l_{1,c-1}}\xi(\diag(w_{k,1}{}^{l_{1,c-1}}v^1,w_{k,c-1}{}^{l_{1,c-1}}v^4))\psi^{-1}(v^1)\psi^{-1}(v^4)\,dv^1\,dv^4
\end{align*}
is a nonzero multiple of
\begin{align*}
&\int\limits_{V^4}\int\limits_{V^1}
\xi'(\diag(w_{k,1}{}^{l_{1,c-1}}v^1,w_{k,c-1}{}^{l_{1,c-1}}v^4))\psi^{-1}(v^1)\psi^{-1}(v^4)\,dv^1\,dv^4.
\end{align*}
Since
$\xi'(\diag(x,I_{k(c-1)}))$ (resp., $\xi'(\diag(I_{k},y))$) is the normalized unramified vector in the space of $\sigma_{k,1}$ (resp., $\sigma_{k,c-1}$), and by the inductive hypothesis the $(k,1)$ (resp., $(k,c-1)$) functional is nonzero on this element,
we conclude that the $(k,c)$ functional is nonzero on $\xi$.
\end{proof}

Recall the diagonal embedding $\GL_c\rightarrow \GL_{kc}$ given by $h\mapsto h^{\triangle}=\diag(h,\ldots,h)$. We prove a local analog of
Claim~\ref{claim:extra invariance}.
\begin{proposition}\label{proposition:k n functionals invariance}
Let $\lambda$ be a $(k,c)$ functional on $\sigma_{k,c}$. For a vector $\xi$ in the space of $\sigma_{k,c}$, let
$\lambda_{\xi}(g)=\lambda(g\cdot\xi)$ ($g\in\GL_{kc}$). Then $\lambda_{\xi}(h^{\triangle}g)=\tau((\det{h})I_k)\lambda_{\xi}(g)$ for all $h\in \GL_c$. In particular $\lambda_{\xi}(h^{\triangle}g)=\lambda_{\xi}(g)$ for $h\in \SL_c$. The same assertion applies to the representation $\rho_c(\tau)$.
\end{proposition}
\begin{proof}
Since the representation $\sigma_{k,c}$ admits a unique $(k,c)$ functional, we can
assume that $\lambda$ is given by \eqref{eq:nk functional using w_{k,c}}.
Since $h^{\triangle}$ normalizes $V_{(c^k)}$ without changing the measure, stabilizes $\psi$ and commutes with $w_{k,c}$,
\begin{align*}
\lambda_{\xi}(h^{\triangle}g)=\int\limits_{V_{(c^{k})}}\xi(h^{\triangle}w_{k,c}vg)\psi^{-1}(v)\,dv.
\end{align*}
Now the assertion follows because $\xi(h^{\triangle}g)=\prod_i\chi_i(\det(h))\xi(g)$ by the definition of $\sigma_{k,c}$. Since $W_c(\tau)$ is a summand of the $(k,c)$ model of $\sigma_{k,c}$, the same result is valid for $\rho_c(\tau)$.
\end{proof}
\begin{remark}
Since the space of $(k,c)$ functionals on $\sigma_{k,c}$ is one-dimensional, $J_{V_{(c^{k})},\psi}(\sigma_{k,c})$ is one-dimensional. Hence
a priori $\GL_c^{\triangle}$ acts by a character (which must then be trivial on $\SL_c^{\triangle}$).
\end{remark}

\begin{claim}\label{claim:tau unitary k n functionals nonzero}
Let $\tau$ be an unramified twist of an irreducible generic unramified and unitary representation of $\GL_k$. Then
$\sigma_{k,c}$ is irreducible and the assumption of Corollary~\ref{corollary:k n functionals nonzero} is satisfied.
\end{claim}
\begin{proof}
Let $d\in\R$ be such that $\tau=|\det|^{d}\tau_0$, where
$\tau_0$ is an irreducible generic unramified and unitary representation of $\GL_k$.
Write $\tau_0=\Ind_{B_{\GL_k}}^{\GL_k}(\chi^0_1\otimes\ldots\otimes\chi^0_k)$.
As we explained in \S~\ref{speh}, the representation
$\Ind_{P_{(c^k)}}^{\GL_{kc}}(\chi^0_1\otimes \ldots \otimes \chi^0_k)$ is irreducible and
$q^{-1/2}<|\chi^0_i(\varpi)|<q^{1/2}$ for all $i$. Since $\chi_i=|~|^d\chi^0_i$, we conclude that
$\sigma_{k,c}$ is irreducible and
$|\chi_i(\varpi)\chi_j^{-1}(\varpi)|=|\chi^0_i(\varpi)(\chi^0_j)^{-1}(\varpi)|<q$.
\end{proof}
\subsection{The computation of the integrals with unramified data}
In this section we compute the integrals from \eqref{eq:almost Euler} with unramified data.
We use the notation and conventions from \S~\ref{prelim unr}. Let $G=\Sp_{2n}$ or $\SO_{2n}$. Let $\pi$ be an irreducible unramified representation of $G$. Let $\tau$ be an unramified twist of an irreducible unitary generic unramified representation of $\GL_k$.
Recall that the $G\times\GL_k$ integrals
were described in \S~\ref{Local factors for classical} and the $\GL_n\times\GL_k$ integrals were defined in \S~\ref{Local factors for GL}. We also use notation from these sections.

Define
\begin{align}\label{eq:dtau}
d_{\tau}(s)=&
\left[\frac{L(\alpha s+1/2,\tau)}{L(\alpha s+n+1/2,\tau)}\right]
\prod_{1\leq j\leq \lfloor n/2\rfloor}
\frac{L(2\alpha s+2j,\tau,\mathrm{Sym}^2)}{L(2\alpha s+2j+2n-2\lfloor n/2\rfloor-1,\tau,\mathrm{Sym}^2)}\\\nonumber&\times
\prod_{1\leq j\leq \lceil n/2\rceil}
\frac{L(2\alpha s+2j-1,\tau,\wedge^2)}{L(2\alpha s+2j+2n-2\lceil n/2\rceil,\tau,\wedge^2)},
\end{align}
where if $G=\Sp_{2n}$, $\alpha=2kn+1$, and if $G=\SO_{2n}$, $\alpha=2kn-1$; and the factor in square brackets here and in Theorem~\ref{theorem:unramified computation for Sp(2n),SO(2n)} below is included only for $\Sp_{2n}$.

Let $\omega^0$ be the unramified matrix coefficient of $\pi^{\vee}$ normalized such that $\omega^0(I_{2n})=1$. Let $f_{W_c(\tau)}^0$ be the unramified element in the space of $\Ind_{P}^{H}(W_c(\tau))$ normalized by $f_{W_c(\tau)}^0(I_{2kc})=1$, and extended to a standard section of
$\Ind_{P}^{H}(W_c(\tau)\delta_P^s)$.

The following lemma reduces the $G\times \GL_k$ integral to the
$\GL_n\times \GL_k$ integral. Its proof occupies \S~\ref{proof of reduction lemma} below.
\begin{lemma}\label{lemma:reduction from classical to GLn}
Assume $\pi$ is an irreducible quotient of $\Ind_{R}^{G}(\pi_n)$, where
$R=M_R\ltimes U_R$ is the Siegel parabolic subgroup with $M_R=\{\diag(a,a^*):a\in\GL_n\}$ and $\pi_n$
is an irreducible unramified representation of $\GL_n$. Let $\omega_n^0$ be the normalized unramified matrix coefficient of $\pi_n^{\vee}$ and $\rho_{W_n(\tau)\otimes W_n(\tau^{\vee})}^0$ be the normalized unramified function in
the space of
\begin{align*}
\Ind_{P_{(kn,kn)}}^{\GL_{2kn}}(W_n(\tau)\otimes W_n(\tau^{\vee})).
\end{align*}
Then
\begin{align}\label{eq:lemma decomp classical}
Z(s,\omega^0,f_{W_c(\tau)}^0)=d_{\tau}(s)Z(\alpha s/(kn),\omega_n^0,\rho_{W_n(\tau)\otimes W_n(\tau^{\vee})}^0).
\end{align}
\end{lemma}
Since $\pi$ is irreducible and unramified, by Langlands' classification one can choose an unramified principal series representation ${\pi'_n}$ of $\GL_n$, such that $\pi$ is a quotient of $\Ind_{R}^{G}(\pi'_n)$ and in addition, ${\pi'_n}$ contains an irreducible unramified quotient ${\pi_n}$. Then $\Ind_{R}^{G}(\pi_n)$ is an unramified quotient of $\Ind_{R}^{G}({\pi'_n})$, hence contains $\pi$.
Thus the assumption of the lemma is always satisfied.
\begin{theorem}\label{theorem:unramified computation for GL(n)}
For irreducible unramified representations $\pi$ of $\GL_n$, $\tau$ and $\tau'$ of $\GL_k$ (as in \S~\ref{Local factors for GL}), if $\omega^0$ is the normalized unramified matrix coefficient of $\pi^{\vee}$ and
$f_{W_n(\tau)\otimes W_n(\tau')}^0$ is the normalized unramified element in the space of $\Ind_{P_{(kn,kn)}}^{\GL_{2kn}}(W_n(\tau)\otimes W_n(\tau'))$,
\begin{align*}
Z(s,\omega^0,f_{W_n(\tau)\otimes W_n(\tau')}^0)=\frac{L(kns+1/2,\pi^{\vee}\times\tau)L(kns+1/2,\pi\times{\tau'}^{\vee})}{\prod_{j=1}^{n}L(2kns+j,\tau\times{\tau'}^{\vee})}.
\end{align*}
\end{theorem}
This theorem is proved in \S~\ref{Local factors for GLn}.
As a corollary we obtain the computation of the $G\times\GL_k$ integrals with unramified data.
\begin{theorem}\label{theorem:unramified computation for Sp(2n),SO(2n)}\footnote{There was a typo in the formula in the original announcement
\cite{CFGK}; we would like to thank Dihua Jiang for pointing it out to us.}
When all data are unramified,
\begin{align*}
Z(s,\omega^0,f_{W_c(\tau)}^0)=\frac{L(\alpha s+1/2,\pi\times\tau)}
{[L(\alpha s+n+1/2,\tau)]\prod\limits_{1\leq j\leq n}L(2\alpha s+2j,\tau,\wedge^2)
L(2\alpha s+2j-1,\tau,\mathrm{Sym}^2)}.
\end{align*}
\end{theorem}
\begin{proof}[Proof of Theorem~\ref{theorem:unramified computation for Sp(2n),SO(2n)}]
According to Theorem~\ref{theorem:unramified computation for GL(n)} the $\GL_n\times\GL_{k}$ integral
\begin{align*}
Z(\alpha s/(kn),\omega_{n}^0,\rho_{W_n(\tau)\otimes W_n(\tau^{\vee})}^0)=\frac{L(\alpha s+1/2,\pi_n^{\vee}\times\tau)L(\alpha s+1/2,\pi_n\times\tau)}{\prod_{j=0}^{n-1}L(2\alpha s+j+1,\tau\times\tau)}.
\end{align*}
Combining this with Lemma~\ref{lemma:reduction from classical to GLn}, the formula~\eqref{eq:dtau} for $d_{\tau}(s)$ and using the identities
\begin{align*}
&L(s,\pi\times\tau)=[L(s,\tau)]L(s,\pi_n\times\tau)L(s,\pi_n^{\vee}\times\tau),\\
&L(s,\tau\times\tau)=L(s,\tau,\mathrm{Sym}^2)L(s,\tau,\wedge^2)
\end{align*}
(see \S~\ref{prelim unr}) gives the result.
\end{proof}
Now we can deduce the meromorphic continuation of the global partial $L$-function.
\begin{theorem}\label{theorem:mero}
Let $\pi$ and $\tau$ be irreducible automorphic cuspidal representations of $G(\A)$ and $\GL_k(\A)$, respectively. Let $S$ be a finite set of places of $F$, outside which all data are unramified.
Then $L^S(s,\pi\times\tau)$ admits meromorphic continuation to $\C$.
\end{theorem}
\begin{remark}
This theorem is not new, it follows from Langlands' general theory of Eisenstein series \cite{La2,La5}, which is applicable in a much wider setting (e.g., for a large class of groups $G$). It is provided as an illustration of the applicability of our results.
\end{remark}
\begin{proof}
According to Theorem~\ref{theorem:main theorem classical groups}, the global integral $Z(s,\varphi_1,\varphi_2,f)$ admits meromorphic continuation to $\C$, and for $\Real(s)\gg0$ coincides with \eqref{global2}. For decomposable data, we can write \eqref{global2} in the form \eqref{eq:almost Euler}: the product of an integral $Z_S$ and infinitely many local integrals $Z_{\nu}$ for the places $\nu\notin S$.
The integral $Z_S$ is meromorphic and can be chosen to be holomorphic and nonzero, in a neighborhood of a given $s\in\C$.
This can be proved along the lines of Theorem~\ref{theorem:local integrals} (which deals with one place).
Therefore the product of integrals over the places outside $S$ admits meromorphic continuation.

For each integral $Z_{\nu}$ with $\nu\notin S$, all data are unramified:
the local representations $\pi_{\nu}$ and $\tau_{\nu}$ are irreducible unramified, $\tau_{\nu}$ is also generic, and $\psi_{\nu}$ is unramified. In addition,
because $\tau=|\det|^d\tau_0$ for some $d\in \R$ where $\tau_0$ is unitary, $\tau_{\nu}$ is the unramified twist of a unitary representation.

By virtue of Theorem~\ref{theorem:unramified computation for Sp(2n),SO(2n)} (applied to $Z_{\nu}$), the product of local integrals over all $\nu\notin S$ is precisely
$L^S(s,\pi\times\tau)$ divided by products of partial $L$-functions $L^S(s,\tau)$, $L^S(s,\tau,\wedge^2)$ and $L^S(s,\tau,\mathrm{Sym}^2)$ (with $s$ replaced by a suitable linear polynomial of $s$). Since by Langlands' general theory of Eisenstein series \cite{La2,La5}, each of the $L$-functions in the denominator is meromorphic, we deduce that $L^S(s,\pi\times\tau)$ admits meromorphic continuation.
\end{proof}
\begin{remark}
In a subsequent paper (\cite{CFK}) we develop the local theory of the doubling integrals over all places of $F$ (including the ramified and archimedean ones), and define the local $\gamma$-, $\epsilon$- and $L$-factors. This enables us to define the complete $L$-function
$L(s,\pi\times\tau)$, and study its analytic behavior.
In particular we show that it satisfies a global functional equation $L(s,\pi\times\tau)=\epsilon(s,\pi\times\tau)L(1-s,\pi^{\vee}\times\tau^{\vee})$.
\end{remark}
\subsubsection{Proof of Lemma~\ref{lemma:reduction from classical to GLn}}\label{proof of reduction lemma}
The proof consists of two steps.
First, we use the realization of the $(k,c)$ functional using $(k,a)$ and $(k,b)$ functionals given in \S~\ref{decomposition of functionals}, for $a=b=n$. Note that here $c=2n$. This changes the inducing data of $f_{W_c(\tau)}$. Then we write the unipotent integration over $U_0$ as an iterated integral, where the inner part is ``almost" an intertwining operator (some coordinates are missing, they are taken from $U_R$), the middle part is the unipotent integration of a $\GL_n\times \GL_k$ integral,
and the outer integral reduces to a constant. This essentially completes the reduction, with the remaining part being to compute the
proportionality factor $d_{\tau}(s)$ of the operator.

We replace the matrix coefficient with a suitable element of an unramified principal series. Since $\pi$ is an irreducible quotient of $\Ind_{R}^{G}(\pi_n)$, the representation $\pi^{\vee}$ is a subrepresentation of $\Ind_{R}^{G}(\pi_n^{\vee})$, and we can further regard
$\pi_n^{\vee}$ as a subrepresentation of an unramified principal series representation of $\GL_n$. By transitivity of induction,
$\pi^{\vee}$ is embedded in an unramified principal series of $G$. Specifically, this is obtained by taking a function
${\phi^{\vee}}$ in the space of $\Ind_{R}^{G}(\pi_n^{\vee})$ and evaluating at the identity of $G$. Thus we can realize the $G$-invariant pairing on $\pi\times\pi^{\vee}$ using the Iwasawa decomposition $G=B_GK_G$.
Let $\phi^0$ and ${\phi^{\vee,0}}$ be the unramified vectors in the spaces of $\Ind_{R}^{G}(\pi_n)$ and $\Ind_{R}^{G}(\pi_n^{\vee})$, respectively, normalized by $\phi^0(I_{2n})={\phi^{\vee,0}}(I_{2n})=1$. Then
\begin{align*}
\omega^0(g)=\int\limits_{K_G}\phi^0(o){\phi^{\vee}}^0(og)\,do=\int\limits_{K_G}\phi^{\vee,0}(og)\,do.
\end{align*}

Observe that for any $g_0\in G$,
\begin{align*}
&\int\limits_{U_0}f_{W_c(\tau)}^0(\delta u_0(g_0,{}^{\iota}g_0)(1,{}^{\iota}g),s)\,\psi_U(u_0)\,du_0
\\&=\int\limits_{U_0}f_{W_c(\tau)}^0(\diag(g_0,\ldots,g_0,g_0^*,\ldots,g_0^*)\delta u_0(1,{}^{\iota}g),s)\,\psi_U(u_0)\,du_0
\end{align*}
(direct computation)
and by Proposition~\ref{proposition:k n functionals invariance}, for any $h\in H$,
\begin{align*}
f_{W_c(\tau)}^0(\diag(g_0,\ldots,g_0,g_0^*,\ldots,g_0^*)h,s)=f_{W_c(\tau)}^0(h,s).
\end{align*}
In addition, the embeddings of the two copies of $G$ in $H$ commute and $f_{W_c(\tau)}^0$ is right $K_H$-invariant, so that
for any $o\in K_G$,
\begin{align*}
f_{W_c(\tau)}^0(h(1,{}^{\iota}(o^{-1}g)),s)=
f_{W_c(\tau)}^0(h(1,{}^{\iota}(o^{-1}g))(o^{-1},1),s)
=f_{W_c(\tau)}^0(h(o^{-1},{}^{\iota}o^{-1})(1,{}^{\iota}g),s).
\end{align*}
Therefore
\begin{align}\notag
Z(s,\omega^0,f_{W_c(\tau)}^0)=&\int\limits_{G}
(\int\limits_{K_G}\phi^{\vee,0}(og)\,do)
\int\limits_{U_0}f_{W_c(\tau)}^0(\delta u_0(1,{}^{\iota}g),s)\,\psi_U(u_0)\,du_0\,dg\\\notag
=&\int\limits_{G}
\int\limits_{K_G}\phi^{\vee,0}(g)
\int\limits_{U_0}f_{W_c(\tau)}^0(\delta u_0(1,{}^{\iota}(o^{-1}g)),s)\,\psi_U(u_0)\,du_0\,do\,dg
\\\label{int:1}
=&\int\limits_{G}\phi^{\vee,0}(g)\int\limits_{U_0}f_{W_c(\tau)}^0(\delta u_0(1,{}^{\iota}g),s)\,\psi_U(u_0)\,du_0\,dg.
\end{align}
Note that the measure of $K_G$ was taken to be $1$.
Apply Lemma~\ref{lemma:decomposition for V functionals, n=a+b} to the function on $\GL_{kc}$ given by $x\mapsto f_{W_c(\tau)}(\diag(x,x^*)h,s)$ with $a=b=n$. Then with $V^3$ and $l_{n,n}$ as defined in \S~\ref{decomposition of functionals},
\begin{align*}
&f_{W_c(\tau)}^0(h,s)=\int\limits_{V^3}f_{W_n(\tau)\otimes W_n(\tau)}^0(l_{n,n}vh,s)\,dv.
\end{align*}
Using transitivity of induction and \eqref{eq:decomp a b space}, we see that $f_{W_n(\tau)\otimes W_n(\tau)}^0(h,s)$ belongs to the space of the representation
\begin{align}\label{rep:induced f before M(s)}
\Ind_{L}^{H}(|\det|^{-n/2+\alpha s}W_{n}(\tau)\otimes |\det|^{n/2+\alpha s}W_{n}(\tau)),
\end{align}
where $L$ is the standard parabolic subgroup of $H$ with a Levi part $M_L=\GL_{kn}\times\GL_{kn}$. It is an unramified function. In addition, $f_{W_n(\tau)\otimes W_n(\tau)}^0(I_{2kc},s)=1$ because by Lemma~\ref{lemma:decomposition for V functionals, n=a+b},
if we assume that $W_c(\tau)$ is realized by \eqref{eq:nk functional using w_{k,c}},
\begin{align*}
1=f_{W_c(\tau)}^0(I_{2kc},s)=\int\limits_{V^3}f_{W_n(\tau)\otimes W_n(\tau)}^0({}^{l_{n,n}}v,s)=f_{W_n(\tau)\otimes W_n(\tau)}^0(I_{2kc},s),
\end{align*}
where for the last equality see the proof of Corollary~\ref{corollary:k n functionals nonzero} (and recall that the volume of $\mathcal{O}$ is $1$).

With the above modifications, integral~\eqref{int:1} becomes
\begin{align}\label{int:2}
\int\limits_{G}
\phi^{\vee,0}(g)\int\limits_{U_0}\int\limits_{V^3}f_{W_n(\tau)\otimes W_n(\tau)}^0(l_{n,n}v\delta u_0(1,{}^{\iota}g),s)\psi_U(u_0)\,dv\,du_0\,dg.
\end{align}
This integral is absolutely convergent for $\Real(s)\gg0$ as a triple integral; this is obtained using the auxiliary complex parameters which guarantee the convergence of \eqref{eq:nk functional using w_{k,c}} (if $\zeta_1\gg\ldots\gg\zeta_k\gg0$, then
$\Real(s)\gg\zeta_1$, see e.g., \cite[Lemma~3.1]{Soudry2}, \cite[Claim~5.20]{me5}). All forthcoming manipulations are justified in this right half-plane.

Next we shift $v$ to the right of $(1,{}^{\iota}g)$.
Observe the following properties, which are immediate to verify.
\begin{enumerate}[leftmargin=*]
\item\label{it:GL observe 1} $\delta_0$ normalizes $V^3$.
\item\label{it:GL observe 2} For $v\in V^3$, ${}^v\delta_1=\delta_1u'$ where $u'\in U_0$ and $\psi_U(u')=1$.
\item\label{it:GL observe 3} The elements of $V^3$ normalize $U_0$ and fix $\psi_U|_{U_0}$.
\item\label{it:GL observe 4} $V^3$ commutes with $(1,{}^{\iota}g)$.
\item\label{it:GL observe 5} $\delta_0$ commutes with $l_{n,n}$.
\item\label{it:GL observe 6} $l_{n,n}$ commutes with $(1,{}^{\iota}g)$.
\end{enumerate}
($\delta_0$, $\delta_1$ were given in \S~\ref{Local factors for classical}.)
We also see that
\begin{align}\label{eq:U_0' classical}
U_0'={}^{l_{n,n}}U_0=\left\{\begin{pmatrix}I_{kn}&&U_1&U_2\\&I_{kn}&U_3&U_4\\&&I_{kn}\\&&&I_{kn}\end{pmatrix}\right\},
\end{align}
where $U_1=\left(\begin{smallmatrix}*&*\\0&*\end{smallmatrix}\right)$ with $0\in\Mat_{n}$, so that
\begin{align}\label{eq:U_1 classical}
\left\{\begin{pmatrix}I_{kn}&U_1\\&I_{kn}\end{pmatrix}\right\}
\end{align}
is the unipotent subgroup appearing in the integral for $\GL_n\times \GL_k$ (defined in \S~\ref{Local factors for GL}); for $\Sp_{2n}$ (resp., $\SO_{2n}$),
restriction of $\psi_U$ to the coordinates of $U_1$ gives the character $\psi_U$ (resp., $\psi_U^{-1}$) for the $\GL_n\times \GL_k$ integral; $\psi_U$ is trivial on $U_2$ and $U_3$;
$U_2$ and $U_3$ each takes the form $\left(\begin{smallmatrix}*&*\\0&*\end{smallmatrix}\right)$ where $0\in\Mat_n$; and $U_4$ is already determined by $U_1$ and the form defining $H$.

Utilizing properties \eqref{it:GL observe 1}--\eqref{it:GL observe 6}, integral~\eqref{int:2} equals
\begin{align}\label{int:3}
&\int\limits_{G}\phi^{\vee,0}(g)\int\limits_{U_0'}
\int\limits_{V^3}f_{W_n(\tau)\otimes W_n(\tau)}^0(\delta_0({}^{l_{n,n}}\delta_1) u_0'(1,{}^{\iota}g)l_{n,n}v,s)
\,\psi_U(u_0')\,dv\,du_0'\,dg.
\end{align}
Here $\psi_U$ is regarded as a character of $U_0'$ using conjugation (${}^{l_{n,n}^{-1}}u_0'\in U_0$).

To produce a $\GL_n\times\GL_k$ integral (pertaining to the statement of the lemma), we must alter $f_{W_n(\tau)\otimes W_n(\tau)}^0$ such that its restriction to $\GL_{2kn}\cong M_P$ becomes a section of a representation induced from
$W_n(\tau)\otimes W_n(\tau^{\vee})$. This would be the result of an application of an intertwining operator. If we had arbitrary coordinates in the bottom left $n\times n$ block of $U_2$, then the integral over $U_2$ together with the Weyl element $\diag(I_{kn},\left(\begin{smallmatrix}&I_{kn}\\\epsilon_0I_{kn}\end{smallmatrix}\right),I_{kn})$ (when $G=\Sp_{2n}$ or $kn$ is even) would constitute this operator (recall that $\epsilon_0=-1$ for $\Sp_{2n}$ and $\epsilon_0=1$ for $\SO_{2n}$). To fill in these missing coordinates we factor the integral through $U_R$. Refer to \S~\ref{Local factors for GL} for the definition of the $\GL_n\times\GL_k$ integral.

Let $U_0^{\bullet}$ be the group obtained from $U_0'$ by replacing the $0$ block of $U_2$ with arbitrary coordinates. This group will ``receive" the coordinates from $U_R$. We can still write the elements of $U_0^{\bullet}$ in the form \eqref{eq:U_0' classical}, i.e.,
\begin{align}\label{eq:U_0' classical2}
\left\{\begin{pmatrix}I_{kn}&&U_1&U_2\\&I_{kn}&U_3&U_4\\&&I_{kn}\\&&&I_{kn}\end{pmatrix}\right\},
\end{align}
the only difference being the block $U_2$, which now does not contain the $0$ block. Now factoring
\eqref{int:3} through $U_R$, it becomes
\begin{align}\label{int:3.5}
&\int\limits_{U_R\backslash G}\int\limits_{U_R}\phi^{\vee,0}(zg)\int\limits_{U_0'}
\int\limits_{V^3}f_{W_n(\tau)\otimes W_n(\tau)}^0(\delta_0({}^{l_{n,n}}\delta_1) u_0'(1,{}^{\iota}(zg))l_{n,n}v,s)
\,\psi_U(u_0')\,dv\,du_0'\,dz\,dg.
\end{align}
By definition $\phi^{\vee,0}(zg)=\phi^{\vee,0}(g)$.
For $z\in U_R$ and $u_0'\in U_0'$,
\begin{align*}
{}^{(1,{}^{\iota}z^{-1})}(({}^{l_{n,n}}\delta_1) u_0')=m_z({}^{l_{n,n}}\delta_1)u_z,
\end{align*}
where
$m_z$ belongs to the unipotent subgroup $V_{((2n)^k)}$ of $M_P$ and $u_z\in U_0^{\bullet}$. Moreover, as $z$ and $u_0'$ vary over $U_R$ and $U_0'$, $u_z$ varies over $U_0^{\bullet}$. In coordinates, put $z=\left(\begin{smallmatrix}I_n&z\\&I_n\end{smallmatrix}\right)$ and for $1\leq l\leq 4$ and $1\leq i,j\leq k$, denote the $(i,j)$-th block of $U_l$ appearing in \eqref{eq:U_0' classical2} by $U_l^{i,j}\in\Mat_n$. For $u_0'$, the block corresponding
to $U_l^{i,j}$ is denoted by $[u_0']_l^{i,j}$. Then
$m_z\in V_{((2k-1)n,n)}$, ${}^{\delta_0}m_z\in V_{(n,(2k-1)n)}$ and the top $n$ rows of ${}^{\delta_0}m_z$ are
\begin{align*}
\begin{pmatrix}I_n&\epsilon_0z[u_0']_3^{k,2}&\cdots&\epsilon_0z[u_0']_3^{k,k}&\epsilon_0z&\epsilon_0z[u_0']_4^{k,2}&\cdots&\epsilon_0z[u_0']_4^{k,k}\end{pmatrix}.
\end{align*}
We change variables in $u_z$ to remove the dependency on $z$. The change is described as follows: for $l\in\{1,3\}$, $1\leq i\leq k-1$ and $2\leq j\leq k$,
\begin{align*}
&[u_0']_l^{i,j}\mapsto [u_0']_l^{i,j}-\epsilon_0[u_0']_l^{i,1}z[u_0']_3^{k,j}, \qquad
[u_0']_1^{k,j}\mapsto [u_0']_1^{k,j}+z[u_0']_3^{k,j},\\
&[u_0']_2^{i,j}\mapsto [u_0']_2^{i,j}-\epsilon_0[u_0']_1^{i,1}z[u_0']_4^{k,j},\qquad
[u_0']_2^{i,1}\mapsto [u_0']_2^{i,1}-\epsilon_0[u_0']_1^{i,1}z.
\end{align*}
In this list, changes to
$[u_0']_2^{i,j}$ and $[u_0']_3^{i,j}$ are only applied if $i+j\leq k+1$ for $\Sp_{2n}$ and $i+j\leq k$ for $\SO_{2n}$, because outside of this range the coordinates are already determined by the definition of $H$.
Only the change to $[u_0']_1^{k,2}$ affects $\psi_U$, and we obtain $\psi(\mathrm{tr}(-\epsilon_0z[u_0']_3^{k,2}))$ (for $\SO_{2n}$ as mentioned after \eqref{eq:U_1 classical} $\psi_U$ restricts to the inverse of the character for the $\GL_n\times\GL_k$ integral, i.e., to $\psi^{-1}(\mathrm{tr}(U_1^{k,2}))$). In addition,
for any $h\in H$,
\begin{align*}
f_{W_n(\tau)\otimes W_n(\tau)}^0(m_zh,s)=\psi(\mathrm{tr}(\epsilon_0z[u_0']_3^{k,2}))f_{W_n(\tau)\otimes W_n(\tau)}^0(h,s)
\end{align*}
because of the character of the (top left) $(k,n)$-functional $W_n(\tau)$. Therefore \eqref{int:3.5} becomes
\begin{align}\label{int:4}
&\int\limits_{U_R\backslash G}\phi^{\vee,0}(g)\int\limits_{U_0^{\bullet}}
\int\limits_{V^3} f_{W_n(\tau)\otimes W_n(\tau)}^0(\delta_0({}^{l_{n,n}}\delta_1) u_0^{\bullet}(1,{}^{\iota}g)l_{n,n}v,s)\psi_U(u_0^{\bullet})\,dv\,du_0^{\bullet}\,dg.
\end{align}
Note that here the integration over $U_R$ is incorporated into the integration over $U_0^{\bullet}$.

Write $\delta_0=w\delta_0'w'$ as follows.
For $G=\Sp_{2n}$, $\delta_0'$ is the embedding in $M_P$ of the element
$\left(\begin{smallmatrix}&I_{kn}\\I_{kn}\end{smallmatrix}\right)$ corresponding to $\delta_0$ in the $\GL_n\times\GL_{k}$ integral, and
\begin{align*}
w=w'=&\begin{pmatrix}I_{kn}\\&&I_{kn}\\&-I_{kn}\\&&&I_{kn}\end{pmatrix}.
\end{align*}
If $G=\SO_{2n}$, set $\kappa=I_{4kn}$ if $kn$ is even, otherwise
$\kappa=\diag(I_{2kn-1},\left(\begin{smallmatrix}&1\\1\end{smallmatrix}\right),I_{2kn-1})$. Then $\delta_{0}'$ is the embedding
of $\left(\begin{smallmatrix}&I_{kn}\\I_{kn}\end{smallmatrix}\right)$ in $M_{{}^{\kappa}P}$, i.e., when $kn$ is odd, it is obtained from the embedding in $M_P$ by conjugation with $\kappa$, and
\begin{align*}
w=&\begin{pmatrix}I_{kn}\\&&I_{kn}\\&I_{kn}\\&&&I_{kn}\end{pmatrix}\kappa,\qquad
w'=\kappa\begin{pmatrix}I_{kn}\\&&I_{kn}\\&I_{kn}\\&&&I_{kn}\end{pmatrix}.
\end{align*}
(The element $\kappa$ is needed because when $kn$ is odd, we must have $\det{w}=\det{w'}=1$.) To make the notation uniform,
set $\kappa=I_{4kn}$ when $G=\Sp_{2n}$.

For $u_0^{\bullet}\in U_0^{\bullet}$, let $u^i$ denote the element obtained from $u_0^{\bullet}$ by zeroing out the coordinates in the blocks $U_j$ with $j\ne i$ (see \eqref{eq:U_0' classical2}). Write
\begin{align*}
\delta_0u_0^{\bullet}=w\cdot{}^{(\delta_{0}'w')}u^2\cdot\delta_0'\cdot {}^{w'}(u^1u^4)\cdot w'u^3.
\end{align*}
Since ${}^{l_{n,n}}\delta_1\in U_P$, it commutes with $u_0^{\bullet}$ and with $u^3$.
Also $\delta_1'={}^{w'}({}^{l_{n,n}}\delta_1)$ and $\delta'=\delta_0'\delta_1'$ are the embeddings in $M_{{}^{\kappa}P}$ of the elements corresponding to $\delta_1$ and $\delta$ for the $\GL_n\times\GL_k$ integral, except that for $G=\SO_{2n}$, $\delta_1'$ is actually the embedding of $\delta_1^{-1}$, and $\delta'$ is the embedding of
\begin{align*}
\left(\begin{smallmatrix}&I_{kn}\\I_{kn}\end{smallmatrix}\right)\diag(I_{(k-1)n},\left(\begin{smallmatrix}I_n&-I_n\\&I_n\end{smallmatrix}\right),I_{(k-1)n}).
\end{align*}
Then
\begin{align*}
\delta_0({}^{l_{n,n}}\delta_1)u_0^{\bullet}=w\cdot{}^{(\delta_{0}'w')}u^2\cdot\delta'\cdot {}^{w'}(u^1u^4)\cdot w'u^3.
\end{align*}

Denote the subgroup of elements $^{(\delta_{0}'w')}u^2$ by $U^2$,
let $U^{1,4}$ be the subgroup of elements ${}^{w'}(u^1u^4)$ and $U^3$ be the subgroup of elements $u^3$. For example,
\begin{align*}
U^2=\left\{\begin{pmatrix}I_{kn}\\&I_{kn}&Z\\&&I_{kn}\\&&&I_{kn}\end{pmatrix}\in H\right\}.
\end{align*}
Recall that $u^4$ is uniquely determined by $u^1$ and $H$. Then for any $h\in H$,
\begin{align*}
&\int\limits_{U_0^{\bullet}}f_{W_n(\tau)\otimes W_n(\tau)}^0(\delta_0({}^{l_{n,n}}\delta_1)u_0^{\bullet}h,s)\psi_U(u_0^{\bullet})\,du_0^{\bullet}\\&
=\int\limits_{U^3}\int\limits_{U^{1,4}}\int\limits_{U^2}f_{W_n(\tau)\otimes W_n(\tau)}^0(wu^2
\delta'uw'u^3h,s)\psi_U(u)\,du^2\,du\,du^{3}.
\end{align*}
Below we will show that the integration over $U^3$ evaluates to the constant $1$. The $du$-integral is the unipotent integration appearing in the $\GL_n\times\GL_k$ integral (defined in \S~\ref{Local factors for GL}), when we identify $\GL_{2kn}$ with $M_{{}^{\kappa}P}$.
The integration over $U^2$ defines an intertwining operator $M(s)$ from the space of \eqref{rep:induced f before M(s)} to
\begin{align*}
\Ind_{{}^{\kappa}L}^{H}(|\det|^{-n/2+\alpha s}W_{n}(\tau)\otimes |\det|^{-n/2-\alpha s}W_{n}(\tau^{\vee})).
\end{align*}
The image $M(s)f_{W_n(\tau)\otimes W_n(\tau)}^0$ of $M(s)$ on $f_{W_n(\tau)\otimes W_n(\tau)}^0$ is the normalized unramified vector multiplied by a constant which we denote $d_{\tau}(s)$, and we indeed prove below that it is equal to \eqref{eq:dtau}.
Again, identify $\GL_{2kn}$ with $M_{{}^{\kappa}P}$.
When we restrict $M(s)f_{W_n(\tau)\otimes W_n(\tau)}^0$ to $\GL_{2kn}$ we obtain a rational section
of
\begin{align}\label{rep:M(s)f}
|\det|^{(\alpha-n)/2}\Ind_{P_{(kn,kn)}}^{\GL_{2kn}}((W_n(\tau)\otimes W_n(\tau^{\vee}))\delta_{P_{(kn,kn)}}^{\ell s}),\qquad\ell=\alpha/(kn).
\end{align}
Let $\rho_{W_n(\tau)\otimes W_n(\tau^{\vee})}^0$ be the normalized unramified vector in the space of
\begin{align*}
\Ind_{P_{(kn,kn)}}^{\GL_{2kn}}(W_n(\tau)\otimes W_n(\tau^{\vee})).
\end{align*}
Then for any $h\in \GL_{2kn}$,
\begin{align}\label{eq:M(s) of f with Wn}
M(s)f_{W_n(\tau)\otimes W_n(\tau)}^0(h,s)=|\det h|^{(\alpha-n)/2}d_{\tau}(s)\rho_{W_n(\tau)\otimes W_n(\tau^{\vee})}^0(h,\ell s),
\end{align}
where $h\mapsto\rho_{W_n(\tau)\otimes W_n(\tau^{\vee})}^0(h,\ell s)$ is the standard section of
\begin{align*}
\Ind_{P_{(kn,kn)}}^{\GL_{2kn}}((W_n(\tau)\otimes W_n(\tau^{\vee}))\delta_{P_{(kn,kn)}}^{\ell s})
\end{align*}
corresponding to $\rho_{W_n(\tau)\otimes W_n(\tau^{\vee})}^0$. Now \eqref{int:4} takes the form
\begin{align}\label{int:4.5}
&\int\limits_{U_R\backslash G}\phi^{\vee,0}(g)
\int\limits_{V^3}
\int\limits_{U^3}\int\limits_{U^{1,4}}M(s)f_{W_n(\tau)\otimes W_n(\tau)}^0(
\delta'uw'u^3(1,{}^{\iota}g)l_{n,n}v,s)\psi_U(u)\,du\,du^{3}\,dv\,dg.
\end{align}

Let $g=\diag(g',{g'}^*)\in M_R$, $g'\in\GL_n$. Then
\begin{align*}
{}^{w'}(1,{}^{\iota}g)=\diag(I_{kn},g',I_{(k-1)n})\in\GL_{2kn}
\end{align*}
is the embedding $(I_n,\GL_n)$ in the construction of the $\GL_n\times\GL_k$ integral.
Apply the Iwasawa decomposition $G=RK_G$. The change of measure $\delta_R^{-1}(g)$ incurred by this decomposition, the conjugation of $U^3$ by $(1,{}^{\iota}g)$ and the additional $\delta_R^{1/2}(g)$ emitted from $\phi^{\vee,0}$, multiply the integrand by $|\det g' |^{(n-\alpha)/2}$
(which will cancel out with the power of $|\det|$ from \eqref{rep:M(s)f}). Also note that $\phi^{\vee,0}(g)=\delta_R^{1/2}(g)\phi_n^{\vee,0}(g')$, where $\phi_n^{\vee,0}$ is the normalized unramified vector in the space of $\pi_n^{\vee}$. Then \eqref{int:4.5} equals
\begin{align}\label{int:5}
\int\limits_{V^3}\int\limits_{U^3}
\int\limits_{\GL_n}
\int\limits_{U^{1,4}}
|\det g' |^{(n-\alpha)/2}\phi_n^{\vee,0}(g')
 M(s)f_{W_n(\tau)\otimes W_n(\tau)}^0(\delta'u(1,g')w'u^3l_{n,n}v,s)\psi_U(u)\,du\,dg'\,du^3\,dv.
\end{align}

Let $\phi_n^0$ be the normalized unramified vector in the space of $\pi_n$.
Since for $g'\in\GL_n$,
\begin{align*}
\omega_n^0(g')=\int\limits_{K_{\GL_n}}\phi_n^0(o){\phi_n^{\vee,0}}(og')\,do=\int\limits_{K_{\GL_n}}\phi_n^{\vee,0}(og')\,do,
\end{align*}
as in the beginning of this section we can replace $\phi_n^{\vee,0}(g')$ with $\omega_n^0(g')$ (see \S~\ref{proof of lemma:reduction from GLn to GLa GLb} for more details). Then \eqref{int:5} becomes
\begin{align}\notag
&\int\limits_{V^3}\int\limits_{U^3}\int\limits_{\GL_n}
\int\limits_{U^{1,4}}
|\det g' |^{(n-\alpha)/2}\omega_n^{0}(g')
 M(s)f_{W_n(\tau)\otimes W_n(\tau)}^0(\delta'u(1,g')w'u^3l_{n,n}v,s)\psi_U(u)\,du\,dg'\,du^3\,dv\\
 &=
\int\limits_{V^3}\int\limits_{U^3}Z'(|\det |^{(n-\alpha)/2}\omega_n^{0},(w'u^3l_{n,n}v)\cdot
M(s)f_{W_n(\tau)\otimes W_n(\tau)}^0)\,du^3\,dv,\label{int:6}
\end{align}
where $Z'(\cdots)$ is the $\GL_n\times\GL_{k}$ integral, with the exception that
for $\SO_{2n}$, $\delta_1'$ ($\delta_1'$ is the unipotent part of $\delta'$) and $\psi_U$ appearing in $Z'$ are the inverses of those defined in \S~\ref{Local factors for GL}. The value of the $\GL_n\times\GL_{k}$ integral with unramified data is invariant with respect to this change. To see this, replace the section on $\GL_{2kn}$ with its right translate by $\diag(-I_{kn},I_{kn})$ (this matrix commutes with the embedding $(\GL_n,\GL_n)<\GL_{2kn}$).

We will show that the $du^3$-integral in \eqref{int:6} vanishes unless $v\in K_H$ and $u^3\in K_H$. Since $M(s)f_{W_n(\tau)\otimes W_n(\tau)}^0$ is unramified, for any $v,u^3\in K_H$ and $h\in H$ we have
\begin{align*}
M(s)f_{W_n(\tau)\otimes W_n(\tau)}^0(h(w'u^3l_{n,n}v),s)=M(s)f_{W_n(\tau)\otimes W_n(\tau)}^0(h,s),
\end{align*}
thus \eqref{int:6} becomes
\begin{align*}
&Z'(|\det|^{(n-\alpha)/2}\omega_n^{0},
M(s)f_{W_n(\tau)\otimes W_n(\tau)}^0)\int\limits_{V^3}\int\limits_{U^3}1\,du^3\,dv
=1\times Z'(|\det|^{(n-\alpha)/2}\omega_n^{0},M(s)f_{W_n(\tau)\otimes W_n(\tau)}^0),
\end{align*}
because the measure of $\mathcal{O}$ is $1$.
Finally using \eqref{eq:M(s) of f with Wn},
\begin{align*}
Z'(|\det|^{(n-\alpha)/2}\omega_n^{0},M(s)f_{W_n(\tau)\otimes W_n(\tau)}^0)=d_{\tau}(s)Z(\ell s,\omega_n^0,\rho_{W_n(\tau)\otimes W_n(\tau^{\vee})}^0),
\end{align*}
which is the integral appearing on the right hand side of \eqref{eq:lemma decomp classical}. This will complete the proof of the lemma,
once we handle the integrals over $V^3$ and $U^3$ in \eqref{int:6} and compute $d_{\tau}(s)$.

Conjugate the elements $w'$ and $l_{n,n}$ to the right. They disappear because $M(s)f_{W_n(\tau)\otimes W_n(\tau)}^0$ is $K_H$-invariant on the right. In coordinates, for $G=\Sp_{2n}$,
\begin{align*}
&{}^{w'}(U^3\cdot{}^{l_{n,n}}V^3)=\left\{\begin{pmatrix}I_{kn}\\&I_{kn}\\V&U^3&I_{kn}\\0& V' &&I_{kn}\end{pmatrix}\in H\right\},\qquad
V=\begin{pmatrix}0&V_{1,2}^3&\cdots&V_{1,k}^3\\\vdots&&\ddots&\vdots\\\vdots&&&V_{k-1,k}^3\\0&\cdots&\cdots&0\end{pmatrix},\qquad V_{i,j}^3\in\Mat_{n},
\end{align*}
where $V'$ is uniquely defined given $V$ and $H$.

To show that the coordinates of $V_{i,j}^3$ must belong to $\mathcal{O}$, otherwise the $du^3$-integral vanishes, consider matrices
\begin{align*}
x=\begin{pmatrix}
I_{kn}&&[x]
\\&I_{kn}&&[x]'
\\&&I_{kn}\\&&&I_{kn}
\end{pmatrix}\in H,\qquad
[x]=\begin{pmatrix}
x_{1,2}&0&\cdots&0\\
\vdots&\ddots&\ddots&\vdots\\
x_{1,k}&\cdots&x_{k-1,k}&0\\
0&\cdots&\cdots&0
\end{pmatrix},\qquad x_{i,j}\in\Mat_n.
\end{align*}
For each $1\leq i\leq k-1$ and $2\leq j\leq k$, let $\mathcal{X}_{i,j}$ be the subgroup of these matrices $x$ where the only nonzero block in $[x]$ is $x_{i,j}$, which takes arbitrary coordinates in $\mathcal{O}$. Then $\mathcal{X}_{i,j}<K_H$. We handle $V_{i,j}^3$ using
$\mathcal{X}_{i,j}$. Starting with $V_{1,2}^3$, we proceed along the diagonal $(l,l+1)$ with $l=2,\ldots,k-1$ in increasing order,
then the diagonal $(l,l+2)$, $l=1,\ldots, k-2$, etc., the last block of $V$ to handle being $V_{1,k}^3$, for which we
use $\mathcal{X}_{1,k}$.

Consider $z\in{}^{w'}(U^3\cdot{}^{l_{n,n}}V^3)$. Let $v_{i,j}^3$ be the block of $z$ corresponding to $V_{i,j}^3$. If $v_{i,j}^3\in\Mat_n(\mathcal{O})$, we can assume $v_{i,j}^3=0$ since $M(s)f_{W_n(\tau)\otimes W_n(\tau)}^0$ is $K_H$-invariant on the right. To show $v_{i,j}^3\in\Mat_n(\mathcal{O})$, assuming we have already shown this for the previous blocks in the order along the diagonals, note that for $x\in \mathcal{X}_{i,j}$, ${}^{x^{-1}}z=u_xz_x$ with $u_x\in P$ and
$z_x\in{}^{w'}(U^3\cdot{}^{l_{n,n}}V^3)$. The projection of $u_x$ to $M_P$ belongs to the unipotent group $U$ of the $\GL_n\times\GL_k$ integral. The invariance properties of this integral (see \eqref{eq:homspace GLn GLk}) imply that
\begin{align*}
&Z'(|\det |^{(n-\alpha)/2}\omega_n^{0},(zx)\cdot
M(s)f_{W_n(\tau)\otimes W_n(\tau)}^0)
\\&=\psi(\mathrm{tr}(v_{i,j}^3x_{i,j}))Z'(|\det |^{(n-\alpha)/2}\omega_n^{0},z_x\cdot
M(s)f_{W_n(\tau)\otimes W_n(\tau)}^0).
\end{align*}
Regarding $z_x$, the coordinates depending on $x$ belong to the blocks of $U^3$ and this dependence can be removed by a change of variables. Therefore if we consider the integral $du^3$ over the coordinates of $U^3$ appearing in $z$,
\begin{align*}
&\int\limits_{U^3}Z'(|\det |^{(n-\alpha)/2}\omega_n^{0},z\cdot
M(s)f_{W_n(\tau)\otimes W_n(\tau)}^0)\,du^3\\&=
\int\limits_{U^3}Z'(|\det |^{(n-\alpha)/2}\omega_n^{0},(zx)\cdot
M(s)f_{W_n(\tau)\otimes W_n(\tau)}^0)\,du^3\\&=
\psi(\mathrm{tr}(v_{i,j}^3x_{i,j}))\int\limits_{U^3}Z'(|\det |^{(n-\alpha)/2}\omega_n^{0},z\cdot
M(s)f_{W_n(\tau)\otimes W_n(\tau)}^0)\,du^3.
\end{align*}
Hence the integral $du^3$ is zero unless $v_{i,j}^3\in\Mat_{n}(\mathcal{O})$ (cf. the proof of Corollary~\ref{corollary:k n functionals nonzero}), and we can proceed to the next block.

Next we handle the coordinates of $U^3$. Let $U_{i,j}$ denote the $(i,j)$-th $n\times n$ block of $U^3$ ($1\leq i,j\leq k$) and note that
$U_{k,1}$ is $0$, because this is the bottom left block of $U^3$ (see after \eqref{eq:U_1 classical}). We show that the coordinates of $U_{i,j}$ can be taken in $\mathcal{O}$. Consider
\begin{align*}
x=\begin{pmatrix}
I_{(k-1)n}&&\\
&I_{n}&&[x]\\
&&I_{kn}&[y]&[x]'\\
&&&I_{kn}&&\\
&&&&I_{n}\\
&&&&&I_{(k-1)n}
\end{pmatrix}\in H,\qquad
\begin{matrix}
[x]=\begin{pmatrix}x_{1,2}&\cdots& x_{k,2}\end{pmatrix},\\ \\
[y]=\begin{pmatrix}
0&x_{k,k}'&\cdots&x_{k,3}'&0\\
x_{1,3}&x_{2,3}&\cdots&x_{k-1,3}&x_{k,3}\\
\vdots&&&&\vdots\\
x_{1,k}&x_{2,k}&\cdots&x_{k-1,k}&x_{k,k}\\
0&x_{1,k}'&\cdots&x_{1,3}'&0
\end{pmatrix},
\end{matrix}
\end{align*}
where $x_{i,j}\in\Mat_n$. Note that while $[x]$ can take arbitrary coordinates, in the notation for $[y]$ coordinates are dependent, since $x\in H$. Let $\mathcal{X}_{i,j}$ be the subgroup of matrices $x$ such that $x_{i,j}\in\Mat_n(\mathcal{O})$ (if $j>2$, $x_{i,j}$ also depends on $H$) and all other blocks which are independent of $x_{i,j}$ are $0$. We handle $U_{i,j}$ using $\mathcal{X}_{i,j}$ and as above, the order matters.

Let $z\in{}^{w'}U^3$. Denote the coordinates in $z$ corresponding to the blocks $U_{i,j}$ by $u_{i,j}$.
For any $i$ and $j$, if $u_{i,j}\in\Mat_n(\mathcal{O})$, we can assume $u_{i,j}=0$ because the section is right $K_H$-invariant.
Let $1\leq i\leq k$. If $u_{l_1,1}=u_{l_2,2}=0$ for all $l_1\geq i$ and $l_2>i$, then for $x\in\mathcal{X}_{i,2}$ we have  ${}^{x^{-1}}z=u_xz$, where $u_x\in M_P$, $u_x$ belongs to the unipotent subgroup $U_0$ appearing in the $\GL_n\times\GL_k$ integral and
\begin{align*}
&Z'(|\det |^{(n-\alpha)/2}\omega_n^{0},z\cdot
M(s)f_{W_n(\tau)\otimes W_n(\tau)}^0)
\\&=Z'(|\det |^{(n-\alpha)/2}\omega_n^{0},(zx)\cdot
M(s)f_{W_n(\tau)\otimes W_n(\tau)}^0)
\\&=\psi(\mathrm{tr}(x_{i,2}u_{i,2}))Z'(|\det |^{(n-\alpha)/2}\omega_n^{0},z\cdot
M(s)f_{W_n(\tau)\otimes W_n(\tau)}^0).
\end{align*}
Thus the integrand vanishes unless $u_{i,2}\in\Mat_n(\mathcal{O})$.
We start with $U_{k,2}$ using $\mathcal{X}_{k,2}$ and deduce
$U_{k,2}\subset\Mat_n(\mathcal{O})$, which implies $U_{k-1,1}\subset\Mat_{n}(\mathcal{O})$. Then $U_{k-1,2}\subset\Mat_n(\mathcal{O})$, using $\mathcal{X}_{k-1,2}$.

To handle $U_{k,3}$ a change of variables is needed. For $x\in\mathcal{X}_{k,3}$, ${}^{x^{-1}}z=u_xz_x$, where $z_x$ belongs to ${}^{w'}U^3$ but depends on $x$. However, we can change variables in the blocks $U_{i,j}$ with $i\leq k-2$ and $j\geq3$ (in particular, blocks which have not been handled) to remove this dependency. The element $u_x$ belongs to $P$ and its projection to $M_P$ is in the subgroup $U$ of the $\GL_n\times\GL_k$ integral. It follows that
\begin{align*}
&\int\limits_{U^3}Z'(|\det |^{(n-\alpha)/2}\omega_n^{0},z\cdot
M(s)f_{W_n(\tau)\otimes W_n(\tau)}^0)\,du^3\\&=
\int\limits_{U^3}\int\limits_{\mathcal{X}_{k,3}}Z'(|\det |^{(n-\alpha)/2}\omega_n^{0},(zx)\cdot
M(s)f_{W_n(\tau)\otimes W_n(\tau)}^0)\,dx\,du^3\\&=
\int\limits_{U^3}Z'(|\det |^{(n-\alpha)/2}\omega_n^{0},z\cdot
M(s)f_{W_n(\tau)\otimes W_n(\tau)}^0)\,du^3\int\limits_{\mathcal{X}_{k,3}}\psi(\mathrm{tr}(x_{k,3}u_{k,3}))\,dx,
\end{align*}
which equals zero unless $u_{k,3}\in\Mat_n(\mathcal{O})$ (the measure of $\mathcal{X}_{k,3}$ was taken to be $1$). Thus
$U_{k,3}\subset\Mat_n(\mathcal{O})$ and then also $U_{k-2,1}\subset\Mat_n(\mathcal{O})$, so that we can proceed with $U_{k-2,2}$.
In general, once $U_{i,2}\subset\Mat_n(\mathcal{O})$, we handle the diagonal
$U_{i+l,2+l}$, $l=k-i,\ldots,1$. Then in particular $U_{i-1,1}\subset\Mat_n(\mathcal{O})$ so that we can continue with
$U_{i-1,2}$. Once $U_{1,2}\subset\Mat_n(\mathcal{O})$ (thereby $U_{k-1,k}\subset\Mat_n(\mathcal{O})$), we proceed with the remaining blocks $U_{i,j}$ on the diagonals, from bottom to top: the first diagonal is $U_{l,k},U_{l-1,k-1},\ldots, U_{1,k-l+1}$ with $l=k-1$, the next diagonal
$U_{l,k},\ldots, U_{1,k-l+1}$ with $l=k-2$, etc. The last block is $U_{1,k}$. In this way, the changes of variables are always to blocks which have not been considered.

The case of $\SO_{2n}$ is similar: $V^3$ is the same, there are fewer coordinates in $U^3$.

It remains to compute $d_{\tau}(s)$. Consider the standard Levi subgroup of $H$ isomorphic to $\GL_{kn}\times H'$, where $H'$ is a classical group of the same type as $H$ but with rank $kn$. Regard $H'$ as a subgroup of $H$ via this isomorphism. Fix the Borel subgroup $B_{H'}=B_H\cap H'$.
Looking at \eqref{rep:induced f before M(s)} we see that the restriction of $f_{W_n(\tau)\otimes W_n(\tau)}^0$ to ${H'}$ is the normalized unramified element, in the space of the unramified principal series representation of ${H'}$ induced (normalized induction) from
\begin{align*}
\otimes_{1\leq i\leq k,1\leq j\leq n}\chi_i|~|^{\alpha s+j-1/2}.
\end{align*}
Let $\Delta(\tau,n,s)$ be the irreducible unramified constituent of this representation.

Assume $H'=\Sp_{2kn}$. The subgroup $U^2$ is the unipotent radical of a standard parabolic subgroup of $H'$ whose Levi part is $\GL_{kn}$.
The adjoint action of $\GL_{kn}(\C)$ on the Lie algebra of the $L$-group of $U^2$ is $\text{st}\oplus\wedge^2$,
where $\text{st}$ is the standard representation.
By Langlands' theory \cite{La2} and the Gindikin--Karpelevich formula \cite[Theorem~3.1]{CS1},
\begin{align*}
d_{\tau}(s)=\frac{L(0,\Delta(\tau,n,s),\text{st})}{L(1,\Delta(\tau,n,s),\text{st})}\frac{
L(0,\Delta(\tau,n,s),\wedge^2)}{L(1,\Delta(\tau,n,s),\wedge^2)}.
\end{align*}
The first quotient (for the standard representation) contributes
\begin{align}\label{eq:constant term st}
\prod_{1\leq i\leq k}\prod_{1\leq j\leq n}\frac{L(\alpha s +j-1/2,\chi_i)}{L(\alpha s
+j+1/2,\chi_i)}
=\frac{L(\alpha s +1/2,\tau)}{L(\alpha s +n+1/2,\tau)}.
\end{align}
The second quotient (the exterior square) contributes, for each pair $1\leq i\ne i'\leq k$,
\begin{align*}
\prod_{1\leq j,j'\leq n}\frac{L(2\alpha s +j+j'-1,\chi_i\chi_{i'})}{L(2\alpha s +j+j',\chi_i\chi_{i'})}=
\prod_{1\leq j\leq n}\frac{L(2\alpha s +j,\chi_i\chi_{i'})}{L(2\alpha s +j+n,\chi_i\chi_{i'})},
\end{align*}
and for $1\leq i\leq k$,
\begin{align*}
&\prod_{1\leq j_1<n}\,
\prod_{j_1<j_2\leq n}\frac{L(2\alpha s +j_1+j_2-1,\chi_i^2)}{L(2\alpha s +j_1+j_2,\chi_i^2)}
=\prod_{1\leq j<n}
\frac{L(2\alpha s +2j,\chi_i^2)}{L(2\alpha s +j+n,\chi_i^2)}.
\end{align*}
(This product is empty for $n=1$, then it simply equals $1$.)
Thus when $n$ is odd we obtain
\begin{align*}
\prod_{1\leq j\leq (n-1)/2}
\frac{L(2\alpha s +2j,\tau,\mathrm{Sym}^2)}{L(2\alpha s +2j+n,\tau,\mathrm{Sym}^2)}
\prod_{1\leq j\leq (n+1)/2}
\frac{L(2\alpha s +2j-1,\tau,\wedge^2)}{L(2\alpha s +2j+n-1,\tau,\wedge^2)},
\end{align*}
and for even $n$,
\begin{align*}
\prod_{1\leq j\leq n/2}
\frac{L(2\alpha s +2j,\tau,\mathrm{Sym}^2)}{L(2\alpha s +2j+n-1,\tau,\mathrm{Sym}^2)}
\prod_{1\leq j\leq n/2}
\frac{L(2\alpha s +2j-1,\tau,\wedge^2)}{L(2\alpha s +2j+n,\tau,\wedge^2)}.
\end{align*}
The constant $d_{\tau}(s)$ is this product multiplied by \eqref{eq:constant term st}.
For $H'=\SO_{2kn}$ the adjoint action of $\GL_{kn}(\C)$ on the Lie algebra of the dual group of $U^2$ is $\wedge^2$.
The only change in the computation above is we omit \eqref{eq:constant term st}.
The proof of the lemma is complete.

\subsection{Local factors for $\GL_n$}\label{Local factors for GLn}
In this section we prove Theorem~\ref{theorem:unramified computation for GL(n)}. We proceed with the set-up from
\S~\ref{prelim unr}. Let $G=\GL_n$ and $\pi$ be an irreducible unramified representation of $G$. Let $\tau$ and $\tau'$
be unramified twists of irreducible unitary generic unramified representations of $\GL_k$, with the additional assumption on the central characters of $\tau$ and $\tau'$, namely $\tau\tau'(aI_k)=1$ for all $a\in F^*$. For the definition of the $\GL_n\times\GL_k$ integral see \S~\ref{Local factors for GL}, but we recall that $H=\GL_{2kn}$, $P=P_{(kn,kn)}$ and the sections belong to $\Ind_{P}^{H}((W_n(\tau)\otimes W_n(\tau'))\delta_P^{s})$.

Let $\omega^0$ and $f_{W_n(\tau)\otimes W_n(\tau')}^0$ be the normalized unramified elements given by the theorem.
We reduce the $G\times \GL_k$ integral to the
case $n=1$, which is computed directly. Put $\alpha=kn$ and for any positive integers $a$ and $b$ such that $a+b=n$,
\begin{align*}
&d_{\tau,\tau',a,b}(s)=\prod_{1\leq j\leq b}\frac{L(2\alpha s+j,\tau\times{\tau'}^{\vee})}{L(2\alpha s+a+j,\tau\times{\tau'}^{\vee})}.
\end{align*}
\begin{lemma}\label{lemma:reduction from GLn to GLa GLb}
For $a$ and $b$ as above, write $\pi$ as a quotient of $\Ind_{R}^{G}(\pi_a\otimes\pi_b)$, where $R=P_{(a,b)}$ and $\pi_a$ and $\pi_b$ are irreducible unramified representations of $\GL_a$ and $\GL_b$. Let
$\omega_a^0$ and $\omega_b^0$ be the normalized unramified matrix coefficients of $\pi_a^{\vee}$ and $\pi_b^{\vee}$, $\rho_{W_a(\tau)\otimes W_a(\tau')}^0$ be the normalized unramified function in the space of
\begin{align*}
\Ind_{P_{(ka,ka)}}^{\GL_{2ka}}(W_a(\tau)\otimes W_a(\tau')),
\end{align*}
and $\varrho_{W_b(\tau)\otimes W_b(\tau')}^0$ be the normalized unramified function in the space of
\begin{align*}
\Ind_{P_{(kb,kb)}}^{\GL_{2kb}}(W_b(\tau)\otimes W_b(\tau')).
\end{align*}
Then
\begin{align*}
Z(s,\omega^0,f_{W_n(\tau)\otimes W_n(\tau')}^0)=d_{\tau,\tau',a,b}(s)Z(\alpha s/(ka),\omega_{a}^0,\rho_{W_a(\tau)\otimes W_a(\tau')}^0)Z(\alpha s/(kb),\omega_{b}^0,\varrho_{W_b(\tau)\otimes W_b(\tau')}^0).
\end{align*}
\end{lemma}
\begin{proposition}\label{proposition:unramified computation for GL(1)}
For $n=1$ and when all data are unramified,
\begin{align*}
Z(s,\omega^0,f_{W_1(\tau)\otimes W_1(\tau')}^0)=\frac{L(ks+1/2,\pi^{-1}\times\tau)L(ks+1/2,\pi\times{\tau'}^{\vee})}
{L(2ks+1,\tau\times{\tau'}^{\vee})}.
\end{align*}
\end{proposition}
The lemma is proved in \S~\ref{proof of lemma:reduction from GLn to GLa GLb} and the proposition in \S~\ref{final reduction n = 1 linear groups}.
Now we can compute the integral inductively, for all $n$.
\begin{proof}[Proof of Theorem~\ref{theorem:unramified computation for GL(n)}]
We argue using induction on $n$. We have to show that the integral with unramified data equals
\begin{align*}
\frac{L(kn s+1/2,\pi^{\vee}\times\tau)L(kn s+1/2,\pi\times{\tau'}^{\vee})}{\prod_{j=1}^{n}L(2kn s+j,\tau\times{\tau'}^{\vee})}.
\end{align*}
The result holds for $n=1$ by Proposition~\ref{proposition:unramified computation for GL(1)}.
Consider a $\GL_n\times \GL_{k}$ integral.
Assume
the formula is true for $n-1$ and apply Lemma~\ref{lemma:reduction from GLn to GLa GLb} to the integral with $a=1$ and $b=n-1$. The integral becomes the product of $d_{\tau,\tau',1,n-1}(s)$,
the $\GL_1\times \GL_{k}$ integral and the $\GL_{n-1}\times \GL_{k}$ integral. Using the case $n=1$,
\begin{align*}
Z(ns,\omega_{1}^0,\rho_{W_1(\tau)\otimes W_1(\tau')}^0)=\frac{L(kns+1/2,\pi_1^{-1}\times\tau)L(kns+1/2,\pi_1\times{\tau'}^{\vee})}{
L(2kns+1,\tau\times{\tau'}^{\vee})}.
\end{align*}
Applying the induction hypothesis to the $\GL_{n-1}\times \GL_{k}$ integral,
\begin{align*}
Z(ns/(n-1),\omega_{n-1}^0,\varrho_{W_{n-1}(\tau)\otimes W_{n-1}(\tau')}^0)=\frac{L(kn s+1/2,\pi_{n-1}^{\vee}\times\tau)L(kns+1/2,\pi_{n-1}\times{\tau'}^{\vee})}{\prod_{j=1}^{n-1}L(2kns+j,\tau\times{\tau'}^{\vee})}.
\end{align*}
Together we obtain
\begin{align*}
\frac{L(kn s+1/2,\pi^{\vee}\times\tau)L(kn s+1/2,\pi\times{\tau'}^{\vee})}
{L(2kn s+1,\tau\times{\tau'}^{\vee})^2\prod_{j=2}^{n-1}L(2kn s+j,\tau\times{\tau'}^{\vee})},
\end{align*}
and multiplying this by
\begin{align*}
d_{\tau,\tau',1,n-1}(s)=\prod_{1\leq j\leq n-1}\frac{L(2kn s+j,\tau\times{\tau'}^{\vee})}{L(2kn s+1+j,\tau\times{\tau'}^{\vee})}
=\frac{L(2kn s+1,\tau\times{\tau'}^{\vee})}{L(2kns+n,\tau\times{\tau'}^{\vee})}
\end{align*}
gives the result.
\end{proof}
We turn to the proof of the reduction lemma.
\subsubsection{Proof of Lemma~\ref{lemma:reduction from GLn to GLa GLb}}\label{proof of lemma:reduction from GLn to GLa GLb}
The proof is a straightforward modification of the proof of Lemma~\ref{lemma:reduction from classical to GLn}, but manipulations applied to the $(k,c)$ functional are now doubled, because we work with both $W_n(\tau)$ and $W_n(\tau')$. We focus on the differences between the proofs, and when possible, use similar notation. Also recall that the definitions of $U_0$, $\delta$, $\delta_0$, $\delta_1$ and $\psi_U$ were given in
\S~\ref{Local factors for GL}.

We replace $\omega^0$ with the normalized unramified vector $\phi^{\vee,0}$ in the space of $\Ind_{R}^{G}(\pi^{\vee})$. As in the proof of Lemma~\ref{lemma:reduction from classical to GLn}, we write
\begin{align*}
\omega^0(g)=\int\limits_{K_G}\phi^{\vee,0}(og)\,do.
\end{align*}
Then for any $g_0\in G$,
\begin{align*}
&\int\limits_{U_0}f_{W_n(\tau)\otimes W_n(\tau')}^0(\delta u_0(g_0,g_0)(1,g),s)\,\psi_U(u_0)\,du_0
\\&=\int\limits_{U_0}f_{W_n(\tau)\otimes W_n(\tau')}^0(\diag(g_0,\ldots,g_0)\delta u_0(1,g),s)\,\psi_U(u_0)\,du_0
\\&=\tau(\det g_0I_k)\tau'(\det g_0I_k)\int\limits_{U_0}f_{W_n(\tau)\otimes W_n(\tau')}^0(\delta u_0(1,g),s)\,\psi_U(u_0)\,du_0.
\end{align*}
Here the second equality follows from Proposition~\ref{proposition:k n functionals invariance}. Our condition on the central characters of $\tau$ and $\tau'$ implies $\tau(\det g_0I_k)\tau'(\det g_0I_k)=1$. Combining this with the right $K_H$-invariance of
$f_{W_n(\tau)\otimes W_n(\tau')}^0$, we deduce
\begin{align}\notag
Z(s,\omega^0,f_{W_n(\tau)\otimes W_n(\tau')}^0)=&\int\limits_{G}
(\int\limits_{K_G}\phi^{\vee,0}(og)\,do)
\int\limits_{U_0}f_{W_n(\tau)\otimes W_n(\tau')}^0(\delta u_0(1,g),s)\,\psi_U(u_0)\,du_0\,dg
\\\label{int:GL start}
=&\int\limits_{G}\phi^{\vee,0}(g)\int\limits_{U_0}f_{W_n(\tau)\otimes W_n(\tau')}^0(\delta u_0(1,g),s)\,\psi_U(u_0)\,du_0\,dg.
\end{align}
(Cf. \eqref{int:1}.)

Let $a,b\geq1$ be given by the statement of the lemma ($a$ and $b$ need not be equal). Apply
Lemma~\ref{lemma:decomposition for V functionals, n=a+b} twice, to the functions on $\GL_{kn}$ given by
\begin{align*}
x\mapsto f_{W_n(\tau)\otimes W_n(\tau')}^0(\diag(x,I_{kn})h,s),\qquad
y\mapsto f_{W_n(\tau)\otimes W_n(\tau')}^0(\diag(I_{kn},y)h,s),
\end{align*}
where $h\in H$ is fixed.
In the notation of that lemma,
\begin{align}\label{eq:section to plug in GL}
&f_{W_n(\tau)\otimes W_n(\tau')}^0(h,s)=\int\limits_{V^3}\int\limits_{V^3}f_{(W_a(\tau)\otimes W_b(\tau))\otimes(W_a(\tau')\otimes W_b(\tau'))}^0(\diag(l_{a,b},l_{a,b})\diag(v,v')h,s)\,dv\,dv'.
\end{align}
This is a section in the space of the representation
\begin{align}\label{rep:induced f before M(s) GL}
\Ind_{L}^{H}(|\det|^{-b/2+\alpha s}W_{a}(\tau)\otimes |\det|^{a/2+\alpha s}W_{b}(\tau)\otimes |\det|^{-b/2-\alpha s}W_{a}(\tau')\otimes |\det|^{a/2-\alpha s}W_{b}(\tau')),
\end{align}
where $L=P_{(ka,kb,ka,kb)}$.

Substituting \eqref{eq:section to plug in GL} into \eqref{int:GL start} one obtains
\begin{align}\label{int:GL start 2}
\int\limits_{G}
\phi^{\vee,0}(g)\int\limits_{U_0}
\int\limits_{V^3}\int\limits_{V^3}f_{\cdots}^0(\diag(l_{a,b},l_{a,b})\diag(v,v')\delta u_0(1,g),s)\psi_U(u_0)\,dv\,dv'\,du_0\,dg.
\end{align}
(Cf. \eqref{int:2}.)

Properties \eqref{it:GL observe 1}--\eqref{it:GL observe 6} from \S~\ref{proof of reduction lemma} now take the following form:
\begin{enumerate}[leftmargin=*]
\item\label{it:GL observe 1 in GL} ${}^{\delta_0^{-1}}\diag(v,v')=\diag(v',v)$.
\item\label{it:GL observe 2 in GL} If $v,v'\in V^3$, ${}^{\diag(v',v)}\delta_1=\delta_1u'$ where $u'\in U_0$ and $\psi_U(u')=1$.
\item\label{it:GL observe 3 in GL} The elements of both copies of $V^3$ normalize $U_0$ and fix $\psi_U|_{U_0}$.
\item\label{it:GL observe 4 in GL} The group $\diag(V^3,I_{kn})$ commutes with $(1,g)$.
\item\label{it:GL observe 5 in GL} $\delta_0$ commutes with $\diag(l_{a,b},l_{a,b})$.
\item\label{it:GL observe 6 in GL} $\diag(l_{a,b},I_{kn})$ commutes with $(1,g)$.
\end{enumerate}
Define
\begin{align}\label{eq:U_0' GL}
U_0'={}^{\diag(l_{a,b},l_{a,b})}U_0=\left\{\begin{pmatrix}I_{ka}&&U_1&U_2\\&I_{kb}&U_3&U_4\\&&I_{ka}\\&&&I_{kb}\end{pmatrix}\right\}.
\end{align}
Here $U_4$ is independent of $U_1$, so that
\begin{align*}
\left\{\begin{pmatrix}I_{ka}&U_1\\&I_{ka}\end{pmatrix}\right\},\qquad
\left\{\begin{pmatrix}I_{kb}&U_4\\&I_{kb}\end{pmatrix}\right\}
\end{align*}
are the unipotent subgroups corresponding to the $\GL_a\times \GL_k$ and $\GL_b\times \GL_k$ integrals, and
restriction of $\psi_U$ to the coordinates of $U_1$ and $U_4$ gives the character $\psi_U$ defined for these integrals. The subgroup $U^{1,4}$ is now $U_1\times U_4$, and here we avoid the notation $U^{1,4}$. Also $\psi_U$ is trivial on $U_2$ and $U_3$, and $U_2$ (resp., $U_3$) takes the form $\left(\begin{smallmatrix}*&*\\0&*\end{smallmatrix}\right)$ where $0\in\Mat_{a\times b}$ (resp., $0\in\Mat_{b\times a}$).
Cf. \eqref{eq:U_0' classical} and \eqref{eq:U_1 classical}.

Utilizing properties \eqref{it:GL observe 1 in GL}--\eqref{it:GL observe 6 in GL}, integral~\eqref{int:GL start 2} equals
\begin{align}\label{int:GL start 3}
&\int\limits_{G}\phi^{\vee,0}(g)\int\limits_{U_0'}
\int\limits_{V^3}\int\limits_{V^3} f_{\cdots}^0(\delta_0({}^{\diag(l_{a,b},l_{a,b})}\delta_1) u_0'\diag(I_{kn},l_{a,b}v)(1,g)\diag(l_{a,b}v',I_{kn}),s)\\&\psi_U(u_0')\,dv\,dv'\,du_0'\,dg.\notag
\end{align}
(Cf. \eqref{int:3}.)

Let
\begin{align*}
U_0^{\bullet}=\left\{\begin{pmatrix}I_{ka}&&U_1&U_2\\&I_{kb}&U_3&U_4\\&&I_{ka}\\&&&I_{kb}\end{pmatrix}\right\},
\end{align*}
which is similar to \eqref{eq:U_0' GL} except the block $U_2$, which contains arbitrary coordinates in place of the $0$ block (cf. \eqref{eq:U_0' classical2}).
Then for $u_0^{\bullet}\in U_0^{\bullet}$, $u^i$ denotes the element obtained from $u_0^{\bullet}$ by zeroing out the coordinates in the blocks $U_j$ with $j\ne i$.

As in the proof of Lemma~\ref{lemma:reduction from classical to GLn} (\eqref{int:3}--\eqref{int:3.5}), we proceed by factoring the integral through $U_R$, to produce an intertwining operator. Then \eqref{int:GL start 3} equals
\begin{align}\label{int:GL start 3.5}
&\int\limits_{U_R\backslash G}\int\limits_{U_R}\phi^{\vee,0}(zg)\int\limits_{U_0'}
\int\limits_{V^3}\int\limits_{V^3} f_{\cdots}^0(\delta_0({}^{\diag(l_{a,b},l_{a,b})}\delta_1) u_0'\diag(I_{kn},l_{a,b}v)(1,zg)\diag(l_{a,b}v',I_{kn}),s)\\&\psi_U(u_0')\,dv\,dv'\,du_0'\,dz\,dg.\notag
\end{align}
(Cf. \eqref{int:3.5}.)

Let $z\in U_R$. First we conjugate $\diag(I_{kn},{V^3})$ by $(1,z)$. We can write
\begin{align*}
\diag(I_{kn},l_{a,b}v)(1,z)&=\diag(I_{kn},l_{a,b})\diag(I_{kn},v)(1,z)\\
&=\diag(I_{kn},l_{a,b})(1,z)\diag(I_{kn},v_z)\diag(I_{kn},v)
\\&=\left({}^{\diag(I_{kn},l_{a,b})}(1,z)\right)\diag(I_{kn},l_{a,b})\diag(I_{kn},v_z)\diag(I_{kn},v)
\end{align*}
with $v_z\in V_{(a,kn-a)}$. Then conjugating $({}^{\diag(l_{a,b},l_{a,b})}\delta_1)u_0'$ by ${}^{\diag(I_{kn},l_{a,b})}(1,z)$ we obtain
${}^{\diag(l_{a,b},l_{a,b})}\delta_1$ multiplied by a general element of $U_0^{\bullet}$, and when we take ${}^{\diag(I_{kn},l_{a,b})}(1,z)$ to the left and conjugate by $\delta_0$, and use the invariance properties of $W_a(\tau)$ on the top left $a\times a$ block
($W_a(\tau)$ in the inducing data of $f_{\cdots}^0$), we see that this element vanishes without emitting a character.
Turning to $v_z$, the element ${}^{\diag(I_{kn},l_{a,b})}\diag(I_{kn},v_z)$ normalizes
$U_0^{\bullet}$ with a change of variables that emits a character. This character is cancelled after we take
${}^{\diag(I_{kn},l_{a,b})}\diag(I_{kn},v_z)$ to the left: conjugate it by $\delta_0({}^{\diag(l_{a,b},l_{a,b})}\delta_1)$ and again use the invariance properties of $W_a(\tau)$. Also
$\phi^{\vee,0}(zg)=\phi^{\vee,0}(g)$. Altogether
\eqref{int:GL start 3.5} becomes
\begin{align}\label{int:GL start 4}
&\int\limits_{U_R\backslash G}\phi^{\vee,0}(g)\int\limits_{U_0^{\bullet}}
\int\limits_{V^3}\int\limits_{V^3} f_{\cdots}^0(\delta_0({}^{\diag(l_{a,b},l_{a,b})}\delta_1)u_0^{\bullet}\diag(I_{kn},l_{a,b}v)(1,g)\diag(l_{a,b}v',I_{kn}),s)\\&\psi_U(u_0^{\bullet})\,dv\,dv'\,du_0^{\bullet}\,dg.\notag
\end{align}
Here the $dz$-integration was incorporated into $du_0^{\bullet}$ (cf. \eqref{int:4}).

Put $\delta_0=w\diag(\delta_{0,a},\delta_{0,b})w'$ with
\begin{align*}
&w=\left(\begin{array}{cccc}I_{ka}\\&&I_{kb}\\&I_{ka}\\&&&I_{kb}\end{array}\right),\qquad\delta_{0,a}=\begin{pmatrix}&I_{ka}\\I_{ka}\end{pmatrix},\qquad
\delta_{0,b}=\begin{pmatrix}&I_{kb}\\I_{kb}\end{pmatrix},\qquad w'=w^{-1}.
\end{align*}
Then
\begin{align*}
\delta_0u_0^{\bullet}=w\cdot{}^{(\diag(\delta_{0,a},\delta_{0,b})w')}u^2\cdot\diag(\delta_{0,a},\delta_{0,b})
\cdot {}^{w'}(u^1u^4)\cdot w'u^3.
\end{align*}
Also $\diag(\delta_{0,a},\delta_{0,b})\cdot {}^{w'}({}^{\diag(l_{a,b},l_{a,b})}\delta_1)=\diag(\delta_a',\delta_b')$ is the embedding in $M_{(2ka,2kb)}$ of
the elements $\delta$ corresponding to the $\GL_a\times\GL_k$ and $\GL_b\times\GL_k$ integrals.
Let $U^2$ be the subgroup of elements $^{(\diag(\delta_{0,a},\delta_{0,b})w')}u^2$, $U^{1}$ corresponding to ${}^{w'}(u^1)$,
$U^{4}$ corresponding to ${}^{w'}(u^4)$, and $U^3$ to $u^3$. For $h\in H$,
\begin{align*}
&\int\limits_{U_0^{\bullet}}f_{\cdots}^0(\delta_0({}^{\diag(l_{a,b},l_{a,b})}\delta_1)u_0^{\bullet}h,s)\psi_U(u_0^{\bullet})\,du_0^{\bullet}
\\&=\int\limits_{U^3}\int\limits_{U^{4}}\int\limits_{U^{1}}\int\limits_{U^2}f_{\cdots}^0(wu^2
\diag(\delta_a',\delta_b')u^1u^4w'u^3h,s)\psi_U(u)\,du^2\,du^1\,du^4\,du^{3}.
\end{align*}
The $du^2$-integration over $U^2=\diag(I_{ka},V_{(ka,kb)},I_{kb})$ defines an intertwining operator $M(s)$ from the space of \eqref{rep:induced f before M(s) GL} to the space of
\begin{align}\label{rep:induced f after M(s) GL}
\Ind_{P_{(ak,ak,bk,bk)}}^{H}(|\det|^{-b/2+\alpha s}W_{a}(\tau)\otimes |\det|^{-b/2-\alpha s}W_{a}(\tau')\otimes |\det|^{a/2+\alpha s}W_{b}(\tau)\otimes |\det|^{a/2-\alpha s}W_{b}(\tau')),
\end{align}
applied to $f_{\cdots}^0$. The result is a rational section in the space of \eqref{rep:induced f after M(s) GL}, such that for
all $\diag(x,y)\in M_{(2ka,2kb)}$,
\begin{align*}
&M(s)f_{\cdots}^0(\diag(x,y),s)\\&=d_{\tau,\tau',a,b}(s)|\det{x}|^{kb-b/2}|\det{y}|^{-ka+a/2}\rho_{W_{a}(\tau)\otimes W_{a}(\tau')}^0(x,\alpha s/(ka))\varrho_{W_{b}(\tau)\otimes W_{b}(\tau')}^0(y,\alpha s/(kb)).
\end{align*}
The powers of $|\det{x}|$ and $|\det{y}|$ will cancel out as we explain.

For $g=\diag(x,y)\in M_R$, conjugating $\diag(I_{kn},V^3)$ by $(1,g)$ multiplies the measure by $|\det y|^{(k-1)a}$; conjugating
$U^3$ by ${}^{\diag(I_{kn},l_{a,b})}(1,g)$ multiplies the measure by $|\det x|^{(1-k)b}$; and when we use the Iwasawa decomposition $G=RK_G$ and consider the modulus character emitted by $\phi^{\vee,0}$, the integrand is further multiplied by $\delta_R^{-1/2}(g)$; so that the total change of measure is $|\det{x}|^{-kb+b/2}|\det{y}|^{ka-a/2}$.

In addition,
\begin{align*}
{}^{w'}({}^{\diag(I_{kn},l_{a,b})}(1,g))=\diag(I_{ka},x,I_{(k-1)a},I_{kb},y,I_{(k-1)b})=(1,x)(1,y),
\end{align*}
where $(1,x)$ is the embedding of $\GL_a$ in the
$\GL_a\times\GL_k$ integral on the top left block of $M_{(2ka,2kb)}$, for the representations $\pi_a\times\tau$, and
$(1,y)$ is the embedding corresponding to the $\GL_b\times\GL_k$ integral on the bottom right block of
$M_{(2ka,2kb)}$, for $\pi_b\times\tau$. Regarding the unipotent integrations over the copies of $V^3$ and over $U^3$, we can see (using conjugations as in \S~\ref{proof of reduction lemma}) that the $du^3$-integral vanishes unless both copies of $V^3$ are in $K_H$, then the integral over $U^3$ also vanishes outside $U^3\cap K_H$, so that the integrals $du^3dvdv'$ evaluate to $1$. Integral~\eqref{int:GL start 4} is then equal to
\begin{align*}
d_{\tau,\tau',a,b}(s)Z(\alpha s/(ka),\omega_a^0,\rho_{W_{a}(\tau)\otimes W_{a}(\tau')}^0)Z(\alpha s/(kb),\omega_b^0,\varrho_{W_{b}(\tau)\otimes W_{b}(\tau')}^0).
\end{align*}
In conclusion,
\begin{align*}
Z(s,\omega^0,f_{W_n(\tau)\otimes W_n(\tau')}^0)=d_{\tau,\tau',a,b}(s)Z(\alpha s/(ka),\omega_a^0,\rho_{W_{a}(\tau)\otimes W_{a}(\tau')}^0)Z(\alpha s/(kb),\omega_b^0,\varrho_{W_{b}(\tau)\otimes W_{b}(\tau')}^0).
\end{align*}

Let us turn to $d_{\tau,\tau',a,b}(s)$. Put $H'=\GL_{kn}$.
Restricting the normalized unramified section in the space of \eqref{rep:induced f before M(s) GL} to the subgroup
$\diag(I_{ka},H',I_{kb})$ of $H$, it becomes an unramified element in the space of the unramified principal series representation of ${H'}$ induced
from
\begin{align*}
(\otimes_{1\leq i\leq k,1\leq j\leq b}\chi_i|~|^{\alpha s+(a-b)/2+j-1/2})\otimes
(\otimes_{1\leq i'\leq k,1\leq j'\leq a}\chi'_{i'}|~|^{-\alpha s-n/2+j'-1/2}).
\end{align*}
The adjoint action of $\GL_{kb}(\C)\times\GL_{ka}(\C)$ on the Lie algebra $\Mat_{kb\times ka}(\C)$ is given by $[A,B]\cdot T=ATB^{-1}$. The value of $d_{\tau,\tau',a,b}(s)$ now follows as in \S~\ref{proof of reduction lemma}.

\subsection{Proof of Proposition~\ref{proposition:unramified computation for GL(1)}}\label{final reduction n = 1 linear groups}
Here $\pi$ is an unramified quasi-character of $G=\GL_1=F^*$. Since the $(k,1)$ model $W_1(\tau)$ is simply the Whittaker model, we can in this section consider any irreducible generic unramified representations $\tau$ and $\tau'$ of $\GL_{k}$ (e.g., non-unitary), such that their central characters are inverses of one another. Since $\tau$ is irreducible, $W_1(\tau)$ is isomorphic to $\tau$ (and similarly for $\tau'$).
For $k=1$ Proposition~\ref{proposition:unramified computation for GL(1)} was proved in \cite[\S~6.1]{PSR} (using \cite{GJ}). Henceforth assume $k>1$.

The proof of the proposition and in particular the proof of Claim~\ref{claim:justification} below, is based on the ideas of Soudry \cite{Soudry,Soudry3,Soudry2} (in the context of Rankin--Selberg integrals for $\SO_{2n+1}\times\GL_k$, see also \cite{me3} for the application of these ideas to Rankin--Selberg integrals for $\SO_{2n}\times\GL_k$).

For the $\GL_1\times\GL_k$ integral, $H=\GL_{2k}$ and $P=P_{(k,k)}$. Then $U_P=V_{(k,k)}$.
The section $h\mapsto f_{W_1(\tau)\otimes W_1(\tau')}(h,s)$ is on
\begin{align}\label{ind s}
\mathrm{I}(W_1(\tau),W_1(\tau'),s)=\Ind_{P_{(k,k)}}^{\GL_{2k}}((W_1(\tau)\otimes W_1(\tau'))\delta_{P_{(k,k)}}^{s}).
\end{align}
We recall the definitions of $U$, $\psi_U$, the embedding $(g_1,g_2):\GL_1\times\GL_1\rightarrow\GL_{2k}$ and subgroup $U_0<U$
from \S~\ref{Local factors for GL} (for $n=1$).
Here $U=V_{(1^{k-1},2,1^{k-1})}$ and
\begin{align*}
\psi_U(u)=\psi(-\sum_{i=1}^{k-1}u_{i,i+1}+u_{k,k+2}-\sum_{i=1}^{k-2}u_{k+1+i,k+2+i}).
\end{align*}
Note that $U$ is obtained from $N_{\GL_{2k}}$ by removing the $(k,k+1)$-th coordinate, and $\psi_U$ is ``almost" a generic character
of $N_{\GL_{2k}}$. For brevity, throughout this section we write a general element of $V_{(k,k)}$ in the form
\begin{align*}
[\begin{smallmatrix}y&z\\u&x\end{smallmatrix}]=\begin{pmatrix}I_{k-1}&&y&z\\&1&u&x\\&&1\\&&&I_{k-1}\end{pmatrix}.
\end{align*}
Then $U_0$ is the subgroup of elements $\{[\begin{smallmatrix}y&z\\0&x\end{smallmatrix}]\}$ with arbitrary $x,y$ and $z$, and
$\psi_U([\begin{smallmatrix}y&z\\0&x\end{smallmatrix}])=\psi(x_1)$, where $x_1$ is the leftmost coordinate of (the row) $x$.
The measure $du_0$ on $U_0$ is the product measure on the coordinates of $x,y$ and $z$ separately, e.g., regarding $y$ as an element of $F^{k-1}$.
The product $\GL_1\times\GL_1$ is embedded in the diagonal torus of $\GL_{2k}$ by $(g_1,g_2)=\diag(g_1I_{k},g_2,g_1I_{k-1})$; it normalizes $U$ and stabilizes $\psi_U$. Also
\begin{align*}
\delta_0=\left(\begin{smallmatrix}&I_k\\I_k\end{smallmatrix}\right),\qquad
\delta_1=\diag(I_{k-1},\left(\begin{smallmatrix}1&1\\&1\end{smallmatrix}\right),I_{k-1})=[\begin{smallmatrix}0&0\\1&0\end{smallmatrix}].
\end{align*}

The $\GL_{1}\times \GL_{k}$ integral $Z(s,\omega,f_{W_1(\tau)\otimes W_1(\tau')})$ takes the form
\begin{align}\label{integral before Fourier transform}
\int\limits_{F^*}\int f_{W_1(\tau)\otimes W_1(\tau')}(\delta_0
[\begin{smallmatrix}y&z\\1&x\end{smallmatrix}]\diag(I_{k},a,I_{k-1}),s)
\psi(x_1)\omega(a)\,dx\,dy\,dz\,d^*a.
\end{align}
(We multiplied $\delta_1$ by $u_0\in U_0$.) Here and below, the domains of integration for variables
$[\begin{smallmatrix}y&z\\u&x\end{smallmatrix}]$ are omitted for brevity; they are products of $F$ according to the dimensions of the variables. Since $\pi^{-1}$ is a quasi-character, we can replace the matrix coefficient $\omega$ with
$\pi^{-1}$ in \eqref{integral before Fourier transform}, and denote $Z(s,f_{W_1(\tau)\otimes W_1(\tau')})=Z(s,\pi^{-1},f_{W_1(\tau)\otimes W_1(\tau')})$.
This integral is absolutely convergent in a right half-plane which is independent of the choice of section,
and in this domain it satisfies the following equivariance properties:
\begin{align}\label{eq:equivariance for GL1 GL1}
Z(s,(g_1,g_2)u\cdot f_{W_1(\tau)\otimes W_1(\tau')})=\psi_U^{-1}(u)\pi(g_2)\pi^{-1}(g_1)Z(s,f_{W_1(\tau)\otimes W_1(\tau')}),\qquad\forall g_1,g_2\in\GL_1, u\in U.
\end{align}
Therefore, in its domain of convergence it can be regarded as an element of
\begin{align}\label{eq:space lemma uniqueness GL}
\Hom_{\GL_1\times\GL_1}(J_{U,\psi_U^{-1}}(\mathrm{I}(W_1(\tau),W_1(\tau'),s)),\pi^{-1}\otimes\pi).
\end{align}
(This is \eqref{eq:homspace GLn GLk} for $n=1$.)
\begin{lemma}\label{lemma:uniqueness}
For all but a finite set of values of $q^{-s}$, the space \eqref{eq:space lemma uniqueness GL}
is at most one-dimensional.
\end{lemma}

The proof of the lemma appears at the end of this section. The statement is valid also for $k=1$ (see Remark~\ref{remark:uniqueness k=1}).
There is a choice of section such that $Z(s,f_{W_1(\tau)\otimes W_1(\tau')})$ is absolutely convergent for all $s$, and equals a nonzero constant (independent of $s$). To see this, take $f_{W_1(\tau)\otimes W_1(\tau')}$ such that $\delta_0\cdot f_{W_1(\tau)\otimes W_1(\tau')}$ is right-invariant by $\mathcal{N}$ and supported in
$P{}^{\delta_0}\delta_1\mathcal{N}$, where $\mathcal{N}$ is a small compact open neighborhood of the identity in $\GL_{2k}$. Together with Lemma~\ref{lemma:uniqueness}, Bernstein's continuation principle (in \cite{Banks}) implies that \eqref{integral before Fourier transform} admits meromorphic continuation to a rational function in $q^{-s}$.

We compute \eqref{integral before Fourier transform} by comparing it to another integral defined using the Whittaker model of \eqref{ind s}.
First consider the Jaquet integral realizing the Whittaker functional on \eqref{ind s}, defined by
\begin{align*}
f_{W_1(\tau)\otimes W_1(\tau')}\mapsto\int f_{W_1(\tau)\otimes W_1(\tau')}(\delta_0[\begin{smallmatrix}y&z\\u&x\end{smallmatrix}],s)\psi(u)\,dx\,dy\,dz\,du.
\end{align*}
This integral is absolutely convergent for $\Real(s)\gg0$ and admits analytic continuation to a function in $\C[q^{-s},q^s]$ (for a rational section of \eqref{ind s}, the continuation is in $\C(q^{-s})$). For any fixed $s$, the Whittaker model of \eqref{ind s} consists of the Whittaker functions
\begin{align*}
W_{f_{W_1(\tau)\otimes W_1(\tau')}}(h,s)=\int f_{W_1(\tau)\otimes W_1(\tau')}(\delta_0[\begin{smallmatrix}y&z\\u&x\end{smallmatrix}]h,s)\psi(u)\,dx\,dy\,dz\,du\qquad (h\in\GL_{2k}),
\end{align*}
where on the right hand side the integral is defined by analytic continuation.

For any representation $\vartheta$ of $\GL_{2k}$, let $\vartheta^*$ be the representation on the same space of $\vartheta$, defined by $\vartheta^*(h)=\vartheta(h^*)$, where $h^*=J_{2k}{}^tg^{-1}J_{2k}$.
Denote $\widetilde{W}_{f_{W_1(\tau)\otimes W_1(\tau')}}(h,s)=W_{f_{W_1(\tau)\otimes W_1(\tau')}}(J_{2k}{}^th^{-1},s)$. The Whittaker model of
$\mathrm{I}(W_1(\tau),W_1(\tau'),s)^*$ consists of the functions $\widetilde{W}_{f_{W_1(\tau)\otimes W_1(\tau')}}$ (see \cite[\S~2.1]{JPSS}), again defined by analytic continuation. Also denote
\begin{align*}
[t,v]=\diag(I_{k},\begin{pmatrix}1&&\\&I_{k-2}&\\-t&v&1\end{pmatrix}),\qquad
w'=\diag(I_{k},\left(\begin{smallmatrix}&I_{k-1}\\1\end{smallmatrix}\right)).
\end{align*}
Now consider the following integral
\begin{align}\label{int:after functional equation to compare}
&\Psi(\zeta,s,f_{W_1(\tau)\otimes W_1(\tau')})=\int\limits_{F^*}\int\limits_{F^{k-2}}\int\limits_{F} W_{f_{W_1(\tau)\otimes W_1(\tau')}}(\diag(I_{2k-1},a)[t,v]
w',s)\pi^{-1}(a)|a|^{\zeta+k-1}\,dt\,dv\,d^*a.
\end{align}
Our first step is to show that for any fixed $s$, this integral is absolutely convergent in a left half-plane $\Real(\zeta)\ll0$, and admits meromorphic continuation to (a function in) $\C(q^{-\zeta})$. Observe that the integrand vanishes unless $t$ and $v$ belong to compact subgroups, independent of $a$: if
\begin{align*}
\mathrm{e}(\epsilon)={}^{{w'}^{-1}}\left(\begin{smallmatrix}I_{k-1}\\&1&&\epsilon\\&&I_{k-1}&\\&&&1\end{smallmatrix}\right)
\end{align*}
and $\epsilon$ is sufficiently small, depending on $f_{W_1(\tau)\otimes W_1(\tau')}$,
\begin{align*}
W_{f_{W_1(\tau)\otimes W_1(\tau')}}(\diag(I_{2k-1},a)[t,v]w',s)&=
W_{f_{W_1(\tau)\otimes W_1(\tau')}}(\diag(I_{2k-1},a)[t,v]w'\mathrm{e}(\epsilon),s)
\\&=\psi^{-1}(\epsilon t)W_{f_{W_1(\tau)\otimes W_1(\tau')}}(\diag(I_{2k-1},a)[t,v]w',s).
\end{align*}
Hence if $t$ is large, the Whittaker function must vanish. Therefore \eqref{int:after functional equation to compare} becomes a finite sum of integrals
\begin{align}\label{int:after functional equation to compare 2}
&\int\limits_{F^*}\int\limits_{F^{k-2}}W^{(i)}_{f_{W_1(\tau)\otimes W_1(\tau')}}(\diag(I_{2k-1},a)[0,v]
w',s)\pi^{-1}(a)|a|^{\zeta+k-1}\,dv\,d^*a,
\end{align}
where $W^{(i)}_{f_{W_1(\tau)\otimes W_1(\tau')}}$ are Whittaker functions, right-translations of $W_{f_{W_1(\tau)\otimes W_1(\tau')}}$. Next let
\begin{align*}
\mathrm{e}(\epsilon_1,\ldots,\epsilon_{k-2})={}^{{w'}^{-1}}\diag(I_{k},\begin{pmatrix}1&&&&\epsilon_1\\&\ddots&&&\vdots\\&&1&0&\epsilon_{k-2}\\&&&1&0\\&&&&1\end{pmatrix}),
\end{align*}
and for a fixed $1\leq i\leq k-2$ choose $\epsilon_i$ sufficiently small and $\epsilon_j=0$ for $j\ne i$. Then as with $t$ above, using $\mathrm{e}(\epsilon_1,\ldots,\epsilon_{k-2})$ (instead of $\mathrm{e}(\epsilon)$) we handle the coordinates of $v$ from left to right, starting with $i=1$ up to $i=k-2$, and see that \eqref{int:after functional equation to compare 2} becomes a finite sum of integrals
\begin{align}
&\int\limits_{F^*}W^{(j)}_{f_{W_1(\tau)\otimes W_1(\tau')}^0}(\diag(I_{2k-1},a),s)\pi^{-1}(a)|a|^{\zeta+k-1}\,d^*a\notag
\\&=\int\limits_{F^*}W^{(j)}_{f_{W_1(\tau)\otimes W_1(\tau')}^0}(\diag(I_{2k-1},a^{-1}),s)\pi(a)|a|^{1-k-\zeta}\,d^*a\notag
\\&=\int\limits_{F^*}\widetilde{W}^{(j)}_{f_{W_1(\tau)\otimes W_1(\tau')}^0}(\diag(a,I_{2k-1})J_{2k},s)\pi(a)|a|^{1-k-\zeta}\,d^*a.
\label{eq:int with W tilda}
\end{align}
But each of these is a Rankin--Selberg integral for $\GL_{2k}\times\GL_1$ and its convergence for
$\Real(\zeta)\ll0$ and continuation to $\C(q^{-\zeta})$ was proved in \cite{JPSS} (in fact, already in \cite{GJ}).

In its domain of convergence and in general by meromorphic continuation, $\Psi(\zeta,s,f_{W_1(\tau)\otimes W_1(\tau')})$ satisfies equivariance properties similar to
\eqref{eq:equivariance for GL1 GL1} but with $\pi$ replaced by $|~|^{-\zeta}\pi$:
\begin{align*}
&\Psi(\zeta,s,(g_1,g_2)u\cdot f_{W_1(\tau)\otimes W_1(\tau')})=\psi_U^{-1}(u)|g_2|^{-\zeta}\pi(g_2)|g_1|^{\zeta}\pi^{-1}(g_1)\Psi(\zeta,s,f_{W_1(\tau)\otimes W_1(\tau')}).
\end{align*}
Then for $\zeta=0$, the meromorphic continuation of $\Psi(\zeta,s,f_{W_1(\tau)\otimes W_1(\tau')})$ belongs to
\eqref{eq:space lemma uniqueness GL} and by Lemma~\ref{lemma:uniqueness}, it is proportional to
the meromorphic continuation of $Z(s,f_{W_1(\tau)\otimes W_1(\tau')})$.

\begin{claim}\label{claim:justification}
Let $\Real(s)\ll0$. Then
\begin{align}\label{int:after functional equation}
&\gamma(k s+1/2,\pi^{-1}\times\tau,\psi)
Z(s,f_{W_1(\tau)\otimes W_1(\tau')})=\Psi(0,s,f_{W_1(\tau)\otimes W_1(\tau')}).
\end{align}
Here $\gamma(k s+1/2,\pi^{-1}\times\tau,\psi)$ is the Rankin--Selberg $\gamma$-factor of $\pi^{-1}\times\tau$ (\cite{JPSS}).
\end{claim}

The proof is given below.
Since $\gamma(k s+1/2,\pi^{-1}\times\tau,\psi)\in\C(q^{-s})$, we deduce that $\Psi(0,s,f_{W_1(\tau)\otimes W_1(\tau')})$ admits meromorphic continuation to $\C(q^{-s})$. Then \eqref{int:after functional equation} immediately holds as an identity in $\C(q^{-s})$.
Now we can compute $Z(s,f_{W_1(\tau)\otimes W_1(\tau')}^0)$ using $\Psi(0,s,f_{W_1(\tau)\otimes W_1(\tau')}^0)$. For the computation of the latter, we start with $\Psi(\zeta,s,f_{W_1(\tau)\otimes W_1(\tau')}^0)$ for $\Real(\zeta)\ll0$, then take $\zeta=0$.
Since $\tau$ and $\tau'$ are irreducible, we can also take
$s$ such that \eqref{ind s} is irreducible (e.g., $\Real(s)\ll0$).

Since $f_{W_1(\tau)\otimes W_1(\tau')}^0$ is unramified, it is invariant on the right with respect to $w'$, and
\begin{align*}
&\Psi(\zeta,s,f_{W_1(\tau)\otimes W_1(\tau')}^0)=\int\limits_{F^*}\int\limits_{F^{k-2}}\int\limits_{F} W_{f_{W_1(\tau)\otimes W_1(\tau')}^0}(\diag(I_{2k-1},a)[t,v],s)\pi^{-1}(a)|a|^{\zeta+k-1}\,dt\,dv\,d^*a.\notag
\end{align*}
The $dtdv$-integration can be computed by arguing as above: using conjugations by ${}^{w'}\mathrm{e}(\epsilon)$, now with $\epsilon\in\mathcal{O}^*$ we see that the integrand vanishes unless $t\in\mathcal{O}$, and then since
$f_{W_1(\tau)\otimes W_1(\tau')}^0$ is unramified, the integral $dt$ equals $1$. Similarly we show that the coordinates of $v$ belong in $\mathcal{O}$, and the integral over these coordinates equals $1$. Thus
\begin{align*}
\Psi(\zeta,s,f_{W_1(\tau)\otimes W_1(\tau')}^0)&=
\int\limits_{F^*}\widetilde{W}_{f_{W_1(\tau)\otimes W_1(\tau')}^0}(\diag(a,I_{2k-1}),s)\pi(a)|a|^{1-k-\zeta}\,d^*a
\end{align*}
(see \eqref{eq:int with W tilda}).

Next observe that since $f_{W_1(\tau)\otimes W_1(\tau')}^0$ is normalized and unramified and
\eqref{ind s} is irreducible (by our choice of $s$), by the Casselman--Shalika formula \cite{CS2} and \cite{BZ2,CS1},
\begin{align*}
\widetilde{W}_{f_{W_1(\tau)\otimes W_1(\tau')}^0}(I_{2k},s)=W_{f_{W_1(\tau)\otimes W_1(\tau')}^0}(I_{2k},s)=L(2k s+1,\tau\times{\tau'}^{\vee})^{-1}\ne0.
\end{align*}
Also
\begin{align*}
\mathrm{I}(W_1(\tau),W_1(\tau'),s)^*=\Ind_{P_{(k,k)}}^{\GL_{2k}}(|\det|^{k s}{\tau'}^{\vee}\otimes|\det|^{-k s}\tau^{\vee})
\end{align*}
is irreducible unramified and generic.
Then by \cite[Proposition~2.3]{JS1} (see also \cite[\S~6]{GJ}),
\begin{align*}
\Psi(\zeta,s,f_{W_1(\tau)\otimes W_1(\tau')}^0)&=\frac{L(-\zeta-ks+1/2,\pi\times\tau^{\vee})L(-\zeta+ks+1/2,\pi\times{\tau'}^{\vee})}{L(2k s+1,\tau\times{\tau'}^{\vee})}.
\end{align*}
Then we can take $\zeta=0$ in this equality.
Since also by \cite{GJ,JS1,JPSS},
\begin{align*}
\gamma(ks+1/2,\pi^{-1}\times\tau,\psi)^{-1}
=\frac{L(ks+1/2,\pi^{-1}\times\tau)}{L(-ks+1/2,\pi\times\tau^{\vee})},
\end{align*}
Claim~\ref{claim:justification} implies
\begin{align*}
Z(s,\omega^0,f_{W_1(\tau)\otimes W_1(\tau')}^0)&=
Z(s,f_{W_1(\tau)\otimes W_1(\tau')}^0)\\&=
\gamma(ks+1/2,\pi^{-1}\times\tau,\psi)^{-1}\Psi(0,s,f_{W_1(\tau)\otimes W_1(\tau')}^0)\\&=
\frac{L(k s+1/2,\pi^{-1}\times\tau)L(k s+1/2,\pi\times{\tau'}^{\vee})}{L(2k s+1,\tau\times{\tau'}^{\vee})}.
\end{align*}
This completes the proof of the proposition.

\begin{proof}[Proof of Claim~\ref{claim:justification}]
We start with the left hand side of \eqref{int:after functional equation}. We take $\Real(s)\gg0$, where it is absolutely convergent.
Put
\begin{align*}
\jmath(t)=\diag(I_{k},\left(\begin{smallmatrix}1&-t\\&1\end{smallmatrix}\right),I_{k-2}).
\end{align*}
For fixed $u,t\in F$,
\begin{align*}
&\int f_{W_1(\tau)\otimes W_1(\tau')}(\delta_0[\begin{smallmatrix}y&z\\u&x\end{smallmatrix}]\jmath(t),s)\psi(x_1)\,dx\,dy\,dz\\
&=\psi((u-1)t)\int f_{W_1(\tau)\otimes W_1(\tau')}(\delta_0[\begin{smallmatrix}y&z\\u&x\end{smallmatrix}],s)\psi(x_1)\,dx\,dy\,dz.
\end{align*}
Since $\int_F\psi((u-1)t)dt=0$ unless $u=1$, when we apply this to $Z(s,f_{W_1(\tau)\otimes W_1(\tau')})$ we obtain
\begin{align}\label{int:justification 1}
\int\limits_{F^*}\int f_{W_1(\tau)\otimes W_1(\tau')}(\delta_0[\begin{smallmatrix}y&z\\u&x\end{smallmatrix}]\jmath(t)\diag(I_{k},a,I_{k-1}),s)
\psi(x_1)\pi^{-1}(a)\,dx\,dy\,dz\,dt\,du\,d^*a.
\end{align}

For a Schwartz--Bruhat function $\phi$ on $F$, define
\begin{align*}
\phi f_{W_1(\tau)\otimes W_1(\tau')}(h,s)=
\int\limits_Ff_{W_1(\tau)\otimes W_1(\tau')}(h\jmath'(r),s)\phi(r)\,dr,\qquad
\jmath'(r)=[\begin{smallmatrix}0&0\\r&0\end{smallmatrix}].
\end{align*}
Also let $\widehat{\phi}$ be the Fourier transform of $\phi$, defined by $\widehat{\phi}(t)=\int_F\phi(r)\psi^{-1}(rt)dr$.

Formally, we can change the order of integration $dtdu\mapsto dudt$ and consider the integral
\begin{align}\label{int:justification 2}
&Z'(s,f_{W_1(\tau)\otimes W_1(\tau')})\\&=\notag
\int\limits_{F^*}\int f_{W_1(\tau)\otimes W_1(\tau')}(\delta_0[\begin{smallmatrix}y&z\\u&x\end{smallmatrix}]\jmath(t)\diag(I_{k},a,I_{k-1}),s)
\psi(x_1)\pi^{-1}(a)\,dx\,dy\,dz\,du\,dt\,d^*a.
\end{align}
The convergence of \eqref{int:justification 2} is in the sense that
\begin{align}\label{int:justification 3}
\int\limits_{F^*}\int\limits_F \left|\int f_{W_1(\tau)\otimes W_1(\tau')}(\delta_0[\begin{smallmatrix}y&z\\u&x\end{smallmatrix}]\jmath(t)\diag(I_{k},a,I_{k-1}),s)
\psi(x_1)\pi^{-1}(a)\,dx\,dy\,dz\,du\right|\,dt\,d^*a<\infty.
\end{align}
To see this, note that since $f_{W_1(\tau)\otimes W_1(\tau')}$ is locally constant on the right, one can always choose $\phi$ such that
$\phi f_{W_1(\tau)\otimes W_1(\tau')}=f_{W_1(\tau)\otimes W_1(\tau')}$. Hence the left hand side of \eqref{int:justification 3} becomes
\begin{align}\label{int:justification 4}
\int\limits_{F^*}\int\limits_F \left|\int f_{W_1(\tau)\otimes W_1(\tau')}(\delta_0[\begin{smallmatrix}y&z\\u&x\end{smallmatrix}]\jmath(t)\diag(I_{k},a,I_{k-1})\jmath'(r),s)\phi(r)
\psi(x_1)\pi^{-1}(a)\,dr\,dx\,dy\,dz\,du\right|\,dt\,d^*a.
\end{align}
We can change the order of integration: first integrate over $x, y, z$ and $u$, and then over $r$ because $\phi$ is compactly supported and
the integral over $[\begin{smallmatrix}y&z\\u&x\end{smallmatrix}]$ is absolutely convergent (because $\Real(s)\gg0$, see e.g., \cite[\S~4.4--\S~4.6]{Soudry} and \cite[\S~11.15, Lemma~1]{Soudry}).
Therefore we can conjugate $\jmath'(r)$ to the left and after changing variables in $x_1$ and $u$, obtain $\psi^{-1}(a^{-1}rt)$. We then integrate first over $r$ to obtain $\widehat{\phi}(a^{-1}t)$, and change variables $t\mapsto at$. Then
$\jmath(t)\mapsto \jmath(at)$ and $\jmath(at)\diag(I_{k},a,I_{k-1})=\diag(I_{k},a,I_{k-1})\jmath(t)$. Integral~\eqref{int:justification 4} equals
\begin{align*}
\int\limits_{F^*}\int\limits_F \left|\int f_{W_1(\tau)\otimes W_1(\tau')}(\delta_0[\begin{smallmatrix}y&z\\u&x\end{smallmatrix}]\diag(I_{k},a,I_{k-1})\jmath(t),s)\widehat{\phi}(t)
\psi(x_1)\pi^{-1}(a)|a|\,dx\,dy\,dz\,du\right|\,dt\,d^*a.
\end{align*}
The $dt$-integration produces a finite sum of integrals, and each is bounded in $\Real(s)\gg0$ (see \cite[\S~4.4--\S~4.6]{Soudry}). This proves \eqref{int:justification 3}.

Then for $\Real(s)\gg0$, integral~\eqref{int:justification 2} also belongs to \eqref{eq:space lemma uniqueness GL} hence by Lemma~\ref{lemma:uniqueness}, it is proportional to \eqref{int:justification 1}. The proportionality factor is $1$. Indeed, repeating the manipulations above used for the proof of \eqref{int:justification 3},
\begin{align*}
&Z'(s,\phi f_{W_1(\tau)\otimes W_1(\tau')})
\\&=
\int\limits_{F^*}\int\limits_F \int f_{W_1(\tau)\otimes W_1(\tau')}(\delta_0[\begin{smallmatrix}y&z\\u&x\end{smallmatrix}]\jmath(at)\diag(I_{k},a,I_{k-1}),s)\widehat{\phi}(t)
\psi(x_1)\pi^{-1}(a)|a|\,dx\,dy\,dz\,du\,dt\,d^*a.
\end{align*}
Changing $u\mapsto u+1$, conjugating $\jmath(at)$ to the left and changing variables in $x_1$: $x_1\mapsto x_1+(u+1)at$ and in $z$, and since
$f_{W_1(\tau)\otimes W_1(\tau')}(\delta_0\jmath(at)h,s)=\psi^{-1}(at)f_{W_1(\tau)\otimes W_1(\tau')}(\delta_0h,s)$,
we obtain
\begin{align*}
&\int\limits_{F^*}\int\limits_F \int f_{W_1(\tau)\otimes W_1(\tau')}(\delta_0[\begin{smallmatrix}y&z\\u+1&x\end{smallmatrix}]\diag(I_{k},a,I_{k-1}),s)\widehat{\phi}(t)\psi(aut)
\psi(x_1)\pi^{-1}(a)|a|\,dx\,dy\,dz\,du\,dt\,d^*a.
\end{align*}
Then integrating first over $t$ and since $\int_F\widehat{\phi}(t)\psi(aut)dt=\phi(au)$ by the Fourier inversion formula, the last integral equals
\begin{align*}
&\int\limits_{F^*}\int\limits_F \int f_{W_1(\tau)\otimes W_1(\tau')}(\delta_0[\begin{smallmatrix}y&z\\u+1&x\end{smallmatrix}]\diag(I_{k},a,I_{k-1}),s)\phi(au)
\psi(x_1)\pi^{-1}(a)|a|\,dx\,dy\,dz\,du\,d^*a.
\end{align*}
Noticing that $[\begin{smallmatrix}y&z\\u+1&x\end{smallmatrix}]=[\begin{smallmatrix}y&z\\1&x\end{smallmatrix}]\jmath'(u)$ and
$\jmath'(u)\diag(I_{k},a,I_{k-1})=\diag(I_{k},a,I_{k-1})\jmath'(au)$, and changing $u\mapsto a^{-1}u$, we arrive at
$Z(s,\phi f_{W_1(\tau)\otimes W_1(\tau')})$. Therefore in $\C(q^{-s})$,
\begin{align}\label{eq:Z and Z'}
Z(s,f_{W_1(\tau)\otimes W_1(\tau')})=Z'(s,f_{W_1(\tau)\otimes W_1(\tau')}).
\end{align}

Let $W\in W_1(\tau)$ and choose $f_{W_1(\tau)\otimes W_1(\tau')}$ such that $\delta_0\cdot f_{W_1(\tau)\otimes W_1(\tau')}$ is right-invariant by a small neighborhood of the identity $\mathcal{N}$ in $\GL_{2k}$, supported in $P\mathcal{N}$, and such that for all $a\in\GL_k$, $\delta_0\cdot f_{W_1(\tau)\otimes W_1(\tau')}(\diag(a,I_k),s)=|\det a|^{k(s+1/2)}W(a)$.
Now take a Schwartz--Bruhat function $\phi$ on $F$ such that for all $s$ and $h\in \GL_{2k}$,
\begin{align*}
\int\limits_Ff_{W_1(\tau)\otimes W_1(\tau')}(h\jmath(t),s)\widehat{\phi}(t)\,dt=f_{W_1(\tau)\otimes W_1(\tau')}(h,s).
\end{align*}
Our choice of data for the computation is now the section
$\phi f_{W_1(\tau)\otimes W_1(\tau')}$.
Plugging this section into \eqref{int:justification 2}, we see that $Z'(s,\phi f_{W_1(\tau)\otimes W_1(\tau')})$ equals
\begin{align*}
&\int\limits_{F^*}\int f_{W_1(\tau)\otimes W_1(\tau')}(\delta_0[\begin{smallmatrix}y&z\\u&x\end{smallmatrix}]\diag(I_{k},a,I_{k-1})\jmath(t)\jmath'(r),s)\phi(r)
\psi(x_1)\pi^{-1}(a)|a|\,dr\,dx\,dy\,dz\,du\,dt\,d^*a.
\end{align*}
We can change the order of integration: first integrate over $x, y, z$ and $u$, and then over $r$ because $\phi$ is compactly supported.
Therefore we can conjugate $\jmath'(r)$ to the left and after changing variables in $x_1$ and $u$, obtain $\psi^{-1}(rt)$. Then we integrate first over $r$ to obtain $\widehat{\phi}(t)$. Now integrate over $t$, and by our choice of $\widehat{\phi}$ obtain
\begin{align*}
&\int\limits_{F^*}\int f_{W_1(\tau)\otimes W_1(\tau')}(\delta_0[\begin{smallmatrix}y&z\\u&x\end{smallmatrix}]\diag(I_{k},a,I_{k-1}),s)
\psi(x_1)\pi^{-1}(a)|a|\,dx\,dy\,dz\,du\,d^*a.
\end{align*}
Conjugate $\diag(I_{k},a,I_{k-1})$ to the left. For our choice of $f_{W_1(\tau)\otimes W_1(\tau')}$ we see that the
integrand vanishes unless the coordinates of $[\begin{smallmatrix}y&z\\u&x\end{smallmatrix}]$ are small. Thus
\begin{align*}
Z'(s,\phi f_{W_1(\tau)\otimes W_1(\tau')})=\int\limits_{F^*}W(\diag(a,I_{k-1}))\pi^{-1}(a)|a|^{ks+1/2-(k-1)/2}\,d^*a,
\end{align*}
which is the Rankin--Selberg integral for $\GL_1\times\GL_k$ and $\pi^{-1}\times\tau_0$ (\cite[\S~2.4(3)]{JPSS} with $j=0$). This integral is absolutely convergent for $\Real(s)\gg0$, and admits meromorphic continuation to $\C(q^{-s})$. Together with \eqref{eq:Z and Z'} we deduce, in $\C(q^{-s})$ and in particular when $\Real(s)\ll0$,
\begin{align}\label{int:lhs of RS for n=1 and k>1}
Z(s,f_{W_1(\tau)\otimes W_1(\tau')})=\int\limits_{F^*}W(\diag(a,I_{k-1}))\pi^{-1}(a)|a|^{ks+1/2-(k-1)/2}\,d^*a.
\end{align}

For the right hand side of \eqref{int:after functional equation}, since
\begin{align*}
W_{\phi f_{W_1(\tau)\otimes W_1(\tau')}}(h,s)=\int\limits_FW_{f_{W_1(\tau)\otimes W_1(\tau')}}(h\jmath'(r),s)\phi(r)\,dr,
\end{align*}
a similar (but simpler) computation shows, for $\Real(\zeta)\ll0$,
\begin{align}\label{int:rhs of RS for n=1 and k>1 !}
\Psi(\zeta,s,\phi f_{W_1(\tau)\otimes W_1(\tau')})=
\int\limits_{F^{k-2}}\int\limits_{F^*}W(\begin{pmatrix}0&1&0\\0&0&I_{k-2}\\a&0&v\end{pmatrix})
\pi^{-1}(a)|a|^{\zeta+ks+1/2-(k-1)/2}\,d^*a\,dv.
\end{align}
This is again a Rankin--Selberg integral, now in the complex parameter $\zeta+ks$, which admits
meromorphic continuation to $\C(q^{-\zeta-ks})$ and is absolutely convergent when $\Real(s)\ll0$ for any $\zeta$. Hence we can take $\zeta=0$ on the right hand side of \eqref{int:rhs of RS for n=1 and k>1 !} and obtain
\begin{align*}
\int\limits_{F^{k-2}}\int\limits_{F^*}W(\begin{pmatrix}0&1&0\\0&0&I_{k-2}\\a&0&v\end{pmatrix})\pi^{-1}(a)|a|^{ks+1/2-(k-1)/2}\,d^*a\,dv.
\end{align*}
Therefore when we take $\zeta=0$ on the left hand side of \eqref{int:rhs of RS for n=1 and k>1 !}, when $\Real(s)\ll0$,
\begin{align}\label{int:rhs of RS for n=1 and k>1}
\Psi(0,s,\phi f_{W_1(\tau)\otimes W_1(\tau')})=
\int\limits_{F^{k-2}}\int\limits_{F^*}W(\begin{pmatrix}0&1&0\\0&0&I_{k-2}\\a&0&v\end{pmatrix})
\pi^{-1}(a)|a|^{ks+1/2-(k-1)/2}\,d^*a\,dv.
\end{align}
Finally \eqref{int:lhs of RS for n=1 and k>1} and \eqref{int:rhs of RS for n=1 and k>1} are related by $\gamma(s,\pi^{-1}\times\tau_0,\psi)\pi(-1)^{k-1}$ (see \cite[p.~70]{Soudry} for this version of \cite[Theorem~2.7]{JPSS}), and $\pi(-1)=1$.
\end{proof}
\begin{remark}
Alternatively, one can replace $\pi$ by $|~|^{-\zeta}\pi$ throughout this section. Lemma~\ref{lemma:uniqueness} is then valid outside a finite set of $q^{-s}$ and $q^{-\zeta}$,
Bernstein's continuation principle will imply that the integrals admit continuation to $\C(q^{-s},q^{-\zeta})$, and
Claim~\ref{claim:justification} provides an identity in $\C(q^{-s},q^{-\zeta})$.
\end{remark}
\begin{proof}[Proof of Lemma~\ref{lemma:uniqueness}]
The Jacquet module $J_{U,\psi_U^{-1}}(\mathrm{I}(W_1(\tau),W_1(\tau'),s))$ is a representation of the product $\GL_1\times\GL_1$, but since $\{(g_1,g_1):g_1\in F^*\}=C_{2k}$ (the center of $\GL_{2k}$), which acts trivially on $W_1(\tau)\otimes W_1(\tau')$ by our condition on $\tau$ and $\tau'$ (their central characters are inverses of one another), it is natural to restrict our attention to
one of the copies of $\GL_1$.

Identify $\GL_1$ with $\{(1,g_2):g_2\in F^*\}$, and in this manner regard $J_{U,\psi_U^{-1}}(\mathrm{I}(W_1(\tau),W_1(\tau'),s))$ as a representation of $\GL_1$.
It is enough to prove the statement for
\begin{align*}
\Hom_{\GL_1}(J_{U,\psi_U^{-1}}(\mathrm{I}(W_1(\tau),W_1(\tau'),s)),\pi).
\end{align*}
According to \cite[1.9 (b), (d)]{BZ2}, this space is isomorphic to
\begin{align}
&\Hom_{\GL_{2k}}(\mathrm{I}(W_1(\tau),W_1(\tau'),s),\Ind_{\GL_1U}^{\GL_{2k}}(\pi\otimes\psi_U^{-1}))\notag
\\&\label{bil}
\cong\Bil_{\GL_{2k}}(\ind_{\GL_1U}^{\GL_{2k}}{(\pi^{-1}\otimes\psi_U)},
\mathrm{I}(W_1(\tau),W_1(\tau'),s)).
\end{align}
Here $\Bil(\cdots)$ is the space of $\GL_{2k}$-equivariant bilinear forms and $\ind(\cdots)$ is the compact
induction.

For $h\in P\backslash \GL_{2k}/ \GL_1U$ (a finite set), put
\begin{align}\label{eq:hom h}
\Hom(h)=\Hom_{(\GL_1U)^h}({}^h(\pi^{-1}\otimes\psi_U)\otimes (W_1(\tau)\otimes W_1(\tau')\delta_{P}^{s}),\theta),
\end{align}
where $(\GL_1U)^h={}^h(\GL_1U)\cap P$; for a representation $\vartheta$ of
$\GL_1U$, ${}^h\vartheta$ is the
representation of $(\GL_1U)^h$ on the space of $\vartheta$ given by
${}^h\vartheta(x)=\vartheta({}^{h^{-1}}x)$; and
$\theta(x)=\delta_{\mathcal{C}(h)}(x,{}^{h^{-1}}x)\delta_{P}^{-1/2}(x)$, where
\begin{align*}
\mathcal{C}(h)=\{(x,{}^{h^{-1}}x):x\in (\GL_1U)^h\}<P\times \GL_1U
\end{align*}
 and
$\delta_{\mathcal{C}(h)}$ is the modulus character of $\mathcal{C}(h)$. To us, the only important properties of $\theta$ are that it is independent of $s$ and trivial on unipotent elements (being a modulus character). Also note that by definition, the space of the representation
on the left in $\Hom(h)$ is the space of $W_1(\tau)\otimes W_1(\tau')$.

According to the Bruhat theory (see e.g., \cite[Theorems~1.9.4 and 1.9.5]{Silb}, \cite[p.~48]{Soudry}), the space \eqref{bil} injects into the
semi-simplification
\begin{align*}
\bigoplus_{h\in P\backslash \GL_{2k}/ \GL_1U}\Hom(h).
\end{align*}
We may assume that a representative $h$ is either a permutation $w$ or $w\delta_1$, where
$\delta_1=[\begin{smallmatrix}0&0\\1&0\end{smallmatrix}]$, and recall
$\delta_0=\left(\begin{smallmatrix}&I_{k}\\I_{k}&\end{smallmatrix}\right)$. Put
$\kappa=\diag(I_{k-1},\left(\begin{smallmatrix}&1\\1\end{smallmatrix}\right),I_{k-1})$. Also write $h\sim h'$ if
$Ph\GL_1U=Ph'\GL_1U$.

First assume $w\not\sim\delta_0$ and $w\not\sim\delta_0\kappa$.
We claim that
\begin{align}\label{character nontrivial}
\psi_U|_{{}^{h^{-1}}V_{(k,k)}\cap U}\ne1.
\end{align}
Granted this, we can choose $u\in{{}^{h^{-1}}V_{(k,k)}\cap U}$ such that
$\psi_U(u)\ne1$. Then in \eqref{eq:hom h}, ${}^h(\pi^{-1}\otimes\psi_U)({}^hu)=\psi_U(u)\ne1$, and ${}^hu$ acts trivially on $W_1(\tau)\otimes W_1(\tau')$ (because ${}^hu\in V_{(k,k)}$). Thus the action of ${}^hu$ on the left in \eqref{eq:hom h} is nontrivial, while $\theta(u)=1$ on the right. This implies $\Hom(h)=0$.

We turn to prove \eqref{character nontrivial}. Note that ${}^hV_{(k,k)}={}^wV_{(k,k)}$.
Write $w=\left(\begin{smallmatrix}A_1&A_2\\A_3&A_4\end{smallmatrix}\right)$ with $A_i\in\Mat_k$.  Since in particular $w\ne\delta_0$, we can assume $A_1\ne0$. If the first column of $A_1$ is nonzero, let $i_0$ be the number of consecutive columns of $A_1$ starting from the first which are nonzero, so by definition $0<i_0\leq k$; if the first column of $A_1$ is zero, put $i_0=0$. Assume $i_0>0$. Then
we can write (perhaps after multiplying $w$ by a permutation from $M_P=M_{(k,k)}$)
\begin{align*}
w=\begin{pmatrix}I_{i_0}\\&0\\&I_{i_1}&\\&&\ddots\end{pmatrix},
\end{align*}
where the zero block above $I_{i_1}$ is the $(k-i_0)\times i_1$ zero matrix and $i_1\geq1$.

For any $y\in F$, let
$\jmath_{i,l}(y)\in N_{\GL_{2k}}$ be such that its $(i,l)$-th coordinate is $y$, and the remaining coordinates above the diagonal are zero. We need to show that for some $i,l$ such that $\jmath_{i,l}(y)\in V_{(k,k)}$, ${}^{h^{-1}}\jmath_{i,l}(y)$ belongs to $U$ and $\psi_U({}^{h^{-1}}\jmath_{i,l}(y))\ne1$.

This is clear if $i_0=k$ (then we can take $w=I_{2k}$), using $\jmath_{k,k+2}(y)$ (which commutes with $\delta_1$).
If $i_0=k-1$, either ${}^{h^{-1}}\jmath_{k-1,k+1}(y)=\jmath_{k-1,k}(y)$ if $h=w$, or when $h=w\delta_1$,
\begin{align*}
{}^{h^{-1}}\jmath_{k-1,k+1}(y)={}^{\delta_1^{-1}}\jmath_{k-1,k}(y)=\jmath_{k-1,k}(y)\jmath_{k-1,k+1}(y).
\end{align*}
In both cases we use $\jmath_{k-1,k+1}(y)$.
Also if $0<i_0<k-1$, ${}^{h^{-1}}\jmath_{i_0,k+1}(y)=\jmath_{i_0,i_0+1}(y)$. This verifies \eqref{character nontrivial} when $i_0>0$.

When $i_0=0$, we let $0<i_1<k$ be the number of consecutive columns of $A_1$, starting from the first, which are zero ($A_1\ne0$ whence $i_1<k$). Then we write
\begin{align*}
w=\begin{pmatrix}0&I_{i_2}\\\vdots&&\ddots\\0\\I_{i_1}\\&\ddots\end{pmatrix},
\end{align*}
where $I_{i_1}$ starts at the $(k+1,1)$-th coordinate, $i_2>0$ is the number of
consecutive nonzero columns of $\left(\begin{array}{cc}A_1&A_2\end{array}\right)$ starting from the $(i_1+1)$-th column, and $I_{i_2}$ begins at the $(1,i_1+1)$-th coordinate. Note that
$i_2\leq k$ and $i_1+i_2\leq 2k-1$ (because $i_1<k$).

Thus we can assume
\begin{align*}
w=\begin{pmatrix}0&I_{i_2}\\\vdots&0&\\0&\vdots\\I_{i_1}&0\\&0&I_{i_3}\\&&&\ddots\end{pmatrix},
\end{align*}
where $i_3\geq1$
and $I_{i_3}$ starts at the $(k+i_1+1,i_1+i_2+1)$-th coordinate.
By matrix multiplication,
\begin{align*}
{}^{w^{-1}}\jmath_{i_2,k+i_1+1}(y)=\jmath_{i_1+i_2,i_1+i_2+1}(y).
\end{align*}

If $i_1+i_2\geq k+2$ or $i_1+i_2\leq k-2$, then $\delta_1$ commutes with $\jmath_{i_1+i_2,i_1+i_2+1}(y)$ and $\psi_U$ is nontrivial on
$\jmath_{i_1+i_2,i_1+i_2+1}(y)$, hence we can take $\jmath_{i_2,k+i_1+1}(y)$. If $i_1+i_2=k+1$, then $i_2>1$ (since $i_1<k$) and ${}^{w^{-1}}\jmath_{i_2-1,k+i_1+1}(y)=\jmath_{k,k+2}(y)$, so that we can take $\jmath_{i_2-1,k+i_1+1}(y)$. When $i_1+i_2=k-1$,
${}^{w^{-1}}\jmath_{i_2,k+i_1+1}(y)=\jmath_{k-1,k}(y)$ and in both cases ($h=w$ or $h=w\delta_1$),
$\psi_U$ is nontrivial on ${}^{h^{-1}}\jmath_{i_2,k+i_1+1}(y)$ (as above, when $i_0=k-1$). If $i_1+i_2=k$ and $i_3\geq2$, then again ${}^{h^{-1}}\jmath_{i_2,k+i_1+2}(y)=\jmath_{k,k+2}(y)$.

The remaining case is $i_1+i_2=k$, in particular $i_2<k$, and $i_3=1$. In this case we further write
\begin{align*}
w=\begin{pmatrix}0&I_{i_2}&0\\\vdots&0&\vdots&I_{i_4}\\0&\vdots&&&\ddots\\I_{i_1}&0\\&0&1&0\\&&&0&I_{i_5}\\&&&&&\ddots\end{pmatrix},
\end{align*}
where $i_4>0$ (because $i_2<k$). If $i_5=0$, then $i_1=k-1$ whence $i_2=1$ and $i_4=k-1$, so that
\begin{align*}
w=\left(\begin{smallmatrix}&1\\&&&I_{k-1}\\I_{k-1}\\&&1\end{smallmatrix}\right)=\delta_0\kappa,
\end{align*}
contradicting our assumption ($w\not\sim\delta_0\kappa$). Therefore $i_5>0$. Then $\jmath_{i_2+i_4,k+i_1+2}(y)\in V_{(k,k)}$ because
$i_2+i_4\leq k$ (also $k+i_1+2\leq 2k$ since $k+i_1+1+i_5\leq 2k$),
${}^{w^{-1}}\jmath_{i_2+i_4,k+i_1+2}(y)=\jmath_{k+i_4+1,k+i_4+2}(y)$ which commutes with $\delta_1$ because $k+i_4+1\geq k+2$,
and $\psi_U(\jmath_{k+i_4+1,k+i_4+2}(y))\ne1$. In this case take $\jmath_{i_2+i_4,k+i_1+2}(y)$. This verifies \eqref{character nontrivial} when $i_0=0$, completing all cases.

There are now three possibilities remaining for $h$: $\delta_0$, $\delta_0\kappa$ or $\delta_0\delta_1$ (note that
$\delta_0\kappa\sim\delta_0\kappa\delta_1$), because we proved $\Hom(h)=0$ in all other cases.

Consider $h=\delta_0$. Then
\begin{align*}
(\GL_1U)^{\delta_0}=({}^{\delta_0}\GL_1)\ltimes(N_{\GL_{k}}\times N_{\GL_{k}}).
\end{align*}
Moreover, if we write $\diag(x,I_{2k-1})\diag(v,v')\in (\GL_1U)^{\delta_0}$ where $v,v'\in N_{\GL_k}$,
\begin{align}\label{eq:comp for delta0 uniqueness}
{}^{\delta_0}(\pi^{-1}\otimes\psi_U)(\diag(x,I_{2k-1})\diag(v,v'))=\pi^{-1}(x)\psi^{-1}(\sum_{i=1}^{k-1}v'_{i,i+1})
\psi^{-1}(\sum_{i=2}^{k-1}v_{i,i+1}).
\end{align}
In particular ${}^{\delta_0}\psi_U$ restricts to a degenerate character of the subgroup $\diag(N_{\GL_k},I_k)$ of $(\GL_1U)^{\delta_0}$.
Let $\mathcal{L}\in\Hom(\delta_0)$. For a pure tensor $\xi\otimes\xi'$ in the space of $W_1(\tau)\otimes W_1(\tau')$ and
$v\in V_{(1,k-1)}$, by \eqref{eq:comp for delta0 uniqueness},
\begin{align*}
\mathcal{L}(W_1(\tau)(v)\xi\otimes \xi')=
\mathcal{L}(W_1(\tau)\otimes W_1(\tau')(v,I_k)\xi\otimes\xi')=
\mathcal{L}(\xi\otimes\xi').
\end{align*}
Thus $\mathcal{L}$ factors through the Jacquet module $J_{V_{(1,k-1)}}(W_1(\tau))$ of $W_1(\tau)$ along $V_{(1,k-1)}$. Moreover for $x\in \GL_1$,
\begin{align*}
\mathcal{L}(\pi^{-1}(x)|x|^{ks}W_1(\tau)\otimes W_1(\tau')(\diag(x,I_{k-1}),I_k)\xi\otimes\xi')=
\theta(x)\mathcal{L}(\xi\otimes\xi').
\end{align*}
Hence
\begin{align}\label{eq:uniqueness for k=1 eq 1}
\mathcal{L}(W_1(\tau)(\diag(x,I_{k-1}))\xi\otimes \xi')= \pi(x)|x|^{-ks}\theta(x)\mathcal{L}(\xi\otimes\xi').
\end{align}
Since $\mathcal{L}$ must factor through one of the (finitely many) composition factors in a Jordan--H\"{o}lder series of
$J_{V_{(1,k-1)}}(W_1(\tau))$, $W_1(\tau)(\diag(x,I_{k-1}))\xi=\beta(x)\xi$ for some quasi-character $\beta$ of $F^*$,
which belongs to a finite set of characters and is independent of $s$. We deduce
\begin{align*}
\mathcal{L}(\xi\otimes \xi')= \beta^{-1}(x)\pi(x)|x|^{-ks}\theta(x)\mathcal{L}(\xi\otimes\xi').
\end{align*}
Now if $\mathcal{L}$ is nonzero, it is nonzero on some $\xi\otimes \xi'$, which may depend on $s$, but then
\begin{align*}
|x|^{ks}=\beta^{-1}(x)\pi(x)\theta(x),\qquad\forall x\in F^*.
\end{align*}
This equality can hold for at most finitely many values of $q^{-s}$. Therefore $\mathcal{L}=0$ and
$\Hom(\delta_0)$ vanishes outside finitely many values of $q^{-s}$.

Assume $h=\delta_0\kappa$. In this case
$(\GL_1U)^{\delta_0\kappa}={}^{\delta_0\kappa}\GL_1\ltimes(N_{\GL_{k}}\times N_{\GL_{k}})$ and \eqref{eq:comp for delta0 uniqueness} becomes
\begin{align*}
{}^{\delta_0\kappa}(\pi^{-1}\otimes\psi_U)(\diag(I_{2k-1},x)\diag(v,v'))=\pi^{-1}(x)\psi^{-1}(\sum_{i=1}^{k-2}v'_{i,i+1})
\psi^{-1}(-v_{1,2}+\sum_{i=2}^{k-1}v_{i,i+1}).
\end{align*}
Again ${}^{\delta_0\kappa}\psi_U$ restricts to a degenerate character, now of $\diag(I_{k},N_{\GL_k})$. We can now argue as above: $\mathcal{L}$ factors through $J_{V_{(k-1,1)}}(W_1(\tau'))$ and instead of \eqref{eq:uniqueness for k=1 eq 1} we have
\begin{align}\label{eq:uniqueness for k=1 eq 1 kappa}
\mathcal{L}(W_1(\tau')(\diag(I_{k-1},x))\xi\otimes \xi')= \pi(x)|x|^{ks}\theta(x)\mathcal{L}(\xi\otimes\xi').
\end{align}
Thus $\Hom(\delta_0\kappa)$ vanishes outside finitely many values of $q^{-s}$.

Finally let $h=\delta_0\delta_1$ ($h=\delta$ in the notation of \S~\ref{Local factors for GL}).
Then $(\GL_1U)^h=N_{\GL_{k}}\times N_{\GL_{k}}$
and $\psi_{U}$ restricts to the non-degenerate character $\psi^{-1}(z)=\psi^{-1}(\sum_{i=1}^{k-1}z_{i,i+1})$ on each
$N_{\GL_{k}}$. Thus for $\mathcal{L}\in \Hom(h)$, a pure tensor $\xi\otimes\xi'$ in the space of $W_1(\tau)\otimes W_1(\tau')$ and $v,v'\in N_{\GL_k}$,
\begin{align*}
\mathcal{L}(W_1(\tau)\otimes W_1(\tau')(v,v')\xi\otimes\xi')=
\psi(v)\psi(v')\mathcal{L}(\xi\otimes\xi'),
\end{align*}
so that $\mathcal{L}$ is in particular a Whittaker functional on $W_1(\tau)\otimes W_1(\tau')\cong\tau\otimes\tau'$,
and since $\tau$ and $\tau'$ are irreducible generic, the functional $\mathcal{L}$ is unique up to scaling.
\end{proof}
\begin{remark}\label{remark:uniqueness k=1}
For $k=1$ the proof of Lemma~\ref{lemma:uniqueness} is much simpler. First, the spaces $\Hom(h)$ are a priori at most one-dimensional, because $\tau$ and $\tau'$ are quasi-characters of $F^*$. It is therefore enough to show $\Hom(h)=0$ for $h\in\{\delta_0,I_2\}$ (now $\delta_0\kappa=I_2$), outside finitely many values of $q^{-s}$. For $h=\delta_0$ this follows immediately from \eqref{eq:uniqueness for k=1 eq 1}, because now
$W_1(\tau)(x)=\tau(x)$ (i.e., $\beta=\tau$), and for $h=I_2$ we use \eqref{eq:uniqueness for k=1 eq 1 kappa}.
\end{remark}

\def\cprime{$'$} \def\cprime{$'$} \def\cprime{$'$}

\end{document}